\documentclass[11pt]{amsart}
\usepackage{amsmath, amssymb, amscd, mathrsfs, slashed, enumerate, url}
\usepackage{color}
\usepackage{xcolor}
\usepackage[all,cmtip]{xy}
\usepackage{multicol}
\usepackage{indentfirst}
\usepackage{latexsym}
\usepackage{bm}
\usepackage{graphicx}
\usepackage{subfigure,esint}
\usepackage{float}
\usepackage{verbatim}
\usepackage{comment}
\usepackage[top=1in, bottom=1in, left=1.25in, right=1.25in]{geometry}
\usepackage{epsfig,dsfont,amsthm,amsfonts,amsbsy}
\usepackage[colorlinks]{hyperref}
\usepackage[toc,page]{appendix}
\usepackage{tikz-cd}

\newcommand{\lan}{\langle }
\newcommand{\ran}{\rangle}

\newcommand{\MS}{\mathcal{S}}
\newcommand{\MM}{\mathcal{M}}
\newcommand{\MO}{\mathcal{O}}

\newcommand{\Vol}{\mathrm{Vol}}

\newtheorem{theorem}{Theorem}[section]

\newtheorem{corollary}[theorem]{Corollary}

\newtheorem{definition}[theorem]{Definition}
\newtheorem{question}[theorem]{Question}
\newtheorem{example}[theorem]{Example}

\newtheorem{lemma}[theorem]{Lemma}

\newtheorem{proposition}[theorem]{Proposition}
\newtheorem*{remark}{Remark}

\newcommand{\MC}{\mathcal{C}}
\newcommand{\Tr}{\mathrm{Tr}}

\newcommand{\st}{\star}
\newcommand{\we}{\wedge}
\newcommand{\pa}{\partial}

\newcommand{\RP}{\mathbb R^+}% {(0,+\infty)}

\newcommand{\ti}{\times}

\newcommand{\vp}{\varphi}
\newcommand{\MA}{\mathcal{A}}

\newcommand{\al}{\alpha}

\newcommand{\na}{\nabla}
\newcommand{\ep}{\epsilon}

\newcommand{\be}{\beta}

\newcommand{\MT}{\mathcal{T}}

\newcommand{\lam}{\lambda}

\newcommand{\MF}{\mathcal{F}}

\newcommand{\Lam}{\Lambda}

\newcommand{\MB}{\mathcal{B}}
\newcommand{\MI}{\mathcal{I}}
\newcommand{\ZT}{\mathbb{Z}_2}

\newcommand{\ga}{\gamma}

\newcommand{\Id}{\mathrm{Id}}

\newcommand{\tS}{\tilde{S}}
\newcommand{\tm}{\tilde{m}}

\newcommand{\tr}{\tilde{r}}
\newcommand{\ttheta}{\tilde{\theta}}

\newcommand{\ca}{\mathcal{C}^{,\gamma}}
\newcommand{\SL}{\mathrm{SL}}
\newcommand{\Hess}{\mathrm{Hess}}

\newcommand{\tg}{\tilde{g}}
\newcommand{\tz}{\tilde{z}}
\newcommand{\floor}[1]{\left\lfloor #1 \right\rfloor}

\newcommand{\The}{\Theta}
\newcommand{\tSigma}{\tilde{\Sigma}}

\newcommand{\bna}{\bar{\na}}

\newcommand{\tal}{\tilde{\alpha}}

\newcommand{\MD}{\mathcal{D}}

\newcommand{\MBR}{\mathbb{R}}
\newcommand{\mbX}{\mathbf{X}}

\newcommand{\ze}{z}
\newcommand{\tL}{\tilde{L}}

\newcommand{\MU}{\mathcal{U}}

\newcommand{\tiota}{\tilde{\iota}}

\newcommand{\tx}{\tilde{x}}
\newcommand{\inj}{\mathrm{inj}}
\newcommand{\MMmec}{\MM_{\mathrm{mec}}}
\newcommand{\pSL}{\mathrm{pSL}}

\newcommand{\hna}{\hat{\na}}
\newcommand{\tSL}{\tilde{\SL}}
\newcommand{\Riem}{\mathrm{Riem}}
\newcommand{\hal}{\hat{\al}}
\newcommand{\LS}{\lesssim}
\newcommand{\hmu}{\hat{\mu}}
\newcommand{\hV}{\hat{V}}
\newcommand{\EQG}{E_{Q_{\delta g}}}
\newcommand{\EDQ}{E_{\delta Q}}
\newcommand{\omcan}{\omega_{\mathrm{can}}}

\newcommand{\Sh}{\mathrm{Sh}}

\newcommand{\mfA}{\mathfrak{A}}
\newcommand{\mfa}{\mathfrak{a}}
\newcommand{\mfb}{\mathfrak{b}}

\newcommand{\mfF}{\mathfrak{F}}
\newcommand{\mff}{\mathfrak{f}}
\newcommand{\app}{\mathrm{app}}

\newcommand{\mfh}{\mathfrak{h}}
\newcommand{\ta}{\tilde{a}}
\newcommand{\tU}{\tilde{U}}
\newcommand{\hg}{\hat{g}}

\newcommand{\tmfA}{\tilde{\mfA}}
\newcommand{\odd}{\mathrm{odd}}
\newcommand{\ind}{\mathrm{ind}}

\newcommand{\htau}{\hat{\tau}}

\newcommand{\omegac}{\omega_{\mathbb{C}}}

\usepackage[utf8]{inputenc}
\usepackage[english]{babel}
\usepackage{fancyhdr}
\usepackage{accents}

\begin{document}
	\title[The Branched Deformations of the Special Lagrangian Submanifolds]{The Branched Deformations of the Special Lagrangian Submanifolds}
	\author{Siqi He} 
	\date{\today}
	\address{Simons Center for Geometry and Physics, StonyBrook University\\Stonybrook, NY 11794}
	\email{she@scgp.stonybrook.edu}
	
	\begin{abstract}
		The branched deformations of immersed compact special Lagrangian submanifolds are studied in this paper. If there exists a nondegenerate $\ZT$ harmonic 1-form over a special Lagrangian submanifold $L$, we construct a family of immersed special Lagrangian submanifolds $\tL_t$, that are diffeomorphic to a branched covering of $L$ and $\tL_t$ convergence to $2L$ as current. This answers a question suggested by Donaldson \cite{donaldson2019deformations}. Combining with the work of \cite{abouzaidImagi21}, we obtain constraints for the existence of nondegenerate $\ZT$ harmonic 1-forms.
	\end{abstract}
	\maketitle

	\begin{section}{Introduction}
		Calabi-Yau n-folds $(X,J,\omega,\Omega)$ are compact complex manifold $(X,J)$ of complex dimension $n$, with a K\"ahler metric $g$, a compatible symplectic form $\omega$, and a parallel holomorphic volume form $\Omega$ with specific constant norm. Harvey and Lawson \cite{harveylawson1982calibratedgeometry} introduced the concepts of calibrated submanifolds, with special Lagrangian submanifolds being an important class.
		
		Let $L$ be a closed special Lagrangian submanifold, then the normal bundle $N_{L}$ could be identified with the cotangent bundle $T^{\st}L$ using the complex structure. The exponential map could identify small 1-forms $\al$ with submanifolds $L_{\al}$ close to $L$ and we could regard this as a deformation. According to McLean's deformation theorem \cite{mclean98defomation}, the $\MC^1$ deformations of the special Lagrangian submanifold $L$ could be parameterized by harmonic 1-forms on $L$ with the induced metric. For generalizations in different circumstances, see \cite{joycelecturesonSLgeometry05,marshall2020deformation,salur06deformation}. 
		
		In this work, we'll look at a generalization of McLean's theorem to branched covering deformations using the $\ZT$ harmonic 1-forms, which could be thought of as a multivalued 1-forms satisfy extra constrains. The study of $\mathbb{Z}_2$ harmonic 1-forms started from the work of Taubes \cite{taubes2013compactness} on characterized the non-compactness behavior of flat $\mathrm{PSL}(2,\mathbb{C})$ connections, see also \cite{haydyswalpuski2015compactness,he2020behavior,walpuskizhang2019compactness} for different generalizations. 
		
		Let $(L,g)$ be a compact Riemannian manifold and $\Sigma$ a codimension 2 oriented embedded submanifold of $L$. Let $\MI$ be a flat $\mathbb{R}$ bundle over $L\setminus \Sigma$, such that every small loop locally linking $\Sigma$ has monodromy $-1$. The monodromy of $\MI$ will define a branched covering $p:\tL\to L$, that branches along $\Sigma$. A section $\al\in T^{\st}L\otimes \MI$ is called a multivalued harmonic 1-form if $\al\in L^2$ satisfies the harmonic equations $d\al=d\st \al=0,$
		where $\star$ the metric Hodge star operator on $L$ and the derivative is taken using the flat structure on $\MI$. 
		
		According to the work of Donaldson \cite{donaldson2019deformations}, the multivalued harmonic 1-forms have an asymptotic description near $\Sigma$. Fixed $p\in \Sigma,$ we could find suitable complex coordinates $z$ on a slice through $p$ transverse to $\Sigma$ with the leading asymptotic 
		$$\al=\Re d(A z^{\frac12})+\Re d(B z^{\frac32})+o(|z|^{\frac32-\gamma}),$$
		where $A$ and $B$ are sections of suitable bundles, $d$ is taking the derivative using the flat structure and $0<\gamma<\frac12$. $\al$ is called a $\ZT$ harmonic 1-form, named in \cite{Taubes20133manifoldcompactness}, if $|\al|$ is bounded with $\|\al\|_{L^2}=1$, and nondegenerate if $A=0$ and $B$ is nowhere vanishing along $\Sigma$.  For a more detailed explanation, see Section \ref{sec_harmonicfunctionanalytictheory}.
		
		The deformation theory for nondegenerate $\ZT$ harmonic 1-forms have been studied by Donaldson \cite{donaldson2019deformations}, see also \cite{mazzeoandriyroyosuke2022,takahashi2015moduli} for different settings and generalizations. Donaldson observed that the nondegenerate multivalued harmonic 1-forms could be suitable candidates for the branched deformations of the special Lagrangian submanifolds. To be more explicit, Donaldson asked the following question.
		
		\begin{question}
			\label{question_donaldsonsquestion}
			\cite[Page 3]{donaldson2019deformations}
			Let $\iota_0:L\to X$ be the inclusion map of an immersed special Lagrangian submanifold, suppose over $\iota_0(L)$ there exists a nondegenerate $\ZT$ harmonic 1-form with branched covering $p:\tL\to L$. Could we construct a family of immersed special Lagrangian submanifolds $\tiota_t:\tL\to X$ such that $\tiota_t$ convergences to the double branched covering $\iota_0\circ p$?
		\end{question}
		
		We could give a positive answer to Donaldson's question, in particular, we proved:
		\begin{theorem}
			\label{thm_sec1main}
			Let $(X,J,\omega,\Omega)$ be a Calabi-Yau manifold and $\iota_0:L\to X$ be an immersed special Lagrangian submanifold with induced metric $g_L$. Suppose there exists a nondegenerate $\ZT$ harmonic 1-form $\al$ on $L$ with induced branched covering $p:\tL\to L$, then there exists a positive constant $T$ and a family of special Lagrangian submanifold $\tiota_t:\tL\to X$ for $|t|\leq T$ such that  
			\begin{itemize}
				\item [(i)] $\tiota_t(\tL)$ convergence to 2$\iota_0(L)$ as current,
				\item [(ii)] $\lim_{t\to 0}\|\tiota_t-\iota_0\circ p\|_{\ca}=0$,
				\item [(iii)] in a Weinstein neighborhood $U_L$, there exists a family of diffeomorphisms $\phi_t:L\to L$ such that suppose we write $\tiota_{t\phi_t^{\st}(\al)}:\tL\to U_L$ be the inclusion map induces by the graph of $t\phi_t^{\st}(\al)$, then there exists a t independent constant $C$ such that $$\|\tiota_{t\phi_t^{\st}(\al)}-\tiota_t\|_{\ca}\leq C t^2,$$
			\end{itemize}
			where $\MC^{,\gamma}$ is the H\"older norm taken with respect to the metric $p^{\st}g_L$ over $\tL$ and $0<\gamma<\frac12$.
		\end{theorem}
		
		When $L$ is Riemannian surface, nondegenerate harmonic 1-forms could be identified with quadratic differentials with simple zeros. Given a quadratic differential, the family of spectral curves will be an example of branched deformations for a suitable Calabi-Yau structure in a neighborhood of the zero section in $T^{\st}L$, which we refer Section \ref{sec_5} for more details of this construction.  However, when $\dim(L)\geq 3$, as the graph of a $\ZT$ harmonic 1-form might not be an embedded submanifold in a Weinstein neighborhood of $L$, we don't expect the branched deformation family we constructed above are embedded submanifolds.
		
		By \cite{he21z3symmetry}, there exist examples of real analytic rational homology 3-spheres that admits nondegenerate $\ZT$ harmonic 1-form, which could be realized as a special Lagrangian submanifolds in a neighborhood of the zero section in the cotangent bundle by the Calabi-Yau neighborhood theorem \cite{bryant2000calibratedembedding}. Then for these special Lagrangian submanifolds, as the first Betti number vanishes, the McLean's deformations will not exist. We obtain the following corollary. 
		
		\begin{corollary}
			There exist special Lagrangian submanifolds which are rigid in classical sense but could have branched deformations.
		\end{corollary}
		
		There are two major challenges in answering this question. The unbounded geometry of the branched deformation family presented the first challenge. Let $t$ be the parameter for the family, we may expect the injective radius, as well as the curvatures of the branched submanifolds, to grow to infinity when $t\to0$. This phenomenon also occurred in previous desingularization problems of calibrated submanifolds with singularities. For example, the isolated conical singularity desingularization \cite{joyce2004slag1} and the Lawlor neck problem, see \cite{butscher2004regularizing,imagi2016uniqueness,dan2004connectedsums,nordstromintersectingassociatives}. 
		
		The key observation for the branched deformation problem is that, despite the family's unbounded geometry, the first eigenvalues of the linearization operators over the family are bounded by a uniform positive number. In addition, we have a good understanding of the blow-up order of the geometry of the family, including the second fundamental form, the injective radius and the size of the Weinstein tubular neighborhood. 
		
		The second challenge comes from the singular set moving. The multivalued forms branched along submanifold could be considered as a free boundary problem, while the branch locus itself needs to be determined. For additional information on the linearized problem, we refer \cite{donaldson2019deformations,takahashi2015moduli}. The linearization of the special Lagrangian equation is the harmonic equation for multivalued 1-forms. Even we start with a multivalued harmonic 1-form which solves the linearization equations with branch locus $\Sigma$, it doesn't guarantee that we will be able to solve the non-linear equation using the same branch locus. In actuality, for each $t$, the branched deformation family may branch along a different codimension 2 submanifold $\Sigma_t$ which will be close to the initial branch locus of the $\ZT$ harmonic 1-form. 
		
		The above observation contributes to the construction of approximate special Lagrangian submanifolds, where the method we used is very similar to Donaldson in \cite{donaldson2019deformations,donaldsonslide2021}. Given an initial nondegenerate $\ZT$ harmonic 1-form $\al$, we proved that after moving the singular set in a specified direction, we could make good enough approximate special Lagrangian submanifolds $\tL_t^{\app}$, which is close to $t\phi_t^{\st}\al$, for some diffeomorphisms $\phi_t$ of $L$. Then we could apply a version of Joyce's nearby special Lagrangian theorem \cite{joyce2004slag3} to obtain the branched deformation family. 
		
		\begin{comment}
		The following theorem will be a special case of Joyce's nearby special Lagrangian theorem \cite{joyce2004slag3}. However, the input is more suitable for our purpose. 
		\begin{theorem}
		Let $\iota:L\to X$ be an immersed Lagrangian submanifold in Calabi-Yau $(X,J,\omega,\Omega)$, let $\MD$ be the linearized operator of the special Lagrangian equation on $L$ with first eigenvalue $\lam$, and $\tau$ be the maximum of the second fundamental form. We define $\Upsilon:=(1+\lam^{-1}\Vol(L)\tau^{2+\ga+\frac{n}{2}})^2\tau^{\ga}$, then for $0<\gamma<\frac12$, there exists a constant $C$ such that assume
		$\|\st\iota^{\st}\Im\Omega\|_{\ca}\leq \frac{1}{C\Upsilon}$, then there exists a function $f$ on $L$ such that
		\begin{itemize}
		\item [(i)] in a Weinstein neighborhood of $L$, the graph of $df$ will be a special Lagrangian submanifold.
		\item [(ii)] $f$ satisfies $\int_L f d\Vol_L=0$ and
		$$\|df\|_{\MC^{1,\ga}}\leq C \tau^{-\ga}(1+\lam^{-1}\Vol(L)\tau^{2+\ga+\frac{n}{2}})^{-1}.$$
		\end{itemize}
		\end{theorem}
		\end{comment}
		
		According to the celebrated work of Taubes \cite{Taubes20133manifoldcompactness,taubes2013compactness}, the bounded $\ZT$ harmonic 1-form plays an important role in $\mathrm{PSL}(2,\mathbb{C})$ gauge theory, which characterized the non-compactness behavior of the moduli space of $\mathrm{PSL}(2,\mathbb{C})$ flat connections. It would be fascinating to learn if there are any topological constraints on the existence of nondegenerate $\ZT$ harmonic 1-forms. Combing our main theorem with the work of nearby special Lagrangian theorem by Abouzaid and Imagi \cite{abouzaidImagi21}, we obtain a criterion for non-existence of nondegenerate $\ZT$ harmonic 1-form.
		
		\begin{theorem}
			\label{thm_obstructionforexistence}
			Let $(L,g_L)$ be a real analytic Riemannian manifold and $(\Sigma,\MI,\al)$ be a multivalued harmonic 1-form with associate branched covering $p:\tL\to L$, suppose the following holds:
			\begin{itemize}
				\item [(i)] $\pi_1(L)$ is finite or contains no nonabelian free subgroup,
				\item [(ii)] let $R$ be the canonical involution on $T^{\st}L$, there exists a Calabi-Yau neighborhood $(U_L,J,\omega,\Omega)$ of the zero section in $T^{\st}L$ with $R$ an anti-holomorphic involution,
				\item [(iii)] all $R$-invariant immersed special Lagrangian submanifolds that are diffeomorphic to $\tL$ is unobstructed,
			\end{itemize}
			then $(\Sigma,\MI,\al)$ is not a nondegenerate $\ZT$ harmonic 1-form.
		\end{theorem}
		
		According to the work of Bryant \cite{bryant2000calibratedembedding} and Doice \cite{doicu2015}, the condition (ii) holds when $\chi(L)=0$.  For the condition (iii), in general, it is very hard to check whether a Lagrangian is unobstructed or not. We do, however, have an additional $R$-symmetry given by the anti-holomorphic involution, which allows us to check the assumption (iii) for specific situations. When $\dim(L)=2$, according to \cite{solomon2020involutions}, the condition (i) and (iii) holds for $L=S^2$ and $T^2$. For a single special Lagrangian submanifold, we could introduce topological constraint to avoid obstruction, which give us a Betti number constraint for the branched covering.
		\begin{theorem}
			Let $(\Sigma,\MI,\al)$ be a nondegenerate $\ZT$ harmonic 1-form on $L$ with branched covering $p:\tL\to L$, suppose $\pi_1(L)$ is either finite or contains no nonabelian free subgroups and the Euler characteristic $\chi(L)=0$, then for the second Betti number, we have the strict inequality $b_2(\tL)>b_2(L)$.
		\end{theorem}
		
		The first Betti number inequality $b_1(\tL)>b_1(L)$ is already implied by the presence of nondegenerate $\ZT$ harmonic 1-forms, hence the preceding theorem is trivial when $\dim(L)=3$ by Poincare duality. However, when $\dim(L)\geq 4$, we found examples that the above theorem could be applied.
		
		This paper is organized as follows. In Section \ref{sec_sLaggeometry}, we introduce the background of the special Lagrangian geometry. In Section \ref{sec_sLagequationsovercotangentbundle}, we study the structure of the special Lagrangian equation in a Weinstein neighborhood. In Section \ref{sec_harmonicfunctionanalytictheory} and Appendix \ref{sec_backgroundpolyhomogeneoussolution}, introduce the analytic theory of $\ZT$ harmonic 1-form. In Section \ref{sec_5}, Section \ref{sec_6} and Section \ref{sec_geometryofapproximate}, we construct approximate solutions to the special Lagrangian equation and study their geometry of them. In Section \ref{sec_nearbyspecialLagrangianfamily}, we prove a nearby special Lagrangian theorem to prove the main theorem. In Section \ref{sec_applications}, we discuss the possible applications of our branched deformation theorem.
		
		\textbf{Conventions} We denote by $C>0$ a constant, which depends only on the back ground Calabi-Yau structure $(X,J,\omega,\Omega)$, the initial special Lagrangian submanifold or the initial $\ZT$ harmonic 1-form. All the $\MC^{,\gamma}$ we discuss will assume that $0<\gamma<\frac12$. The values of $C$ could change from one line to the next, and we always specify when a constant depends on further data. In addition, we write $x\LS y$ if $x\leq Cy$ and $O(y)$ denotes a quantity $x$ such that $|x|\LS Cy$. Moreover, without specific statement, all manifolds we considered in this paper will be compact.
		
		\textbf{Acknowledgements.} The author is grateful to Simon Donaldson for introducing this problem and providing countless helpful supports and discussions. The author also wish to thank Rafe Mazzeo, Kenji Fukaya, Thomas Walpuski, Jason Lotay, Robert Bryant, Mark Haskins, Fabian Lehmann, Kevin Sacke, Mohammed Abouzaid, Yohsuke Imagi and Foscolo Lorenzo for numerous helpful discussions and comments. The author want to also thank the Simons Center of Geometry and Physics for their support.
		
	\end{section}
	
	\begin{section}{Special Lagrangian Geometry}
		\label{sec_sLaggeometry}
		In this section, we will define the special Lagrangian submanifolds of a Calabi-Yau n-fold. In addition, we will introduce the McLean's deformation theorem. Some references for this section are Harvey and Lawson \cite{harveylawson1982calibratedgeometry}, Joyce \cite{joycelecturesonSLgeometry05} and McLean \cite{mclean98defomation}.
		\begin{subsection}{Calabi-Yau manifold and the special Lagrangian submanifolds}
			In this subsection, we will define the immersed special Lagrangian submanifolds in a Calabi-Yau manifold, follows from  Harvey and Lawson \cite{harveylawson1982calibratedgeometry}.
			\begin{definition}
				A Calabi-Yau manifold is a quadruple $(X,J,\omega,\Omega)$ such that $(X,J,\omega)$ is a K\"ahler manifold with a K\"ahler metric $g_X$ and K\"ahler class $\omega$.  Let $n$ be the complex dimension of $X$, $\Omega$ is a nowhere vanishing holomorphic $(n,0)$-form on $X$ satisfies
				\begin{equation}
					\label{eq_Calabiyau}
					\frac{\omega^n}{n!}=(-1)^{\frac{n(n-1)}{2}}{(\sqrt{-1})}^n\Omega\we\bar{\Omega}.
				\end{equation}
			\end{definition}
			
			An immersed submanifold is a pair $[L]=(L,\iota)$, where $L$ is a closed oriented manifold and $\iota:L\to X$ is an immersion. $[L]$ is called a Lagrangian submanifold if $\iota^{\st}\omega=0$. For any immersed Lagrangian submanifold $[L]$, for the holomorphic volume form, we have $|\iota^{\st}\Omega|=1$, where the norm is taken using the pull-back metric $\iota^{\st}g_X.$ Therefore, we could write 
			\begin{equation}
				\iota^{\st}\Omega=e^{i\theta}\Vol_{\iota^{\st}g_X},
				\label{eq_beginningLagrangianangle}
			\end{equation}
			where $e^{i\theta}:L\to S^1$ is called a phase function. 
			
			Harvey and Lawson \cite{harveylawson1982calibratedgeometry} introduced the concepts of special Lagrangian submanifolds.
			\begin{definition}
				Let $(X,J,\omega,\Omega)$ be a Calabi-Yau manifold, an immersed oriented submanifold $[L]=(L,\iota)$ is called a special Lagrangian submanifold if $e^{i\theta}=1$. 
			\end{definition}
			By \eqref{eq_beginningLagrangianangle}, $e^{i\theta}=1$ is equivalent to $\iota^{\st}\Im\Omega=0$ or $\iota^{\st}\Re\Omega=\Vol_{\iota^{\st}g_X}$. For a general Lagrangian submanifold, the $S^1$ valued function $e^{i\theta}$ doesn't necessarily have a lift to a real valued function $\theta$, while $d\theta$ is well-defined, which will satisfy
			$$
			d\theta=\iota_H \omega,
			$$
			where $H$ is the mean curvature vector field. 
			Suppose the lifting $\theta:L\to \mathbb{R}$ exists, then we call $\theta$ the Lagrangian angle. In particular, any Lagrangian that sufficiently close to a special Lagrangian will have a global well-defined Lagrangian angle, which are the cases that considered in this paper.
			
			Special Lagrangian submanifolds are an important class of calibrated submanifold. A Lagrangian submanifold $[L]=(L,\iota:L\to X)$ is called calibrated by $\Re\Omega$ if $\iota^{\st}\Re\Omega=d\Vol_{\iota^{\st}g_X}$, where $d\Vol_{\iota^{\st}g_X}$ is the volume form for the pullback metric $\iota^{\st}g_X$. The special Lagrangian submanifold is calibrated by $\Re\Omega$ based on the definition. The following is a fundamental property of calibrated submanifold.
			
			\begin{proposition}
				Let $[L]=(L,\iota)$ be a special Lagrangian submanifold equipped with the pull metric $g_L:=\iota^{\st}g$, then $\iota^{\st}\Im\Omega=d\Vol_L$. In addition, in the homology class $\iota_{\st}[L]$, $L$ is volume-minimizing.
			\end{proposition}
			\begin{proof}
				Let $[L']$ is another special Lagrangian submanifold such that $[L']$ lies in the same homology class of $[L]$. Then we compute $$\Vol(L)=\int_{L}d\Vol(L)=\int_{L}\iota^{\st}\Im\Omega=\int_{L'}\iota^{\st}\Im\Omega\leq \int d\Vol_{L'}=\Vol(L'),$$
				which implies the proposition.
			\end{proof}
			
			There are many interesting examples of special Lagrangian submanifolds have been constructed, which we refer to \cite{harveylawson1982calibratedgeometry,haskins2004cones,haskinskapouleas2007sLagcones,joyce2001constructingslag,joyce2002speciallagrangianincmwithsymmetry}. In addition, with extra assumptions of symmetry, we could also obtain examples of special Lagrangian submanifolds. We will list some of them which will be used later.
			\begin{example}{\cite{hitchin1997modulispace}}
				\label{exm_fixedpointofreallocals}
				Suppose over a Calabi-Yau manifold $(X,J,\omega,\Omega)$ with Calabi-Yau metric $g$, there exists an anti-holomorphic involution $R:X\to X$ such that $R^{\st}\omega=-\omega,\;R_{\st}J=-JR_{\st},\;R^{\st}g=g$ and $R^{\st}\Omega=\bar{\Omega}$, then the fix real locus $\{x\in X|R(x)=x\}$ is a special Lagrangian submanifold.
			\end{example}
			
			The following example could be found in \cite[Example 2.25]{marshall2020deformation}.
			\begin{example}
				Let $(X,g)$ be a hyperK\"ahler manifold of dimension $4k$, with complex strictures $I,J,K$ satisfies $I\circ J\circ K=\Id$ and symplectic structure $\omega_{I},\omega_J,\omega_K$. Take $\Omega:=(\omega_I+i\omega_J)^k$, then $(X,K,\omega_K,\Omega)$ and $(X,K,\omega_K,i\Omega)$ are both Calabi-Yau structures. $\iota:L\to X$ is called a complex Lagrangian submanifold if it is a complex submanifold for $J$ and symplectic for $\omega_K$. Then $\iota^{\st}\Im\Omega=0$ if $k$ is even and $\iota^{\st}\Im(i\Omega)=0$ if $k$ is odd. Therefore, every complex Lagrangian submanifold is a special Lagrangian submanifold for some Calabi-Yau structure.
			\end{example}
		\end{subsection}

		\begin{subsection}{Tubular neighborhood for immersed Lagrangian submanifolds}
			In this subsection, we will discuss the Weinstein tubular neighborhood for immersed Lagrangian submanifolds, which identified a neighborhood of an immersed submanifold with a neighborhood of the zero sections in the cotangent bundle.
			\begin{definition}
				Let $[L]=(L,\iota)$ be an immersed submanifold, we define the normal bundle $N_{L}:=\iota^{\st}TX/TL\cong TL^{\perp}\subset \iota^{\st}TY$, where $TL^{\perp}$ is the complement bundle defined by the pullback metric $\iota^{\st}g_X$ on $\iota^{\st}TX$. 
			\end{definition}
			
			\begin{proposition}
				\label{prop_identification_tangent_normal_bundle}
				Let $[L]=(L.\iota)$ be an immersed Lagrangian submanifold of the Calabi-Yau $(X,J,\omega,\Omega)$, then there exists a canonical isomorphism between the normal bundle $N_{\Sigma}$ and cotangent bundle $T^{\st}L$. 
			\end{proposition}
			\begin{proof}
				Let $g$ be the Riemannian metric on $(X,J,\omega,\Omega)$, using the pull-back metric $\iota^{\st}g$, we could identify $T^{\st}L$ with $TL$. In addition, we have $\omega(v_1,v_2)=g(v_1,Jv_2)$ for any vector fields $v_1,v_2$. The Lagrangian condition $\iota^{\st}\omega=0$ implies that $J$ maps between $TL$ and $N_{L}$, which is an isomorphism by dimension reason.
			\end{proof}
			
			Under the identification of $T^{\st}L$ and $N_{L}$, we can define the tubular neighborhood as follows.
			\begin{definition}
				Let $[L]=(L,\iota)$ be an immersed Lagrangian submanifold, a tubular neighborhood of $[L]$ is an open set $U_L\subset T^{\st}L$ which containing the zero section. We call $U_L(s)$ is a tubular neighborhood of $[L]$ with size $s$ if
				\begin{equation}
					U_L(s)=\{\al\in T^{\st}L|\|\al\|_{\MC^{0}}\leq s\}.
				\end{equation}
			\end{definition}
			
			We write $\vp:T^{\st}L\to N_{L}$ be the isomorphism in Proposition \ref{prop_identification_tangent_normal_bundle}, then we could define 
			\begin{equation}
				\label{eq_normalexpoenntialeqution}			
			\Phi:U_L\to X\;\mathrm{as}\;\Phi(\al)=\exp_{\iota\circ\pi(\al)}\circ\vp(\al),
			\end{equation} where $\exp$ is the exponential map on $X$, $\al\in T^{\st}L$ and $\pi:T^{\st}L\to L$ is the projection map of the bundle. Composing $\Phi$ with self diffeomorphism of $U_L$, we have the Weinstein tubular neighborhood theorem.
			
			\begin{theorem}{\cite{dasilva2001lecturesonsymplecticgeometry}}
				\label{thm_weinsteinnbhd}
				Let $(X,\omega)$ be a symplectic manifold and $[L]=(L,\iota)$ is a compact Lagrangian submanifold, then there exists an open tubular neighborhood $U_L$ with $\Phi:U_L\to X$ such that $\Phi^{\st}\omega=\omega_{0}$, where $\omega_0$ is the canonical symplectic structure on $T^{\st}L$.
			\end{theorem}
		
			Suppose $[L]$ is a Lagrangian submanifold in a Calabi-Yau manifold, then over the Weinstein neighborhood $U_L$, we could have a pull-back Calabi-Yau strcture $(U_L,J,\omega,\Omega)$ with $\omega$ the canonial symplectic form, which is the case that mostly considered in this article.
		\end{subsection}
		\begin{subsection}{McLean's deformation for the special Lagrangian submanifolds}
			The deformation theory of compact special Lagrangian submanifold is first studied by McLean \cite{mclean98defomation}, see also \cite{joycelecturesonSLgeometry05,marshall2020deformation,salur06deformation} in different settings. 
			
			Let $[L_0]$ be an immersed special Lagrangian submanifold, $U_L\subset T^{\st}L_0$ be a Weinstein neighborhood of $[L_0]$ and $\iota_0:L_0\to U_L$ be the inclusion map of the zero section. Then any Lagrangian immersion that is sufficiently $\MC^1$-close to $\iota_0$ will be given by the graph of a 1-form on $L$. Suppose we have a family of special Lagrangian submanifolds $[L_t]=(L,\iota_t)$ on $U_L$ such that $\iota_t$ are given by the graph of 1-forms on $L$ and $\lim_{t\to 0}\|\iota_t-\iota_0\|_{\MC^1}=0$, then we call $[L_t]$ a graphic deformations of $[L_0]$.
			
			\begin{theorem}{\cite{mclean98defomation}}
				\label{theorem_classicaldeformation}
				Let $(X,J,\omega,\Omega)$ be a Calabi-Yau manifold with metric $g_X$, $[L_0]=(L,\iota_0)$ be an immersed special Lagrangian submanifold on $X$, then the graphic deformations of $[L_0]$ are parameterized by harmonic 1-forms on $L$ for the metric $\iota_0^{\st}g_X$.
			\end{theorem}
			\begin{proof}
				We will give a sketch of proof of McLean's theorem by an implicit functional argument. Let $\Phi:U_L\to X$ be a Weinstein neighborhood of $[L_0]$, let $\al$ be a closed 1-form on $L$, for $t\in\mathbb{R}$ a real parameter, we define $[L_t]:=(L,\iota_t)$, with $\iota_t(x)=\Phi(t\al|_x)$, which is the graph of $t\al$.
				
				We construct a map
				\begin{equation}
					\begin{split}
						F_t&:\Omega^1(L)\to \Omega^2(L)\oplus \Omega^n(L),\\
						F_t&(v):=(\iota^{\st}_t\omega,\iota^{\st}_t\Im\Omega).
					\end{split}
				\end{equation}
				By \cite[Page 722]{mclean98defomation}, we obtain
				\begin{equation}
					\label{eq_linearzationcomputation}
					\frac{d}{dt}|_{t=0}\iota^{\st}_t\omega=d\al,\;\frac{d}{dt}|_{t=0}\iota^{\st}_t\Im\Omega=d\st \al,
				\end{equation} where $\st$ is the Hodge star operator on $[L_0]$ w.r.t. the pull-back metric $\iota^{\st}g_X$. In addition, by the special Lagrangian condition $\iota_0^{\st}\omega=0$ and $\iota_0^{\st}\Im\Omega=0$, $\iota_t^{\st}\omega$ and $\iota^{\st}_t\Omega$ are trivial in the cohomology group $H^2(L;\mathbb{R})$ and $H^n(L;\mathbb{R})$. Therefore, $\mathrm{Im} F_t \subset d\Omega^1(L)\oplus d\st \Omega^1(L)$. In a suitable norm $dF_t$ will be surjective and $F_t^{-1}(0,0)$ will be a smooth manifold model by the harmonic 1-forms.
			\end{proof}
			
			The McLean's deformation theorem works straight forward if we consider unbranched smooth deformation. Let $p:\tL\to L$ be any unbranched k-fold covering of $L$, then $[\tL_0]=(\tL,\iota\circ p)$ is also an immersed submanifold. Applying Theorem \ref{theorem_classicaldeformation} for $[\tL_0]$, we obtain the following corollary.
			\begin{corollary}
				\label{cor_unbranchedcoevering}
				Under the previous assumptions, given an harmonic 1-form $v$ on $\tL$ w.r.t. the pull-back smooth metric $(\iota\circ p)^{\st}g_X$, then there exists a family of special Lagrangian submanifolds $[\tL_t]=(\tL,\tiota_t)$ such that $\lim_{t\to 0}\tiota_t=\iota\circ p$.
			\end{corollary}
			
			\begin{remark}
					The unbranched smooth covering condition is an essential condition for Mclean's arguement works. Suppose $p:\tL'\to L$ is a branched covering of $L$, then $(\iota\circ p)^{\st}g$ is no longer a smooth metric along the branch locus and a harmonic 1-form for the singular metric doesn't guarantee a deformation. Suppose the Ricci curvature of $L$ is positive, then by \cite{tsai2018mean}, see also \cite{abouzaidImagi21}, that the branched deformation will not exist.
			\end{remark}
		\end{subsection}
	\end{section}
	\begin{section}{The Special Lagrangian Equation over the Cotangent Bundle}
		\label{sec_sLagequationsovercotangentbundle}
		In a Weinstein neighborhood, the special Lagrangian equation is an equation for 1-forms $\al$ on $T^{\st}L$ which characterized  whether the graph of $\al$ is a special Lagrangian submanifold. The structure of the special Lagrangian equation up to quadratic terms was studied in \cite{joyce2004slag2,mclean98defomation}. In this section, we will study the higher order expansion of the special Lagrangian equation using pseudo holomorphic volume form on $T^{\st}L$, which is an approximation of the real holomorphic volume form.
		
		\subsection{Special Lagrangian equations in $\mathbb{C}^n$ as a graph}
		We first consider the special Lagrangian equation in $\mathbb{C}^n$ for a graph manifold, introduced by Harvey and Lawson \cite{harveylawson1982calibratedgeometry}. Let $\mathbb{C}^n$ be the n-dimensional complex plane with complex coordinates $(z_1,\cdots,z_n)$. We write $z_i=x_i+\sqrt{-1}y_i$. Then $\mathbb{C}^n$ has a natural Calabi-Yau structure
		$$ g_0=\sum_{i=1}^n|dz_i|^2,\;\omega_0=\frac{\sqrt{-1}}{2}\sum_{i=1}^ndz_i\we d\bar{z}_i,\;\mathrm{and}\;\Omega_0=dz_1\we\cdots \we dz_n.
		$$
		Let $\mathbb{R}^n$ be the n-plane span by $(x_1,x_2,\cdots,x_n)$ and we identified $T^{\st}\mathbb{R}^n\cong \mathbb{C}^n$. Let $f$ be a smooth function on $\mathbb{R}^n$, then the graph manifold of $df$ will be
		\begin{equation*}
			\Gamma_f=\{(x_1+\sqrt{-1}\pa_{x_1}f,\;\cdots,\;x_n+\sqrt{-1}\pa_{x_n}f_n)\},
		\end{equation*}
		which is a smooth Lagrangian on $\mathbb{C}^n$. We write $\iota_f:\Gamma_f\to \mathbb{C}^n$ be the natural embedded map, then we have
		\begin{equation}
			\begin{split}
				\iota_f^{\st}\Im \Omega=\Im\det(I+i\Hess f)dx_1\we dx_2\we\cdots\we dx_n,\\
				\iota_f^{\st}\Re\Omega=\Re\det(I+i\Hess f)dx_1\we dx_2\we\cdots\we dx_n.
			\end{split}
		\end{equation}
		The special Lagrangian equation $\iota_f^{\st}\Im\Omega=0$ for $\Gamma_f$ can be written as 
		\begin{equation}
			\label{eq_euclideanspeciallagrangian}
			\Im\det(I+i\Hess f)=0.
		\end{equation}
		
		In particular, when $n=3$, the above equation obtain the following beautiful form
		\begin{equation*}
			-\Delta f+\det(\Hess f)=0.
		\end{equation*}

		\begin{subsection}{The geometry of the cotangent bundle}
			In this subsection, we will introduce the Sasaki metric on the cotangent bundle of a Riemannian manifold and the readers are referred to \cite{kentaroshigeru1973} for further details. Moreover, using the Sasaki metric, we will define an approximation to the holomorphic volume form.
			
			Let $(L,g)$ be a Riemannian manifold with Riemannian metric $g$ and cotangent bundle $T^{\st}L$. Let $p\in T^{\st}L$ with $\pi(p)=x$, using the Levi-Civita connection, we have the following splitting of the tangent bundle of $T^{\st}L$, which is $T_p(T^{\st}L)=H\oplus V,$
			where the horizontal part $H|_p\cong T_xL$ and the vertical part $V|_p\cong T^{\st}_xL$, and we write the isomorphism induced by the Levi-Civita connection as
			\begin{equation}
				\label{eq_identification}
				\begin{split}
					\Gamma:T_p(T^{\st}L)\to T_xL\oplus T_x^{\st}L.
				\end{split}
			\end{equation}
			Let $x=(x_1,\cdots,x_n)$ be a local coordinate system of $L$, and $y=(y_1,\cdots,y_n)$ by the fiber coordinate of $T_x^{\st}L$, such that $\sum_{i=1}^ny_idx_i$ is a section of $T^{\st}_{x}L$. Let $\nabla$ be the Levi-Civita connection with $\na_{\pa_{x_i}}\pa_{x_j}=\sum_{k=1}^n\Gamma_{ij}^k\pa_{x_k}$, then for $\pa_{x_i}$ and $d_{x_i}$, under the identification in \eqref{eq_identification}, we write the horizontal and vertical lifting in local coordinates $(x,y)$ as
			\begin{equation}
				\pa_{x_i}^H=\pa_{x_i}+\sum_{j,k=1}^ny_j\Gamma^{j}_{ik}\pa_{y_k},\; dx_i^V=\pa_{y_i}.
			\end{equation}
			
			\begin{proposition}
				\label{prop_identificaitonpushforwardmap}
				Let $\al$ be a 1-form on $L$, we write $\iota_{\al}:L\to T^{\st}L$ be the inclusion induced by the graph of $\al$, then for any vector field $v$ on $L$, we have $\Gamma\circ(\iota_{\al})_{\st}v=(v,\na_v\al)\in TL\oplus T^{\st}L$.
			\end{proposition}
			\begin{proof}
				Without losing generality, we take $v=\pa_{x_i}$ and write $\al=\al_i dx_i$. We compute 
				$$(\iota_{\al})_{\st}\pa_{x_i}=\pa_{x_i}+\sum_{j=1}^n\pa_{x_i}\al_j\pa_{y_j}=\pa_{x_i}^H+(\na_{\pa_{x_i}}\al)^V,$$
				thus $\Gamma((\iota_{\al})_{\st}\pa_{x_i})=(\pa_{x_i},\na_{\pa_{x_i}}\al).$
			\end{proof}
			
			Let $X,Y$ be vectors on $L$ with horizontal lifting $X^H,Y^H$ and $\al, \be$ be 1-forms on $L$ with vertical lifting $\al^V,\be^V$, then the Sasaki metric $g_0$ on $T^{\st}L$ is defined as
			$$
			g_0(X^H,Y^H)=g(X,Y)\circ \pi,\; g_0(\al^V,\be^V)=g^{-1}(\al,\be)\circ \pi,\; g_0(X^H,\al^V)=0.
			$$
			In addition, there exists a natural almost complex structure $J_0$ on $T^{\st}L$. Let $\vp_g:TL\to T^{\st}L$ be the induced isomorphism of the vector bundle, then we define the almost complex structure as
			$$J_0(X^H):=(\vp_g(X))^V,\; J_0(\vp_g(X)^V)=-X^H,$$
			where $X$ is a vector field on $L$. Let $\omega_0$ is the canonical symplectic structure on $T^{\st}L$, then by a straight forward computation, these three structures $g_0,\omega_0,J_0$ are compatible:
			$$\omega_0(X,Y)=g_0(X,J_0Y).$$
			
			Using the Sasaki metric, we could give an isomorphism of the tangent and cotangent bundle of $T^{\st}L$, which we write as $\Phi_{g_0}:T(T^{\st}L)\to T^{\st}(T^{\st}L)$. As the decomposition $T(T^{\st}L)\cong H\oplus V$ is orthogonal under the Sasaki metric and $T^{\st}(T^{\st}L)\cong H^{\st}\oplus V^{\st}$, $\Phi_{g_0}$ preserve the decomposition and gives the identification $\Phi_{g_0}|_H:H\to H^{\st}$ and $\Phi_{g_0}|_V:V\to V^{\st}$. 
			
			We define the dual basis to $\pa_{x_i}^H$, $dy_i$ and compute
			\begin{equation}
				\begin{split}
					dx_i^H:&=g_0(\pa_i^H,\;)=\sum_jg_{ij}dx_j,\;\\ dy_i^V:&=g_0(dy_i,\;)=\sum_{j}g^{ij}dy_j-\sum_{j,k}y_k\Gamma^k_{sj}g^{is}dx_j.
				\end{split}
			\end{equation}
		\end{subsection}
		\begin{subsection}{The pseudo holomorphic volume form}
			Now we will define a canonical n-form over $T^{\st}L$. 
			\subsubsection{The pseudo holomorphic volume form}
			We write $\xi:V\to T^{\st}L$ be the identification and $\xi_k:=\Lambda^k\xi:\Lam^kV\to \Lam^{k}T^{\st}L$. We could also identify $T^{\st}L$ and $H^{\st}$ by the following map
			\begin{equation}
				\eta:T^{\st}L\rightarrow TL\rightarrow H\rightarrow H^{\st},
			\end{equation}
			where the first map above is given by the Riemannian metric on $L$ and the third map above is given by the  Sasaki metric. A straight forward computation gives $\eta(dx_i)=\pi^{\st}dx_i$, with projection $\pi:T^{\st}L\to L$. To avoid notation, we also write $\pi^{\st}dx_i$ as $dx_i$. In addition, we define $\eta_k:=\Lam^k\eta:\Lam^kT^{\st}L\to \Lam^kH^{\st}$.
			
			Using $\xi_k$ and $\eta_{n-k}$, we defines an endormorphism $\kappa_k$ by
			\begin{equation}
				\kappa_k:\Lambda^kV\xrightarrow{\xi_k} \Lambda^{k}T^{\st}L\xrightarrow{\st}\Lambda^{n-k}T^{\st}L\xrightarrow{\eta_{n-k}} \Lambda^{n-k}H^{\st},
			\end{equation}
			where $\st$ is Hodge star operator induced by the Riemannian metric on $L$ and we write $K_k$ be the corresponding section of $\Lam^kV^{\st}\otimes \Lam^{n-k}H^{\st}$. 
			
			In local coordinates, $K_k$ could be written as
			\begin{equation}
				K_k:=\sum_{j_1,\cdots,j_n}\frac{\sqrt{\det(g)}}{n!}\ep_{j_1\cdots j_n}dy^V_{j_1}\we\cdots\we dy^V_{j_p}\we dx_{j_{p+1}}\we\cdots\we dx_{j_n},
			\end{equation}
			where $\ep_{j_1,\cdots,j_n}$ is the Levi-Civita symbol. Moreover, as $\Lambda^n(T^{\st}(T^{\st}L))=\oplus_{k=1}^n\Lambda^kV^{\st}\otimes \Lambda^{n-k}H^{\st},$ $K_k$ is a n-form on $T^{\st}L$.
			
			\begin{definition}
				The pseudo holomorphic volume form $\Omega_0\in \Lambda^n(T^{\st}(TL)\otimes \mathbb{C})$ is defined by
				\begin{equation}
					\label{eq_expressioncanonicalhomologyform}
					\Im\Omega_0:=\sum_{k\;\mathrm{odd}}(-1)^{\frac{k-1}{2}}K_k,\;\Re\Omega_0:=\sum_{k\;\mathrm{even}}(-1)^{\frac k2}K_k,\;\Omega_0:=\Re\Omega_0+i\Im\Omega_0.
				\end{equation}
			\end{definition}
			
			\begin{proposition}
				The pseudo holomorphic volume form $\Omega_0$ has the following properties:
				\begin{itemize}
					\item [(i)] Let $\iota_0:L\to T^{\st}L$ be the inclusion of the  zero section, then $\iota_0^{\st}\Im\Omega_0=0,\;\iota_0^{\st}\Re\Omega_0=\Vol_L.$
					\item [(ii)] Suppose at $m\in L$, $(x_1,\cdots, x_n)$ are normal coordinates w.r.t. the Riemannian metric on $L$, then at $q=(m,0)\in T^{\st}L$, $\Omega_0|_q=\Lambda_{i=1}^n(dx_i+\sqrt{-1}dy_i).$ 
					\item [(iii)] $\frac{\omega^n}{n!}=(-1)^{\frac{n(n-1)}{2}}\frac{i^n}{2^n}\Omega_0\we\bar{\Omega}_0$ on the zero section.
					\item [(iv)] $d\Omega_0|_{\iota_0(L)}=0$.
				\end{itemize}
			\end{proposition}
			\begin{proof}
				For (i), as $\iota_0^{\st}dy_i^V=0$ and every term of $\Im\Omega$ consists of $dy_i^V$, we obtain $\iota_0^{\st}\Im\Omega=0$. In addition, we have $\iota_0^{\st}\Im\Omega=\iota_0^{\st}K_0=\frac{\sqrt{\det(g)}}{n!}\ep_{k_1\cdots k_n}dx_{k_1}\we\cdots\we dx_{k_n}$, which is the volume form on $L$. As in the normal coordinates $(x_1,\cdots,x_n)$, we have $dy_i^V|_q=dy_i$, which implies (ii). (iii) follows directly from (ii). For (iv), in the normal coordinates above, we obtain $d(dy_i^V)|_q=dg^{ij}\we dy_j|_q- d(y_k\Gamma_{ij}^kg^{is})\we dx_j|_q=0$. Therefore, $d\Omega_0|_{\iota_0(L)}=d(\Vol_X)|_{\iota_0(L)}=0.$ 
			\end{proof}
			\subsubsection{The pseudo special Lagrangian equation}
			Given a closed 1-form $\al$, we write $[L_{\al}]=(L,\iota_{\al})$ be the graph of $\al$. Using the pseudo holomorphic volume form, we could define the pseudo special Lagrangian equation.
			\begin{definition}
				$[L_{\al}]$ is called a pseudo special Lagrangian submanifold if $[L_{\al}]$ is symplectic and $i_{\al}^{\st}(\Im\Omega)=0$.
				Moreover, we define the pseudo special Lagrangian equation for 1-form $\al$ as $\pSL(\al):=\st \iota_0^{\st}(\Im\Omega),$
				where $\st$ is the Hodge star operator w.r.t the induced Riemannian metric on the zero section
			\end{definition}
			
			We could explicitly compute pseudo special Lagrangian equation. Let $\al$ be a closed 1-form on $L$, then $\na\al$ is a symmetric $(0,2)$ tensor, we could use the Riemannian metric to define a $(1,1)$ tensor $A_{\al}:TL\to TL$. We define the $k$-th symmetric power of $A_{\al}$ as $P_k(\na\al):=\Tr(\Lambda^kA_{\al})$. If we choose a orthonormal base of $T^{\st}L$, we could regard $\na\al$ as a $n\ti n$ matrix, then we have $$\det(\Id+t\na\al )=\sum_{k=1}^nP_k(\na\al)t^k.$$  In addition, we define
			\begin{equation}
				\label{eq_nonlineartermP}
				P(\na\al):=\sum_{k=1}^{\lceil \frac{n-1}{2}\rceil}(-1)^kP_{2k+1}(\na\al).
			\end{equation}
			
			\begin{proposition}
				\label{prop_pseudoSLequations}
				The pseudo special Lagrangian equation of $[L_{\al}]$ on $T^{\st}L$ can be written as 
				\begin{equation}
					\begin{split}
						\pSL(\al)=-d^{\st} \al+P(\na \al),
					\end{split}
				\end{equation}
				where $P$ is a smooth function defined in \eqref{eq_nonlineartermP}. 
			\end{proposition}
			\begin{proof}
				We will compute in local coordinates $(x_1,\cdots,x_n,y_1,\cdots,y_n)$ under previous conventions. The linearization of $\pSL(\al)$ would be $\st(\iota_{\al}^{\st}K_1)$, where $$K_1=\sum_{k_1,\cdots,k_n=1}^n\frac{\sqrt{\det g}}{n!}\ep_{k_1\cdots k_n}dy_{k_1}^V\we dx_{k_2}\we\cdots\we dx_{k_n}.$$
				We write $\al=\sum_{k=1}^n\al_kdx_k$, then by a straight computation, we found
				$$\iota_{\al}^{\st}dy_i^V=\sum_{j,l=1}^ng^{il}(\pa_{x_j}\al_l-\sum_{k=1}^n\al_k\Gamma^{k}_{lj})dx_j=\sum_{j,l=1}^ng^{il}\na \al(\pa_{x_j},\pa_{x_l})dx_j.$$
				Therefore, $\st(\iota_{\al}^{\st}K_1)=-d^{\st}\al$.
				
				For the non-linear terms, we define $H^i_j:=\sum_{k=1}^ng^{ik}\na\al(\pa_{x_j},\pa_{x_k})$, then we compute $$\st (\iota_f^{\st}K_{2l+1})=\sum\ep_{k_1\cdots k_{2l+1}}\ep_{j_1\cdots j_{2l+1}}H^{k_1}_{j_1}H^{k_2}_{j_2}\cdots H^{k_{2l+1}}_{j_{2l+1}}=P_{2l+1}(\na\al),$$
				which is the desire term.
			\end{proof}
			When $L=\mathbb{R}^n$, $T^{\st}L\cong \mathbb{C}^n$. When we associating $T^{\st}L$ with the canonical Calabi-Yau structure on $\mathbb{C}^n$, the pseudo special Lagrangian equation will coincident with \eqref{eq_euclideanspeciallagrangian}.
			
			\subsubsection{The relationship between the holomorphic volume forms.}
			Let $[L]$ be a Lagrangian submanifold of a Calabi-Yau $X$, we could choose $U_L\subset T^{\st}L$ be the Weinstein neighborhood of $[L]$ with the pull-back Calabi-Yau structure $(U_L,J,\omega,\Omega)$ with Calabi-Yau metric $g$. We also write the zero section as $L_0$. Using the induced metric $\iota_0^{\st}g$ on the zero section, we could define a Sasaki metric $g_0$ on $U_L$ together with almost complex structure $J_0$ and the pseudo holomorphic volume form $\Omega_0$.
			
			\begin{proposition}
				\label{prop_constraintvolumefor}
				\begin{itemize}
					\item [(i)] Over the zero section $L_0\subset U_L$, we have $g_0|_{L_0}=g|_{L_0}$.
					\item [(ii)] Let $m\in L_0$, we write $T_m^{\perp}L$ be the orthogonal complement of $T_mL$ in $T_xU$ w.r.t $g$, then $T_m^{\perp}L$ is also a Lagrangian subspace of $\omega$.
					\item [(iii)] Let $\theta$ be the Lagrangian angle of $[L]$, then $\Omega|_{L_0}=e^{i\theta}\Omega_0|_{L_0}$.
				\end{itemize}
			\end{proposition}
			\begin{proof}
				Let $m\in L_0$, we choose normal coordinate $(x_1,\cdots,x_n)$ of an open set of $L$ centered at $m$ with fiber coordinates $(y_1,\cdots,y_n)$, we will check the statement in this coordinate system. As $\iota_0^{\st}g=\iota_0^{\st}g_0$, $(x_1,\cdots,x_n)$ is a normal coordinate for both metrics. For (i), based on the definition of the Sasaki metric, $\{\pa_{x_1},\cdots,\pa_{x_n},\pa_{y_1},\cdots, \pa_{y_n}\}$ is a orthonormal frame on $T|_m(T^{\st}L)$ for $g_0$ and  $\{\pa_{x_1},\cdots,\pa_{x_n},J\pa_{x_1},\cdots, J\pa_{x_n}\}$ is an orthonormal frame for $g$. As $L_0$ is a Lagrangian submanifold, we have $g(J\pa_{x_i},\pa_{x_j})=0$. Thus, $J\pa_{x_i}$ lies in the plane span by $\{\pa_{y_1},\cdots,\pa_{y_n}\}$. As $\omega$ is the canonical symplectic structure, we have $g(\pa_{x_i},J\pa_{y_j})=-\omega(\pa_{x_i},\pa_{y_j})=-\delta_{ij}$, which implies $J\pa_{y_j}=-\pa_{x_j}$. Therefore, $\{\pa_{x_1},\cdots,\pa_{x_n},\pa_{y_1},\cdots, \pa_{y_n}\}$ is also an orthonormal frame for $g$, which implies (i). As $\omega(\pa_{y_i},\pa_{y_j})=-g(J\pa_{y_i},\pa_{y_j})=0$, we obtain (ii). As $\{dx_1,\cdots,dx_n,dy_1,\cdots,dy_n\}$ is an orthonormal frame at $m$, by \eqref{eq_Calabiyau}, we have $$\Omega|_m=e^{\sqrt{-1}\theta}\wedge_{i=1}^n(dx_i+\sqrt{-1}dy_i)|_m=e^{\sqrt{-1}\theta}\Omega_0|_m,$$
				which implies (iii).
			\end{proof}
		\end{subsection}

		\begin{subsection}{The special Lagrangian equation in a neighborhood of a Lagrangian}
			Let $[L]$ be a Lagrangian submanifold in Calabi-Yau, $(U_L,J,\omega,\Omega)$ be the pull-back Calabi-Yau structure in the Weinstein neighborhood of $[L]$ with $g$ the Calabi-Yau metric. Let $\al$ be a 1-form on $U_L$ with $\iota_{\al}:L\to U_L$ be the graph manifold defined by $\al$ and we write $L_0$ be the inclusion of the zero section in $U_L\subset T^{\st}L$ with the inlusion $\iota_0$.
			
			The graph of $\al$ is a Lagrangian submanifold if and only if $\al$ is closed. In addition, the special Lagrangian equation for $\al$ is defined as
			\begin{equation}
				\begin{split}
					\SL(\al):=\st \iota_{\al}^{\st}\Im\Omega,
				\end{split}
			\end{equation}
			where $\st$ is the Hodge star operator with the pull-back metric $\iota_0^{\st}g$ on $L$. Taking the Hodge star operator for the metric $\iota^{\st}g$ instead of $\iota_{\al}^{\st}g$ is a more natural metric to consider, which makes the coefficients of the equations independent of $\al$.
			
			\begin{proposition}{\cite[Prop 2.21]{joyce2004slag2}}
				\label{prop_specialLagrangianlineartermstructure}
				The special Lagrangian equation could be written as
				\begin{equation}
					\begin{split}
						\SL(\al)=\sin\theta-d^{\st}(\cos \theta \al)+Q(\al,\na \al),
					\end{split}
				\end{equation}
				where $Q(\al,\na \al)$ is a smooth function with $|Q(\al,\na \al)|=\MO(|\al|^2+|\na \al|^2)$ for small $\al$
			\end{proposition}
			\begin{proof}
				Let $\theta$ be the Lagrangian degree defined in \eqref{eq_beginningLagrangianangle}, then we have $$\SL(0)=\st\iota_0^{\st}\Im\Omega=\st \Im(e^{i\theta}\Vol_{\iota_0^{\st}g})=\sin \theta.$$ By a similar computation as in \eqref{eq_linearzationcomputation}, we obtain $\frac{d}{dt}|_{t=0}\SL(t\al)=-d^{\st}(\cos\theta \al)$ and the proposition follows straightforward.
			\end{proof}
			
			Now, we will focus on the case when $[L]$ is a special Lagrangian submanifold, while the following observations are due to Joyce \cite[Page 15]{joyce2004slag3}. The value of $\SL(\al)$ at $x\in L$ depends on the $T^{\st}_p(\iota_{\al}(L))$, where $p=\iota_{\al}(x)$, and the cotangent bundle depends on $\al|_x$ and $\na \al|_x$. Therefore, $\SL(\al)|_x$ depends pointwisely on $\al|_x$ and $\na\al|_x$.
			
			Let $\Omega_0$ be the pseudo holomorphic volume form, we define a n-form over $U_L$ as 
			\begin{equation}
				T:=\Im\Omega-\Im\Omega_0,
			\end{equation}
			then by Proposition \ref{prop_constraintvolumefor}, $T|_{L_0}=0$. The error term $S$ of the quadratic terms of the special Lagrangian equation could be expressed as 
			$$S(\al,\na \al):=\st\iota_{\al}^{\st}T,$$
			as $S$ depends on the tangent space of the graph manifold of $\al$, $S$ is a smooth function of both $\al$ and $\na\al$.
			
			After fixing $x\in L$, define variable $y\in T_x^{\st}L$ and $z\in\otimes^2T_x^{\st}L$. Under the identification $\Gamma:T_{(x,y)}U_L\to T_xL\oplus T_x^{\st}L$, for any $z\in \otimes^2T_x^{\st}L$, we define the map $$I_{z}:T_xL\to T_xL\oplus T_x^{\st}L,\;I_z(y):=(y,\iota_y z).$$
			We define 
			\begin{equation}
				\begin{split}
					\label{eq_speiclalLagrangisntitlde}
					&\tSL:\{(x,y,z):x\in L,\;y\in T_x^{\st}L,\;z\in \otimes^2T_x^{\st}L\}\to \mathbb{R},\;\\
					&\tSL(x,y,z):=\st I_z^{\st}(\Gamma^{-1})^{\st}\Im \Omega|_{(x,y)}.
				\end{split}
			\end{equation}
			
			Similarly, we define 
			\begin{equation}
				\begin{split}
					&\tS:\{(x,y,z):x\in L,\;y\in T^{\st}_xL,\;z\in \otimes^2T^{\st}_xL\}\to \mathbb{R},\\
					&\tS(x,y,z):=\st I_z^{\st}(\Gamma^{-1})^{\st}T|_{(x,y)}.
				\end{split}
			\end{equation}
			
			After fixing $x\in L$, the variable $y,z$ lies in the vector spaces $T_x^{\st}L$ and $\otimes^2 T_x^{\st}L$. So we could take partial derivative in $y$ and $z$ direction without using a connection. Then we have $$\pa_y^{k_1}\pa_z^{k_2}\tSL|_x,\;\pa_y^{k_1}\pa_z^{k_2}\tS|_x\in S^{k_1}T_xL\otimes S^{k_2}(\otimes^2 T_xL).$$
			
			We have the following direct descriptions for $\tSL$ and $\tS$.
			\begin{proposition}
				\label{prop_3111111}
				$\tSL(x,\al,\na\al)=\SL(\al)|_x,\;\tS(x,\al,\na \al)=S(\al,\na \al)|_x$. Suppose $[L]$ is a special Lagrangian submanifold, then $\tSL,\tS$ are real analytic functions.
			\end{proposition}
			\begin{proof}
				By Proposition \ref{prop_identificaitonpushforwardmap}, we compute $I_{\na \al}(y)=(y,\na_y \al)=\Gamma\circ(\iota_{\al})_{\st}y.$ Therefore, we compute $$\tSL(x,\al,\na\al)=\st I_{\na \al}^{\st}(\Gamma^{-1})^{\st}\Im\Omega=\st (\Gamma^{-1}\circ I_{\na\al})^{\st}\Im\Omega=\st\iota_{\al}^{\st}\Im\Omega,$$
				while a similar computation holds for $\tS$. 
				
				As $\Omega$ is holomorphic, over $U_L$, $\tSL$ is real analytic. When $[L]=(L,\iota)$ is a special Lagrangian submanifold, by \cite{morrey1966multipleintegrals}, $\iota^{\st}g_X$ is a real analytic metric on $L$. Thus, the pseudo holomorphic volume form $\Omega_0$ and $T=\Omega-\Omega_0$ are real analytic forms. Therefore, $\tS$ is also a real analytic function.
			\end{proof}
			
			The following proposition explains the structure of the special Lagrangian equation.
			\begin{proposition}
				\label{prop_structureofspecialLagrangian}
				Let $[L]$ be a special Lagrangian submanifold, over the Weinstein neighborhood $(U_L,J,\omega,\Omega)$ with Calabi-Yau metric $g$, then the following holds:
				\begin{itemize}
					\item [(i)]the special Lagrangian equation for a closed 1-form $\al$ could be written as 
					\begin{equation}
						\begin{split}
							\SL(\al):=\st \iota_{\al}^{\st}\Im\Omega=-d^{\st}(\al)+P(\na \al)+S(\al,\na \al),
						\end{split}
					\end{equation}
					where $\st$ is the Hodge star operator on the zero section, $P$ is a real analytic function of $\na\al$ defined in \eqref{eq_nonlineartermP} and $S$ is a real analytic function of $\al,\na \al$. 
					\item[(ii)]	Let $x=(x_1,\cdots,x_n)$ be a coordinate on a neighborhood center at $m\in L_0$ with fiber coordinates $y=(y_1,\cdots,y_n)$. We define a function $A_{ij}=\lan \na_{\pa_{x_i}}\al, \pa_{x_j}\ran$ with $\na$ the Levi-Civita connection of $g$, then we could write
					
					\begin{equation}
						\label{eq_expressionsofS}
						S(\al,\na \al)=\sum_{k\geq 0}\sum_{\substack{1\leq i_1\neq i_2 \cdots \neq i_k\leq n \\ 1\leq j_1\neq j_2 \cdots \neq j_k\leq n}}F_{i_1\cdots i_kj_1\cdots j_k}(\al)A_{i_1j_1}\cdots A_{i_kj_k},
					\end{equation}
					where $F_{i_1\cdots i_kj_1\cdots j_k}(\al)$ are real analytic functions of $\al$. In addition, $F_0(0)=F_0'(0)=0$ and $F_{i_1j_1}(0)=0$ for any $1\leq i_1,j_1\leq n$.
				\end{itemize}
			\end{proposition}
			\begin{proof}
				For (i), as $\iota_{\al}^{\st}\Im\Omega=\iota_{\al}^{\st}\Im\Omega_0+\iota_{\al}^{\st}\Im T$, and $\frac{d}{dt}|_{t=0}\iota_{t\al}^{\st}\Im\Omega=\frac{d}{dt}|_{t=0}\iota_{t\al}^{\st}\Im\Omega_0$, we could write $$Q(\al,\na\al)=P(\na \al)+S(\al,\na \al).$$ By Proposition \ref{prop_3111111}, $P,S$ are real analytic functions of $\al$ and $\na\al$.
				
				For (ii), as $\al$ is closed, we have $A_{ij}=A_{ji}$. Over a neighborhood with coordinate $(x,y)$, we could write 
				$$
				T=\sum T_{i_1\cdots i_kj_1\cdots j_{n-k}}(x,y)dy_{i_1}\we \cdots \we dy_{i_k}\we dx_{j_1}\we\cdots \we dx_{j_{n-k}},
				$$
				where $T_{i_1\cdots i_kj_1\cdots j_{n-k}}$ is real analytic function with variable $(x,y)$. Then for the pull-back, we compute
				\begin{equation*}
					\begin{split}
						&\iota_{\al}^{\st}dy_i=\sum_{j=1}^n\pa_{x_j}\al_i dx_j=\sum_{j=1}^n (\pa_{x_j}\al_i) dx_j=\sum_{j=1}^n(A_{ij}+\sum_{k=1}^n \al_k\Gamma_{ij}^k)dx_j,\\
						&\iota_{\al}^{\st}dx_i=dx_i,\;(\iota_{\al}^{\st}T_{i_1\cdots i_kj_1\cdots j_{n-k}})|_x=T_{i_1\cdots i_kj_1\cdots j_{n-k}}(x,\al|_x).
					\end{split}
				\end{equation*}
				In addition, as $S(0,0)=0$ and $S(\al,\na\al)=\st\iota_{\al}^{\st}T$, $S$ could be written as the form 
				$$
				S(\al,\na \al)=\sum_{k\geq 0}\sum_{\substack{1\leq i_1\neq i_2 \cdots \neq i_k\leq n \\ 1\leq j_1\neq j_2 \cdots \neq j_k\leq n}}F_{i_1\cdots i_kj_1\cdots j_k}(\al)A_{i_1j_1}\cdots A_{i_kj_k}.
				$$
				In addition, as $\frac{d}{dt}|_{t=0}S(t\al,t\na\al)=0$. For the $k=0$ term $F_0$ in the above expression, we have $F_0(0)=F_0'(0)=0$ and for any $k=1$ terms, we have $F_{i_1j_1}(0)=0$ for any $1\leq i_1,j_1\leq n$.
				
			\end{proof}
		\end{subsection}
	\end{section}
	
	\begin{section}{Multivalued harmonic function and the analytic theory}
		\label{sec_harmonicfunctionanalytictheory}
		In this section, we will introduce the analytic theory for the multivalued harmonic function developed in \cite{donaldson2019deformations}. The analytic aspects of the multivalued harmonic functions have been widely studied, we also refer \cite{taubes2014zero} for different settings.
		
		\begin{subsection}{Multivalued harmonic function and multivalued harmonic 1-form}
			
			Let $(L,g)$ be a compact oriented Riemannian manifold, $\Sigma\subset L$ be a smooth embedded oriented codimensional two submanifold. Let $\Xi$ be the group of isomorphism of the real line, which satisfies the exact sequence $$0\to (\MBR,+)\to \Xi\to \{\pm 1\}\to 0.$$
			
			Let $\chi:\pi_1(L\setminus \Sigma)\to \Xi$ be a representation such that for any small loop linking $\Sigma$, $\chi$ maps to a reflection. We write $V^-$ be the associate affine $\mathbb{R}$ bundle defined on $L\setminus \Sigma$ and $\MI$ be the vertical line bundle of $V^-$. Therefore, a section $V^-$ could locally be regarded as a two-valued function.
			
			The composition of $\chi$ with the quotient map $\Xi\to \{\pm 1\}$ defines a representation $\chi':\pi_1(L\setminus\Sigma)\to \{\pm\}$ with associate flat vector bundle $\MI$ which is vertical tangent bundle of $V^-$. Using the flat structure on $V^-$ and the Riemannian metric on $L$, the Laplacian operator make sense as $$\Delta_g:\Gamma(V^-)\to \Gamma(\MI).$$ 
			
			\begin{definition}
				\label{def_multivaluedeifnition}
				A multivalued harmonic function consists of $(\Sigma,\chi,f)$, where
				\begin{itemize}
					\item [(i)] $\Sigma$ is an oriented smooth embedded codimension 2 submanifold of $L$, 
					\item [(ii)] $\chi$ is a representation $\chi:\pi_1(L\setminus \Sigma)\to \Xi$ such that for any small loop linking $\Sigma$, $\chi$ maps to a reflection,
					\item [(iii)] if we write $V^-$ be the associative affine bundle of $\chi$ with induced flat structure, then $f$ is a section of $V^-$ such that $\Delta_g f=0$, where $\Delta_g$ is defined w.r.t. the flat structure on $V^-$,
					\item [(iv)] as a section of $\MI$, $df\in L^2$.
				\end{itemize}
			\end{definition}
			As $df$ might be considered as a $\MI$ valued 1-form, which leads to the following definition.
			\begin{definition}
				\label{del_multivaluedform}
				A multivalued  harmonic 1-form $\al\in L^2$ is a triple $(\Sigma,\MI,\al)$ consisting of 
				\begin{itemize}
					\item [(i)] a flat $\mathbb{R}$ bundle $\MI$ with holonomy $-1$ on small loop linking $\Sigma$, 
					\item [(ii)] a section $\al\in \Gamma(L\setminus\Sigma, T^{\st}L\otimes \MI)$ such that $d\al=d\st \al=0$, where $d$ and $d\st$ are taken w.r.t. the flat structure of $\MI$.
				\end{itemize}
			\end{definition}
			In addition, $\al$ is called a $\ZT$ harmonic 1-form if $|\al|$ is bounded near $\Sigma$ and $\|\al\|_{L^2}=1.$
			
			A model example of multivalued function would be $L=\mathbb{C}$, $\Sigma=\{0\}$, $f=\Re(z^{k+\frac12})$ with $k\geq 1$ and $v=df=\frac{2k+1}{2}\Re(z^{k-\frac12}dz)$. We could either think of $v$ as a two valued 1-from over $\mathbb{C}\setminus \{0\}$ or a section of a $\mathbb{R}$ vector bundle which has monodromy $-1$ along loops linking $\{0\}$. 
			
			In general, the multivalued harmonic 1-form are purely topological and we will briefly review the construction in \cite{donaldson2019deformations}. 
			
			The kernel of $\chi':\pi_1(L\setminus \Sigma)\to \{\pm 1\}$ is a index two normal subgroup of $\pi_1(L\setminus \Sigma)$, which gives us a covering $p:L'\to L\setminus \Sigma$. In addition, the deck transformation defines an involution on $L'$. We choose a Riemannian metric on $L$ and write $N_{\Sigma}$ be the normal bundle. As $\Sigma$ is oriented, using the co-orientation, $N_{\Sigma}$ is a complex vector bundle and the flat $\MI$ bundle defines a square root $N_{\Sigma}^{\frac12}. $ Using the normal exponential map, a neighborhood of $\Sigma$ in $L$ could be identified with a neighborhood of the zero section in $N_{\Sigma}$. 
			
			We write $\tL:=L'\cup \Sigma$ and we could associate a smooth structure on $\tL$. Let $U$ be an open set of $m\in \Sigma$ with trivialization $N_{\Sigma}|_{U}\cong U\ti \mathbb{C}$. Let $\tx=(\tx_3,\cdots,\tx_n)$ be the coordinate on $U$, $\tz$ be the coordinate on $\mathbb{C}$, then for $\tz\neq 0$, the covering map is locally given by 
			$$
			p:U\ti (\mathbb{C}\setminus\{0\}) \to U\ti (\mathbb{C}\setminus\{0\}),\;(\tz,\tx)\to (\tz^2,\tx),
			$$
			which naturally extends to the zero section and gives $\tL$ a smooth structure. In addition, the involution on $L'$ extends naturally to an involution $\sigma$ on $\tL$ with fixed point the branch locus. To save notation, we also write $\Sigma$ be the branch locus on $\tL$.
			
			A neighborhood of $\Sigma\subset \tL$ could be identified with a neighborhood of the zero section of $N_{\Sigma}^{\frac12}$. Without loss of generality, we could assume that $\sigma$ acts on each fibre of $N_{\Sigma}^{\frac12}$ by changing the sign and $\Sigma$ is the fixed point of the involution. 
			
			Let $(\Sigma,\MI,\al)$ be a multivalued harmonic 1-form and with $p:\tL\to L$ the double branched covering constructed above. Based on our assumptions in Definition \ref{del_multivaluedform}, $p^{\st}\MI$ is trivial on $\tL\setminus \Sigma$ and $p^{\st}\al$ is a 1-form over $\tL$. As $p$ is only Lipschitz along $\Sigma$, $p^{\st}\al$ will not be a smooth 1-from and the pull-back Riemannian metric $p^{\st}g$ will be a Lipschitz metric in a neighborhood of $\Sigma\subset \tL$. As $\al$ is a harmonic section, $p^{\st}\al$ will satisfy the harmonic equations for the pull-back metric $p^{\st}g$.
			
			By \cite[Lemma 1.5]{wang93modulispaces}, the space of $L^2$ cohomology
			$$
			\{\al\in \Omega^1(\tL)|\al\in L^2,\;d\al=d\st_{p^{\st}g}\al=0\;\mathrm{over}\;\tL\setminus Z\}
			$$
			is naturally isomorphism to the singular cohomology $H^1(\tL;\mathbb{R})$. From another side, the involution $\sigma:\tL\to \tL$ induces a decomposition $H^i(\tL;\mathbb{R})=H^i(\tL;\mathbb{R})^{+}\oplus H^i(\tL;\mathbb{R})^{-}$, while $H^1(\tL;\mathbb{R})^{\pm}$ is the $\pm 1$ eigenvalue of $\sigma^{\st}$. Therefore, $p^{\st}\al$ represents an element in $H^1(\tL;\mathbb{R})^-$ with $\sigma^{\st}p^{\st}\al=-p^{\st}\al$.
			
			Given an element $[\delta]\in H^1(\tL;\mathbb{R})^-$, one could find a harmonic representative $\al_{\delta}$ and there exists a harmonic affine bundle section with $\al_{\delta}=df$. Moreover, the representation $\chi:\pi_1(L\setminus \Sigma)\to \Xi$ is up to conjugate determined by $\chi':\pi_1(L\setminus \Sigma)\to \{\pm 1\}$ and $[\delta]\in H^1(\tL;\mathbb{R})^-$.
			
			We have the following proposition proved by Donaldson \cite{donaldson2019deformations}, which is a version of the Hodge theorem for $p^{\st}g$.
			\begin{proposition}{ \cite[Section 4]{donaldson2019deformations}}
				\label{prop_l2variationexisttence}
				Given $\chi:\pi_1(L\setminus \Sigma)\to \Xi$ satisfies (ii) of Definition \ref{def_multivaluedeifnition}, then there exists an unique harmonic section $f$ of $V^-$. In addition, $\sigma^{\st}\pi^{\st}df=-\pi^{\st}df$.
			\end{proposition}
			
			We synthesize all of the above discussions to arrive at the following proposition.
			\begin{proposition}{\cite[Page 2]{donaldson2019deformations}}
				\label{prop_equivalentofconcepts}
				Let $\Sigma$ be an oriented embedded codimensional 2 smooth submanifold, the following four concepts are equivalent :
				\begin{itemize}
					\item [(i)] a representation $\chi:\pi_1(L\setminus \Sigma)\to \Xi$ such that for any small loop linking $\Sigma$, $\chi$ maps to a reflection,
					\item [(ii)] a multivalued harmonic function $(\Sigma,V^-,f)$, 
					\item [(iii)]  a multivalued harmonic 1-form $(\Sigma,\MI,\al)$,
					\item [(iv)] a double branched covering $\tL\to L$ branched along a codimension 2 submanifold $\Sigma$ together with $[\delta]\in H^1(\tL;\mathbb{R})^-$.
				\end{itemize}
			\end{proposition}
		\end{subsection}
		
		\subsection{Topological constraints for multivalued 1-form}
		In this subsection, we will discuss the topology for the branched covering $p:\tL\to L$ and the constraint for the existence of multivalued 1-form. 
		
		Let $\sigma$ be the involution on $\tL$, $C_{\st}(L),C_{\st}(\tL)$ be the singular chains on $L$ and $\tL$, then the transformation morphism for the branched covering $p:\tL \to L$ is a map of singular chain $$T:C_{\st}(L)\to C_{\st}(\tL),\;T(\gamma)=\tilde{\gamma}+\sigma^{\st}\tilde{\gamma},$$ where $\gamma$ is a small enough chain and $\tilde{\gamma}$ is any lifting of $\gamma$. In other word, the transformation morphism map a singular chain to the sum of two distinct lifts.
		
		\begin{lemma}
			\label{lem_decompositioninvolutionss}
			$\mathrm{Im}\;p^{\st}\subset H^i(\tL;\mathbb{R})^+$ and $p^{\st}:H^i(L;\mathbb{R})\to H^i(\tL;\mathbb{R})^+$ is an isomorphism.
		\end{lemma}
		\begin{proof}
			As $p\circ\sigma=p$, for $x\in H^i(L;\mathbb{R})$, we have ${p}^{\st}x\in H^i(L;\mathbb{R})^+$. Let $T_{\st}:H_i(L;\mathbb{R})\to H_i(\tL;\mathbb{R})$ be the induce map of the transformation map, then we have $p_{\st}\circ T_{\st}=2$ and $T_{\st}\circ p_{\st}|_{H_i(\tL;\mathbb{R})^+}=2.$ Therefore, $p^{\st}$ is both injective and surjective thus induces an isomorphism.
		\end{proof}
		
		Let $b_1(\tL)^{\pm}:=\dim H^1(\tL;\mathbb{R})^{\pm}$, then we have $b_1(\tL)=b_1(\tL)^++b_1(\tL)^-$ with $b_1(\tL)^+=b_1(L)$. Combing with the Hodge theorem for the pull-back $L^2$ metric, we obtain the following corollary.
		
		\begin{corollary}
			\label{cor_bettinumberfirst}
			Suppose there exists a multivalued harmonic function $(\Sigma,\chi,f)$ over $L$, then for the double branched covering $p:\tL\to L$, we have $b_1(\tL)>b_1(L)$.
		\end{corollary}
		
		We could construct multivalued harmonic 1-form in the following way.
		\begin{example}
			\label{ex_branched1from}
			Let $(L,g)$ be a Riemannian 3-dimensional rational homology sphere and $\Sigma$ be an oriented link of $L$, there exists a canonical morphism $\mu:\pi_1(L\setminus \Sigma)\to \{\pm 1\}$, which defines a double branched covering $p:\tL\to L$ branched along $\Sigma$. We refer to \cite[Chapter 7]{lickorish97introductiontoknottheory} for details of this construction. By \cite[Corollary 9.2]{lickorish97introductiontoknottheory}, suppose the determinant of the link is zero, then $H^1(\tL;\mathbb{R})\neq 0$. By Lemma \ref{lem_decompositioninvolutionss}, we have $H^1(\tL;\mathbb{R})^-=H^1(\tL;\mathbb{R})$. By $L^2$ Hodge theorem, for the Riemannian metric $p^{\st}g$, for any $[\delta]\in H^1(\tL;\mathbb{R})^-$, we could find a harmonic representative $\al_{[\delta]}$ that gives a multivalued harmonic 1-form on $L$.
		\end{example}
		As the determinant of a link is zero implies the link has more than one component, Haydys \cite{haydys2020seiberg} showed that the branch locus $\Sigma$ in this case is disconnected. In general, let $(\Sigma,\MI,\al)$ be a multivalued harmonic 1-form, then by Proposition \ref{prop_equivalentofconcepts}, $H^1(\tL;\mathbb{R})^->0$, which give constraints of the topological type of $\Sigma$.  
		
		\begin{lemma}{\cite{leeweintraub1995}}
			There exists a long exact sequence 
			\begin{equation}
				\label{eq_longexactsequence}
				\begin{split}
					\cdots	\rightarrow H_i(L,\Sigma;\mathbb{Z}_2)\xrightarrow{T} H_i(\tL;\ZT)\xrightarrow{p_{\st}}H_i(L;\ZT)\xrightarrow{\pa_\st}H_{i-1}(L,\Sigma;\mathbb{Z}_2)\rightarrow\cdots,
				\end{split}
			\end{equation}
			where $T$ is the transformation map.
		\end{lemma}
		\begin{proof}
			We consider the chain with $\mathbb{Z}_2$ coefficients. As $T_{\st}|_{C_{\st}(\Sigma)}=0$, we have short exact sequence
			$$
			0\to C_{\st}(L,\Sigma;\mathbb{Z}_2)\xrightarrow{T_{\st}} C_{\st}(\tL;\mathbb{Z}_2) \xrightarrow{p_{\st}} C_{\st}(L;\mathbb{Z}_2)\to 0,
			$$
			which implies the long exact sequence. 
		\end{proof}
		
		\begin{proposition}
			Suppose there exists a multivalued harmonic 1-form on $L$ with $H_1(L;\mathbb{Z})=0$, then $\Sigma$ has to be disconnected. 
		\end{proposition}
		\begin{proof}
			Let $k$ be the number of connected components, by the long exact sequence of relative homology of $(L,\Sigma)$, we obtain 
			$$\to H_1(L;\mathbb{Z})\to H_1(L,\Sigma;\mathbb{Z})\to H_0(\Sigma;\mathbb{Z})\to H_0(L;\mathbb{Z})\to H_0(L,\Sigma;\mathbb{Z})\to 0.$$
			As $H_1(L;\mathbb{Z})=0$, we obtain $H_1(L,\Sigma;\mathbb{Z})\cong \mathbb{Z}^{k-1}.$ By \eqref{eq_longexactsequence}, we obtain 
			$$
			\to H_1(L,\Sigma;\mathbb{Z}_2)\xrightarrow{T_{\st}} H_1(\tL;\mathbb{Z}_2)\to 0.
			$$
			Suppose $\Sigma$ is connected, $k=1$, then $H_1(\tL;\mathbb{Z}_2)=0$ which implies $b_1(\tL)=0$. However, this violates the existence of multivalued harmonic 1-form.
		\end{proof}
		
		\begin{proposition}
			Let $L$ be a rational homology 3-sphere, then for $i=1, 2$ and $\al\in H^i(\tL;\mathbb{R})$, we have $\sigma^{\st}\al=-\al$. In particular, the rational cohomology ring structure of $\tL$ is trivial.
		\end{proposition}
		\begin{proof}
			By Lemma \ref{lem_decompositioninvolutionss}, for $i=1,2$, we have $H^i(\tL;\mathbb{R})\cong H^i(\tL;\mathbb{R})^-$. For $\al,\be\in H^1(\tL;\mathbb{R})$, we compute $-\al\cup \be=\sigma^{\st}(\al\cup \be)=\sigma^{\st}\al\cup \sigma^{\st}\be=\al\cup \be,$ thus $\al\cup\be=0$. 
		\end{proof}
		
		\subsection{The differential forms on the branched covering}
		\label{subsec_43}
		Let $p:\tL\to L$ be the branched covering map and $\sigma$ be the involution map defined on $\tL$. The involution induces a decomposition of $k$-forms $$\Omega^k(\tL)=\Omega^k(\tL)^-\oplus \Omega^k(\tL)^+,$$
		where $\Omega^k(\tL)^+\cong\Omega^k(L),\;\Omega^k(\tL)^-\cong\Omega^k(L,\MI)$, and $\Omega^k(L,\MI)$ is $\MI$ valued k-form on $L\setminus \Sigma$. Therefore, given any $\tal\in\Omega^k(\tL)$, we could write $\tal=p^{\st}\al^++p^{\st}\al^-$ where $\al^+\in\Omega^1(\tL)$ and $\al^-\in \Omega^k(L,\MI)$. We define $\mfa:=\al^++\al^-$ and $p^{\st}\mfa=\tal.$ Moreover, as $\al^-$ could be understood as a two valued 1-form over $L$, we could also regard $\mfa$ be a two valued 1-form on $L$. We usually call $\mfa$ as a pair as $\mfa$ consists of two parts.
		
		Near the branch locus, let $m\in\Sigma\subset L$ with $U\subset L$ a neighborhood of $m$, we choose $x=(z=x_1+\sqrt{-1}x_2,x_3,\cdots,x_n)$ such that $\Sigma\cap U=\{z=0\}$. Let $\tm=p^{-1}(m)$ and we could choose a coordinate system $(\tz=\tx_1+\sqrt{-1}\tx_2,\tx_3,\cdots,\tx_n)$ on $\tU=p^{-1}(U)$ such that $p$ is given by
		\begin{equation*}
			\begin{split}
				p(\tz,\tx_3,\cdots,\tx_n)=(\tz^2,\tx_3,\cdots,\tx_n).
			\end{split}
		\end{equation*}
		We compute
		$$
		p^{\st}dz=2\tz d\tz,\;p^{\st}(z^{k-\frac12}dz)=2\tz^{2k}d\tz,\;p^{\st}dx_i=d\tx_i.
		$$
		We have an unfavorable situation when $k=0$, where the pullback would be $p^{\st}(z^{-\frac{1}{2}}dz)=2d\tz\in \Omega^1(\tL)^-$. 
		
		For $\tal=p^{\st}\mfa=p^{\st}\al^++p^{\st}\al^-\in\Omega^1(\tL)$, suppose $\al^-$ is bounded, then by the assumption on the holonomy of the bundle, we have $|\al^-||_{\Sigma}=0$ and $\tal|_m=p^{\st}\al^+|_m$. Recall that $p^{\st}$ is an isomorphism over $L\setminus \Sigma$, for $\tal$ will corresponding $\al^-$ bounded, we could define the following inclusion
		\begin{equation}
			\iota_{\tal}=\iota_{\mfa}:\tL\to T^{\st}L,\;\iota_{\tal}(\tx)=\iota_{\mfa}:=(p(\tx),(p^{\st})^{-1}\al|_{\tx}).
			\label{eq_definingtheinclusionmap}
		\end{equation}
		Geometrically, $\iota_{\tal}$ will be just the graph of the two-valued 1-forms $\tal=\al^++\al^-$ and we defines as follows. 
		\begin{definition}
			Let $\iota:\tL\to T^{\st}L$ be an inclusion with image $\Gamma_{\iota}:=\iota(\tL)\subset T^{\st}L$. We call $\Gamma_{\iota}$ a 2-valued graph if 
			\begin{itemize}
				\item [(i)] for all $\tx\in\tL$ ,we have $\pi\circ \iota(\tx)=\pi\circ\iota(\sigma(\tx))$,
				\item [(ii)] $\iota(\tx)=\iota\circ\sigma(\tx)$ if and only if $\tx\in\Sigma$.
			\end{itemize}
		\end{definition}
		
		\begin{proposition}
			Given a 2-valued graph $\Gamma_{\iota}$, there exists $\tal\in \Omega^1(\tL;\mathbb{R})$ such that $\Gamma_{\iota}$ is the graph of $\tal$.
		\end{proposition}
		\begin{proof}
			Let $x\in L$, then we could write $p^{-1}(x)=\{\tx,\sigma(\tx)\}$. We define $\al^+|_x=\frac{\iota(\tx)+\iota(\sigma(\tx))}{2}$, then $\al^+$ is a well-defined 1-form on $L$. Using the linear structure on each fiber of $T^{\st}L$, we could define a shift of graph $\Gamma'=\Gamma-\al^+$. To be more explicit,  we define $\iota'$ as
			$$
			\iota': \tL\to T^{\st}L,\;\iota'(\tx):=\iota(\tx)-\al^+|_{p(\tx)},
			$$
			which satisfies $\iota'(\tx)=-\iota'(\sigma(\tx))$ and $\Gamma'$ will be the image of $\iota'$ The condition $\iota'(\tx)=-\iota'(\sigma(\tx))$ implies that $\iota'=\iota_{p^{\st}\al^-}$ for some $\al^-\in\Omega^1(L,\MI)$. In addition, we have $\iota=\iota_{\tal}$, with $\tal:=p^{\st}\al^++p^{\st}\al^-$.
		\end{proof}
		\subsection{Norms of multivalued forms.}
		In this subsection, we will discuss functional spaces that will be used in our paper, for more details, we refer to Donaldson \cite{donaldson2019deformations}. Let $(L,g)$ be a Riemannian manifold, $\Sigma$ be an embedded codimension 2 submanifold, $\MI$ be a flat $\mathbb{R}$ bundle statisfies Definition \ref{del_multivaluedform} with branched covering $p:\tL\to L$. 
		\subsubsection{Norms on $\MI$ valued forms.}
		Let $U$ be a open neighborhood of a point $m\in\Sigma$ in $L$ and $U'$ be an open neighborhood of the zero section of $N_{\Sigma}$, then for a suitable choice of $U$ and $U'$, there exists a diffeomorphism $z:U\to U'$ which is the inverse of the normal exponential map. The co-orientation of $\Sigma$ and the real line bundle $\MI$ makes $N_{\Sigma}$ a complex vector bundle with square root. If $p$ is a half integer, for $\sigma_p\in\Gamma(N_{\Sigma}^{-p})$, $\sigma_p\ze^p$ is a section of the complexfied bundle $N_{\Sigma}$ over $U\setminus \Sigma$. Let $t=(t_3,\cdots,t_n)$ be coordinate on $\Sigma$, then together with $z=re^{i\theta}$, we obtain a coordinate system on $U$.
		
		Fix $\gamma\in (0,\frac12)$, we define the H\"older norm on sections of $V$ to be 
		$$
		\|s\|_{\MC^{,\gamma}}=\sup_{p\neq p'}\frac{|s(p)-s(p')|}{|p-p'|^{\gamma}},
		$$
		where the supremum is taken for $p=(z,t),p'=(z',t')$ satisfy $|p-p'|\leq \frac12\min(|z|,|z'|)$, which means that there is no ambiguity in defining $|s(p)-s(p')|$. By parallel transport around a polygon, one sees that $|s(z,t)|\leq C|s|_{\MC^{,\gamma}}|z|^{\ga}$, which controls the $\MC^{0}$ norm. 
		
		We define $\MT_k$ be the set of differential operators given by elements of degree $k$ in vectors fields of the form $\{r\frac{\pa}{\pa r},\;\frac{\pa}{\pa \theta},\;\frac{\pa}{\pa t_i}\}$. The $\MD^{k,\ga}$ norm of sections of $V^-$ could be written as
		$$\|s\|_{\MD^{k,\ga}}=\max_{0\leq j\leq k,D\in\MT_k}\|Ds\|_{\MC^{,\gamma}}.$$
		Similarly, the $\MD^{k,\ga}$ norm could be defined for $\Omega^k(L,\MI)$ using the Riemannian metric on $L$.
		\subsubsection{Norms on pairs.}
		In most cases considered in our paper, we only have a natrual metric $g$ on $L$ instead of $\tL$. When we try to measure forms on $\tL$, we usually treat them as a form on $L$ and a section on $\MI$. One might also be able to define similar norms using $p^{\st}g$ over $\tL$.
		
		Let $f$ be a continuous function on $(L,g)$ with $\na$ the Levi-Civita connection, we define the $\MC^{k,\ga}$ norm $\|\cdot\|_{\MC^{k,\ga}}$ on $L$ in the usual way as
		$$
		\|f\|_{\MC^{k,\ga}}=\sum_{i=0}^k\sup|\na^i f|+\sup_{p\neq p'}\frac{|\na^if (p)-\na^i f(p')|}{|p-p'|^{\ga}},
		$$ 
		which extends to differential forms on $L$.
		
		Let $\tal$ be a p-form on $\tL$, then as Subsection \ref{subsec_43}, we could write $\tal=p^{\st}\al^++p^{\st}\al^-$ with $\al^+\in \Omega^p(L)$ and $\al^-\in\Omega^p(L,\MI)$. We also write $\mfa=\al^++\al^-$, then we define the $\MC^{k,\ga}$ norm $\|\cdot\|_{\MC^{k,\ga}}$ of $\tal$ or $\mfa$ as
		\begin{equation}
			\begin{split}
				\|\tal\|_{\MC^{k,\ga}}=\|\mfa\|_{\MC^{k,\ga}}:=\|\al^+\|_{\MC^{k,\ga}}+\|\al^-\|_{\MC^{k,\ga}}.
			\end{split}
		\end{equation}

		\begin{subsection}{Elliptic theory for the harmonic equations with multivalued}
			The elliptic theory for multivalued functions, established in \cite{donaldson2019deformations}, will be introduced in the following subsection. We will introduce the definition of nondegenerate multivalued functions and discuss examples of them.
			
			\begin{definition}
				\label{def_nondegenerate}
				Given a representation $\chi$ satisfies Definition \ref{def_multivaluedeifnition} (ii), we write $V^-$ be the associative affine bundle with $\MI$ the vertical line bundle. Given a multivalued function $f\in \Gamma(V^-)$, $f$ is called nondegenerate if
				\begin{itemize}
					\item [(i)] $f$ has an expansion 
					\begin{equation}
						f=\Re(A\ze^{\frac12}+B\ze^{\frac32})+o(r^{2})
					\end{equation}
					\item [(ii)]$A\in \Gamma(N^{-\frac12})$ is identically zero and $B\in \Gamma(N^{-\frac32})$ is nowhere vanishing section along $\Sigma$. 
				\end{itemize}
				In addition, $\al^-\in \Omega^1(L,\MI)$ is called nondegenerate if there exists $V^-$ and $f\in\Gamma(V^-)$ with $df=\al^-$ such that $f$ is nondegenerate. 
				
				A pair $\mfa=\al^++\al^-$ is called nondegenerate if $\al^-$ is nondegenerate and harmonic if $\al^{\pm}$ are both harmonic.
			\end{definition}
			
			The Laplacian operator is understood as $$\Delta_g:\MD^{k+2,\al}(V^-)\to \MD^{k,\al}(\MI),$$ 
			and by the work of Donaldson \cite{donaldson2019deformations}, it will be an isomorphism. 
			\begin{theorem}{\cite{donaldson2019deformations}}
				\label{thm_donaldsonremainningestimate}
				Let $\rho\in \Gamma(\MI)$ be a $\MD^{k,\ga}$ section, then there exists a $f\in \MD^{k+2,\ga}(V^-)$ satisfies $\Delta_g f=\rho$. In addition, $f$ has an asymptotic expansion
				\begin{equation}
					f=\Re(A\ze^{\frac12}+B\ze^{\frac32})+E,
				\end{equation} 
				with $$A\in  \MC^{k+1,\ga+\frac12}(N^{-\frac{1}{2}}),\;B\in \MC^{k,\ga+\frac12}(N^{-\frac32})$$ and
				$$|E|\leq r^{2+\ga}\|\rho\|_{\MD^{2,\ga}},\;|\na E|\leq r^{1+\ga}\|\rho\|_{\MD^{2,\ga}},\;|\na\na E|\leq r^{\ga}\|\rho\|_{\MD^{,\ga}}.$$ 
			\end{theorem}
			
			In particular, the harmonic multivalued function will satisfy (i) of Definition \ref{def_nondegenerate}. The following example explains the reason to introduce the concepts of nondegenerate.
			\begin{example}
				Let $(L,g)$ be a Riemannian surface and let $(\Sigma,g,\MI,v)$ be a $\ZT$ harmonic 1-form, with $\Sigma$ a collection of points on $L$. Then there exists a holomorphic $\MI$ valued 1-form $\al$ such that $\Re(\al)=v$ and $\al\otimes\al$ is a meromorphic quadratic differential. Let $p\in \Sigma$, we use $z$ be a local coordinates around $p$ with $p=\{z=0\}$. Then near $p$, $\al\otimes \al$ will have the following asymptotic
				$$\al\otimes\al=A z^{-1}dz\otimes dz+ Bzdz\otimes dz+\MO(z^2).$$
				The condition $A=0$ means $\al\otimes\al$ is a holomorphic quadratic differential and $B\neq 0$ means $\al\otimes\al$ has simple zeroes. A holomorphic quadratic differential with simple zeros has $4g-4$ different zeros, therefore $\Sigma$ must contains $4g-4$ different points. 
			\end{example}
			
			Another source of nondegenerate $\ZT$ harmonic 1-form coming from K\"ahler manifold.
			\begin{example}
				Let $L$ be a K\"ahler manifold and $\Omega_L^{(1,0)}$ be the bundles of (1,0) form, we write $S^2\Omega_L^{1,0}$ be its' second symmetric power. At each point $x\in L$, we say a section $s|_x\in S^2\Omega_L^{1,0}$ is rank $1$ if $s$ could be written as $s|_x=v\otimes v$, where $v\in \Omega_L^{1,0}$ and $s$ is rank $0$ if $s$ vanishes at $L$. Then one could consider the following space
				$$\MB_L:=\{s\in S^2\Omega_L^{(1,0)}|\bar{\pa}s=0,\;\mathrm{rank}\;s\leq 1\}.$$
				
				Let $s\in\MB_L$, suppose $\Sigma:=s^{-1}(0)$ is a smooth codimension 2 submanifold and suppose $s$ has odd vanishing order along $D$, then for any $\chi:\pi_1(L\setminus \Sigma)\to \{\pm 1\}$, one could defines a flat $\mathbb{R}$ bundle $\MI$ and a square root of $\Re s$ such that $\sqrt{\Re s}\in \MI\otimes \Omega^1$. The condition that the flat $\MI$ bundle has monodromy $-1$ is equivalent to $s$ has odd vanishing order along $D$. 
			\end{example}
			
			There are also examples of nondegenerate $\ZT$ harmonic 1-forms over 3-manifolds. 
			\begin{example}{\cite{he21z3symmetry}}
				There are examples of nondegenerate $\ZT$ harmonic 1-forms constructed over 3-manifolds in \cite{he21z3symmetry}. Let $L$ be a closed 3-manifold and $\Sigma$ be a link on $L$ with determinant $\det(\Sigma)=0$, then for generic metric, there exists nondegenerate $\ZT$ harmonic 1-form over the three cyclic branched covering of $L$ along $\Sigma$. Moreover, \cite{he21z3symmetry} construct infinite numbers of rational homology spheres that admit nondegenerate $\ZT$ harmonic 1-forms. 
			\end{example}
			
			\begin{comment}
			\begin{proposition}
			Let $\rho$ be an $L^2$ section of $\MI$ which could be written as $\rho=\Re(\tau z^{-\frac12})+\xi$,  let $u$ be the unique solution to the equation $\Delta u=\rho$, then 
			\begin{equation}
			u=\Re(\frac{1}{2}r^2\tau z^{-\frac{1}{2}})+u',
			\end{equation} 
			with $\|u'\|_{\MD^{k+2,\al}}\leq \|\xi\|_{\MD^{k,\ga}}+\|\tau\|_{k+2,\al}$
			\end{proposition}
			\begin{proof}
			It is straight forward to compute that $$\Delta (\frac{1}{2}r^2\tau\ze^{-\frac12})=\tau \ze^{-\frac12}+\rho'$$ with $$\|\rho'\|_{\MD^{k,\ga}}\leq C_k\|\tau\|_{k+2,\al}.$$
			Therefore, $\Delta(u-\frac12 r^2\tau \ze^{-\frac12})=\xi-\rho'.$ By Theorem \ref{thm_isomorphism}, we obtain
			$$\|u-\frac12 r^2\tau\xi^{-\frac12}\|_{\MD^{k+2,\al}}\leq \|\xi\|_{\MD^{k,\ga}}+\|\tau\|_{k+2,\al}.$$
			\end{proof}
			\end{comment}
			
		\end{subsection}
		\begin{subsection}{polyhomogeneous expansions}
			In this subsection, we will introduce the expansion theory for multivalued harmonic functions and solutions to linearized multivalued equations. We refer \cite{Mazzeo1991} for more details. 
			
			Let $(\Sigma,\chi,f)$ be a multivalued function with associate affine bundle $V^-$ and $f\in\Gamma(V^-)$ be a section. We call $f$ is conormal if $$r^{-\lam_0}|(r\pa_r)^i\pa_{\theta}^j \pa_t^kf|\leq C,$$ where $i,j,k$ are non-negative integers. We say $f$ is polyhomogeneous at $\Sigma$ if $f$ is conormal and there exists a discrete index set $E=\{\gamma_j\}$ with $\lim_{j\to\infty}\gamma_j=\infty$ such that $f$ has an asymptotic expansion
			\begin{equation}
				f\sim \sum r^{\gamma_j}(\log r)^pf_{jp}(\theta,t),
			\end{equation}
			where the exponents $\gamma_j$ lies in some discrete index set $E\subset \mathbb{R}$ with $\lim_{j\to\infty}\gamma_j=\infty$. In addition, the powers $p$ of $log\;r$ are all positive integers which are only finitely many for each $\gamma_j$. Moreover, the $\sim$ means 
			\begin{equation*}
				|f-\sum_{j\leq N} r^{\gamma_j}(\log r)^pf_{jp}(\theta,t)|\leq Cr^{\gamma_{N+1}}(\log r)^q,
			\end{equation*}
			where the term on the right is the next most singular term in the expansion. In addition, we requires that the corresponding statement holds by differentiating any finite number of times. 
			
			\begin{theorem}{\cite{Mazzeo1991}}
				\label{thm_polyhomo}
				Suppose $\rho\in\ca(\MI)$ is polyhomogeneous with order $\frac12$ and index set $\{k+\frac12|k\in\mathbb{N}\}$, let $f\in V^-$ be the solution to $\Delta f=\rho$ in Theorem \ref{thm_donaldsonremainningestimate}, then $f$ has a polyhomogeneous expansion with index set $\{k+\frac12|k\in\mathbb{N}\}$. In particular, we could write
				$
				f=\Re(Az^\frac12+Bz^{\frac32})+E,
				$
				with $E$ an expansions $$E\sim \sum_{k\geq 2}\sum_{0\leq p\leq p_k} E_{k,p}(\theta,t)r^{k+\frac12}(\log r)^p,$$
				where $p_k$ are finite integers.
			\end{theorem}
			\begin{proof}
				$f$ is conormal follows by \cite[Section 3]{donaldson2019deformations}. We will give a proof of existence of polyhomogeneous expansion in the Appendix \ref{sec_backgroundpolyhomogeneoussolution}, Proposition \ref{prop_appendixpolyh}. 
			\end{proof}
			\begin{remark}
				When we talk about the polyhomogeneous expansion, we always assume that we choose a trivialization of the $N_{\Sigma}$ in a neighborhood of a point $p\in\Sigma$. However, the definition of nondegenerate will be independent of this choice.
			\end{remark}
			
			Let $\MS$ be the space of codimension 2 submanifolds in $L$, $\MMmec$ be the moduli space of Riemannian metric on $L$, which both we associate a $\MC^{\infty}$ topology. Given $\Sigma\in\MS$, let $R$ be the conjugate equivalent class of $\pi_1(L\setminus \Sigma)\to \Gamma$, then by Proposition \ref{prop_equivalentofconcepts}, $R$ could be identified with a real vector space $H^1(\tL;\mathbb{R})^-$.
			
			\begin{theorem}{\cite[Theorem 1]{donaldson2019deformations}}
				\label{thm_donaldsondeformation}
				Suppose $(\Sigma_0,\chi_0,f)$ is a nondegenerate $\ZT$ harmonic function over Riemannian manifold $(L,g_0)$, then there exists a neighborhood $\MU_1$ of $(g_0,\chi_0)\in \MMmec\ti R$ and a neighborhood $\MU_2$ of $\Sigma_0$ in $\MS$ such that for any $(g,\chi)\in\MU_1$, there exists a unique $\Sigma\in\MU_2$ such that the $\ZT$ harmonic function for $(\chi,\Sigma)$ on metric $g$ is nondegenerate.
			\end{theorem}
			
			\begin{remark}
				Let $f_{(\Sigma,g,\chi)}$ be the $\ZT$ harmonic function obtained in Theorem \ref{thm_donaldsondeformation}, we write the leading expansion of $$f_{(\Sigma,g,\chi)}\sim \Re(A(\Sigma,g,\chi)z^{\frac12}+B(\Sigma,g,\chi)z^{\frac32}).$$
				The original statement of Theorem \ref{thm_donaldsondeformation} is for $(g,\chi)$ sufficiently close to $(g_0,\chi_0)$, we have $A(\Sigma,g,\chi)=0$. However, by \cite[Proposition 8]{donaldson2019deformations}, $B(\Sigma,g,\chi)$ is a continuous function on $\Sigma$. Therefore, as $B(\Sigma_0,g_0,\chi_0)$ is nowhere vanishing along $\Sigma_0$, $B(\Sigma,g,\chi)$ also no where vanishing along $\Sigma$, which implies $f_{(\Sigma,g,\chi)}$ is nondegenerate. 
			\end{remark}
		\end{subsection}
	\end{section}
	
	\begin{section}{The special Lagrangian equation for multivalued 1-forms}
		\label{sec_5}
		By McLean's deformation theorem (Theorem \ref{theorem_classicaldeformation}), the $\MC^1$ deformations of special Lagrangian submanifolds are parameterized by harmonic 1-forms. Motivated by Donaldson \cite{donaldson2019deformations} as in Question \ref{question_donaldsonsquestion}, we will explicit discuss the special Lagrangian equation for multivalued 1-forms.
		\begin{subsection}{Examples of branched deformations}
			To warm up, we introduce some toy examples of the branched deformation of special Lagrangian submanifolds when the background Calabi-Yau is hyperK\"ahler, where we could use hyperK\"ahler rotation to produce examples.
			\subsubsection{Local model for the branched deformation}
			\label{subsubsection_localmodel}
			We start with a toy model for Question \ref{question_donaldsonsquestion} of the branched deformation. Let $\mathbb{C}^2$ be the two dimensional complex plane with complex coordinates $z=x_1+\sqrt{-1}y_1,w=x_2+\sqrt{-1}y_2$. We consider the following Calabi-Yau structure $(\mathbb{C}^2$,$I$,$\omega$,$\Omega)$ with metric $g$, where
			\begin{equation}
				\begin{split}
					&g=dx_1^2+dx_2^2+dy_1^2+dy_2^2,\\
					&J(\pa_{x_1})=\pa_{x_2},\;J(\pa_{y_1})=-\pa_{y_2},\;\omega=dx_1\we dx_2-dy_1\we dy_2,\\
					&\Im\Omega=dx_1\we dy_2+dy_1\we dx_2,\;\Re\Omega=-dx_1\we dy_1-dx_2\we dy_2. 
				\end{split}
			\end{equation}
			It is straight forward to see that $dz\we dw=\omega-i\Im \Omega$. 
			
			As a real 2-dimensional submanifold $L\subset \mathbb{C}^2$ is a special Lagrangian submanifold if $\omega|_{L}=\Im\Omega|_{L}=0$. Therefore, $L$ is a special Lagrangian if and only if $dz\we dw|_{M}=0$, which holds when $L$ is a holomorphic submanifold with respect to the complex coordinates $(z,w)$. 
			
			Let $L_0$ be the complex $z$ plane, which is also a special Lagrangian submanifold. We consider a multivalued harmonic function $f_t^k=t\frac{2}{2k+1}\Re(z^{\frac{1}{2}+k})$ over $L_0\setminus\{z=0\}$ and write $z=re^{i\theta}$, then the corresponding multivalued harmonic 1-form would be $$v_t^k:=df_t^k=t(r^{\frac{2k-1}{2}}\cos((k-\frac12)\theta)dx_1-r^{\frac{2k-1}{2}}\sin((k-\frac12)\theta)dy_1).$$ 
			
			A section of $T^{\st}L_0\cong \mathbb{C}^2$ can be identified with a submanifold under the composition of the following maps $$\Phi:T^{\st}L_0\rightarrow TL_0\rightarrow N_{L_0}\rightarrow X,$$
			where the first map above is given by the Riemannian metric on $L_0$ and the second map is given by the complex structure $J$. 
			
			The defining equations of the graph manifold of $\Phi(v_t^k)$ could be written as $$L_t^k:=\{(z,\Phi(df_t^k|_z))\in\mathbb{C}^2\}=\{(z,w)\in\mathbb{C}^2|w^2=t^2z^{2k-1}\}.$$ 
			
			However, only when $k=1$, $L_t^1$ will be a family of smooth special Lagrangian submanifolds such that $L_t^1\to 2L_0$ as current when $t\to 0$. 
			
			When $k>1$, we found $L_t^k$ will not be a smooth submanifold and when $k=0$, as $df_t^k$ behave as $z^{-1}$ along $\{z=0\}$, $L_t^0$ will be a large deformation of $L_0$. This example also explain the reason for introducing the concept of nondegenerate in Definition \ref{def_nondegenerate}. 
			
			We should also note that the above example is very special because we use extra symmetries, especially hyperK\"ahler rotations to produce special Lagrangian submanifolds, which are branched along the same branch locus $\Sigma=\{z=0\}$.
			In general, we might hope that we could glue back the model solution to obtain a real branched cover, where we refer to \cite{liu2011immersed} for some attempts. However, when $n=\dim_{\mathbb{C}}X\geq 3$, we expect the branch locus itself will be changed for the family, sharing the same behavior as in Theorem \ref{thm_donaldsondeformation}.
			
			\subsubsection{The deformations induced by holomorphic quadratic differentials}
			Besides the toy model, the quadratic differentials over a Riemannian surface $\Sigma$, with genus $g(\Sigma)>1$, are important sources for $\ZT$ harmonic 1-forms. We could also use the nearby hyperK\"ahler structure to produce examples of branched deformations.
			
			The holomorphic structure on $\Sigma$ induces an holomorphic structure on $T^{\st}\Sigma$, which we denote as $I$. Let $K_{\Sigma}$ be the canonical bundle of $\Sigma$ and $\pi:K_{\Sigma}\to \Sigma$ be the projection, then there exists an holomorphic tautological section $\lam\in \Gamma(\pi^{\st}K_{\Sigma})$ such that for any vector $v\in T(T^{\st}\Sigma)$, $\lam|_{(x,\al)}(v)=\al(\pi_{\st}v)$, where $(x,\al)\in K_{\Sigma}$ with $x\in\Sigma$, $\al\in K_{\Sigma}|_x$. We define $\omegac=d\lambda$, then $\omegac$ will be a holomorphic symplectic form. In a neighborhood of the zero section of $T^{\st}\Sigma$, there exists a Feix–Kaledin metric to make it hyperK\"ahler.
			
			\begin{theorem}{\cite{feix2001hk,kaledin1999canonical}}
				There exists $X_{\Sigma}\subset T^{\st}\Sigma$, an open neighborhood of zero section and a Riemannian metric $g$ on $X$ such that $(X,g,I,\omegac)$ is a hyperK\"ahler manifold. 
			\end{theorem}
			
			Let $q$ be a quadratic differential with simple zeroes, $t\in\mathbb{R}$ be a real parameter, we define the following family of curves
			\begin{equation}
				\label{eq_spectralcurve}
				\tSigma_t=\{\lam^2-tq=0\}\subset T^{\st}\Sigma,
			\end{equation}
			which could also be understood as the spectral curves for Higgs bundles \cite{hitchin1987self}.
			
			For small enough $t$, we have $\tSigma_t\subset X_{\Sigma}$. In addition, as $q$ is a quadratic differential with simple zeroes, $\tSigma_t$ are all smooth and have the same topological type, which we denote as $\tSigma$, with the branched covering map $p:\tSigma\to \Sigma$. We write $\tiota_t:\tSigma\to X_{\Sigma}$ be the inclusion map of $\tSigma_t$ and $\tiota_0:\Sigma\to T^{\st}\Sigma$ be the inclusion of the zero section. 
			
			We write $\omega_J:=\Re\omegac,\;\omega_K:=\Im\omegac$ and the corresponding complex structures $J,K$. We define $\Omega=\omega_I+\sqrt{-1}\omega_K$, which is a covariant constant nowhere vanishing form, then $(X_{\Sigma},J,\omega_J,\Omega)$ is a Calabi-Yau structure. 
			\begin{proposition}
				\label{prop_RiemanniansurfacespecialLagrangian}
				Over the Calabi-Yau manifold $(X_{\Sigma},J,\omega_J,\Omega)$, the following holds:
				\begin{itemize}
					\item [(i)] $\lim_{t\to 0}\|\tiota_t-\iota_0\circ p\|_{\ca}=0$,
					\item [(ii)] $[\tSigma_t]=(\tSigma,\iota_t)$ are special Lagrangian submanifolds of $X_{\Sigma}$.
				\end{itemize}
			\end{proposition}
			\begin{proof}
				(i) is straight forward. For (ii), as $\tiota_t^{\st}\lam d\lam=\bar{\pa}q=0,$ we have $\tiota_t^{\st}\omega_J=\tiota_t^{\st}\omega_K=0$. Therefore, $\tiota_t^{\st}\Im\Omega=\tiota_t^{\st}\omega_K=0$, which implies (ii).  
			\end{proof}
			
			\subsection{The splitting of the special Lagrangian equation}
			\label{subsection52splitting}
			Let $(U_L,J,\omega,\Omega)$ be a Calabi-Yau structure on a Weinstein neighborhood $U_L\subset T^{\st}L$ and $\theta$ be the Lagrangian angle of the zero section. 
			
			Given a pair $\mfa=\al^++\al^-$ and we write $\tal=p^{\st}\mfa\in \Omega^1(\tL)$ with the inclusion $[\tL_{\tal}]:=(\tL,\iota_{\al})$ defined in \eqref{eq_definingtheinclusionmap}, which is the graph of $\tal$ on $U_L$. Then $[\tL_{\tal}]$ is a symplectic manifold if $\iota_{\tal}^{\st}\omega=0$, which is equivalent to $d\mfa=0$. The special Lagrangian condition for $\tL$ would be $\iota_{\tal}^{\st}\Im\Omega=0$. 
			\begin{definition}
				\label{def_singularmetricddd}
				The special Lagrangian equation for $[L_{\tal}]$ is defined as $$\SL(\tal):=\st_{p^{\st}g}\iota_{\tal}^{\st}\Im\Omega.$$
			\end{definition}
		The above equation could be understood as an equation on $L$. As $\iota_{\tal}^{\st}\Im\Omega$ is an n-form over $\tL$, under the splitting induced by the involution, we could write 
			\begin{equation}
				\label{eq_32decompositionovertikdel}
				\iota_{\tal}^{\st}\Im\Omega=p^{\st}\be^++p^{\st}\be^-,
			\end{equation} where $\be^+\in\Omega^n(L)$ and $\be^-\in \Omega^n(L,\MI)$. We define $\SL(\mfa)^{\pm}=\st_g\be^{\pm}$, where $\SL(\mfa)^+$ is a function on $L$ and $\SL(\mfa)^-$ is an $\MI$ valued function on $L$. As $\st_{p^{\st}g}p^{\st}\be^{\pm}=p^{\st}(\st_g\be^{\pm}),$ we obtain
			\begin{equation}
				\begin{split}
					\SL(\tal)=p^{\st}(\SL(\mfa)^+)+p^{\st}(\SL(\mfa)^-).
					\label{eq_33equationdecomposition}
				\end{split}
			\end{equation}
			If we write $\sigma$ be the involution over $\tL$, then $\sigma^{\st}p^{\st}(\SL(\mfa)^{\pm})=p^{\st}(\SL(\mfa)^{\pm})$
			
			In particular, we could define $$\SL(\mfa):=\SL(\mfa)^++\SL(\mfa)^-,$$ which could be considered as a two-valued function on $L$. Moreover, as $\mfa$ could be regarded as a two-valued form, we could write 
			$$
			\SL(\mfa)=\sin\theta-d^{\st}\mfa+Q(\mfa),
			$$
			where each terms could also be understood as a two-valued function.
			
			\begin{proposition}
				$[\tL_{\tal}]=(\tL,\iota_{\tal})$ is a special Lagrangian submanifold if and only if $d\mfa=0$ and $\SL(\mfa)=0$.
			\end{proposition}
			\begin{proof}
				The non-trivial part is the $"\mathrm{if}"$ part. By \eqref{eq_33equationdecomposition}, $\SL(\mfa)=0$ implies $\SL(\mfa)^{\pm}=0$, thus $\st_g\be^{\pm}=0$. In addition, as $p^{\st}(\st_g\be^{\pm})=\st_{p^{\st}g}p^{\st}\be^{\pm}$, by \eqref{eq_32decompositionovertikdel}, $\SL(\mfa)=0$ implies $\iota_{\tal}^{\st}\Im\Omega=0$. Therefore, $[\tL_{\tal}]$ is a special Lagrangian submanifold.
			\end{proof}
			
			Similarly, for $\mfa=\al^++\al^-$, for each terms in the decomposition of special Lagrangian equation in Proposition \ref{prop_specialLagrangianlineartermstructure}, \ref{prop_structureofspecialLagrangian}, we could define $Q(\mfa),P(\mfa),S(\mfa)$ with the decomposition $$P(\mfa)=P^+(\mfa)+P^-(\mfa),\;S(\mfa)=S^+(\mfa)+S^-(\mfa).$$ 
			As $(d^{\st}\mfa)^-=d^{\st}\al^-,\;(d^{\st}\mfa)^+=d^{\st}\al^+$, then the decomposition of special Lagrangian equation could be written as
			\begin{equation}
				\label{eq_termsinvolvingpm}
				\begin{split}
					\SL(\mfa)^+=\sin\theta-d^{\st}\al^++Q(\mfa)^+,\;\SL(\mfa)^-=-d^{\st}\al^-+Q(\mfa)^-.
				\end{split}
			\end{equation}
			
			Let $\iota_0:L\to U_L\subset T^{\st}L$ be the inclusion of the zero section, we write $[\iota_0^{\st}\Im\Omega]$ be the homology class in $H^n(L;\mathbb{R})$, then we have the following.
			\begin{proposition}
				\label{prop_integrationvanishes}
				Suppose $[\iota_0^{\st}\Im\Omega]=0\in H^n(L;\mathbb{R})$ , then $\int_L\SL(\mfa)^+d\Vol_g=0$.
			\end{proposition}
			\begin{proof}
				As $[\iota_0^{\st}\Im\Omega]=0$, then $[(\iota_0\circ p)^{\st}\Im\Omega]=0$. We write $\iota_{t\tal}:\tL\to U_L(s)$ be the family of submanifolds defined by the graph of $t\tal$ with $t\in\mathbb{R}$, then 
				$$
				\int_{\tL}\SL(\tal)d\Vol_{p^{\st}g}=\int_{\tL}(\iota_{t\tal})^{\st}\Im\Omega=\int_{\tL}(\iota_0\circ p)^{\st}\Im\Omega=0.
				$$
				In addition, by \eqref{eq_33equationdecomposition}, we have
				\begin{equation*}
					\begin{split}
						\int_{\tL}\SL(\tal)d\Vol_{p^{\st}g}=\int_{\tL}p^{\st}(\SL(\mfa)^+)d\Vol_{p^{\st}g}+\int_{\tL}p^{\st}(\SL(\mfa)^-)d\Vol_{p^{\st}g}.
					\end{split}
				\end{equation*}
				
				As $\sigma^{\st}d\Vol_{p^{\st}g}=d\Vol_{p^{\st}g}$ and $\sigma^{\st}p^{\st}\SL(\mfa)^-=-p^{\st}\SL(\mfa)^-$, we obtain $\int_{\tL}p^{\st}\SL(\mfa)^-d\Vol_{p^{\st}g}=0$, which implies $\int_{\tL}p^{\st}\SL(\al)^+d\Vol_{p^{\st}g}=0$. By definition, we obtain $$\int_L\SL(\al)^+d\Vol_{g}=\frac{1}{2}\int_{\tL}p^{\st}\SL(\al)^+d\Vol_{p^{\st}g}=0.$$
			\end{proof}
		\end{subsection}
		
		\subsection{Special Lagrangian equations over fixed real locus.}
		\label{subsec_5_z2symmetry}
		In Example \ref{exm_fixedpointofreallocals}, we see that the fixed point of an anti-holomorphic involution, which is also called a fixed real locus, will be a special Lagrangian submanifold. An antiholomorphic involution $R$ over a Calabi-Yau manifold $(X,J,\omega,\Omega)$ with Calabi-Yau metric $g$ will satisfy $R^{\st}\omega=-\omega,\;R_{\st}\circ J=-J\circ R_{\st},\;R^{\st}g=g,\;R^{\st}\Omega=\bar{\Omega}.$
		
		Moreover, over $T^{\st}L$, there exists a canonical involution map $R_0:T^{\st}L\to T^{\st}L$ which maps $(x,v)\to (x,-v)$, where $v\in T^{\st}_xL$. 
		\begin{definition}
			\label{def_locallyantiholomorphicinvolution}
			A special Lagrangian manifold $[L]$ is called locally a fixed real locus if there exists a $R_0$-invariant Weinstein neighborhood $U_L\subset T^{\st}L$ such that $R_0$ will be an anti-holomorphic involution for the pull-back Calabi-Yau structure over $U_L$.
		\end{definition}
		
		Unfortunately, not every special Lagrangian submanifold is locally a fixed real locus. As $R_0$ preserve the Riemannian metric, then a fixed real locus will be a minimal submanifold. Over $\mathbb{C}^2$, using the hyperK\"ahler rotations, every holomorphic submanifold is a special Lagrangian submanifold, which is not necessary to be minimal.
		
		For a fixed real locus, the Weinstein neighborhood also have extra symmetry. 
		\begin{proposition}
			\label{prop_antiholomoinvolutionneighborhood}
			Suppose $[L]$ is the fixed real locus of a global anti-holomorphic involution $R$, then we could find a $R_0$-invariant neighborhood $U_L$ on $T^{\st}L$, a $R$-invariant neighborhood $U$ of $[L]$ in $X$ and a diffeomorphism $\Phi:U_L\to U$ such that $\Phi^{\st}\omega=\omega_0,\;\Phi^{\st}R=R_0$.
		\end{proposition}
		\begin{proof}
			The claim is followed by a step-by-step check of the original proof of the Weinstein neighborhood theorem, therefore we will leave the verification to the reader.
		\end{proof}
		
		\begin{proposition}
			\label{prop_ztsymmetry}
			Let $L$ be a locally fixed real locus and let $\al^-$ be a multivalued 1-form over $L$ with $\iota_{\al^-}:\tL\to T^{\st}L$ be the inclusion, then the special Lagrangian equation for $\al^-$ will satisfy $\SL(\al^-)^+=0$ and $Q(\al^-)^+=0$. 
		\end{proposition}
		\begin{proof}
			As $R\circ\iota_{\al^-}=\iota_{\al^-}\circ\sigma$, where $\sigma$ is the involution on $\tL$, we compute
			$$-\iota_{\al^-}^{\st}\Im\Omega=\iota_{\al^-}^{\st}R^{\st}\Im\Omega=\sigma^{\st}\iota_{\al^-}^{\st}\Im\Omega,$$
			which implies $\SL(\al^-)^+=0$. As $\SL(\al^-)^+=(d\al^-)^++Q(\al^-)^+=Q(\al^-)^+$, we obtain $Q(\al^-)^+=0$.
		\end{proof}

		\begin{subsection}{The Inverse of the branched deformations}
			In this subsection, we will consider the inverse of the branched deformations. For a family of graphic special Lagrangian submanifolds with convergence assumptions, we could recover a nondegenerate pairs. Let $(U_L,J,\omega,\Omega)$ be the pull-back Calabi-Yau structure on the Weinstein neighborhood of a special Lagrangian submanifold $[L]$.
			
			\begin{definition}
				$[\tL]=(\tL,\tiota)$ be an immersed special Lagrangian submanifold with $\tiota:\tL\to U_L$. $[\tL]$ is called graphic on $L$ if there exists a pair $\mfa$ such that $\tiota$ is given by the graph of $\mfa.$
			\end{definition}
			
			\begin{lemma}
				\label{lem_inversedeformationfamily}
				Let $B' \subset L\setminus \Sigma$ be an open set with $B$ a proper subset of $L\setminus \Sigma$, suppose the graph of $\mfa$ is a special Lagrangian submanifold and $\|\mfa\|_{\MC^{k+1,\ga}}(B')\leq D_k$ for sufficiently small constants $D_k$, then $\|\mfa\|_{\MC^{k+1,\ga}(B)}\leq C_k\|\mfa\|_{L^2(B')}$.
			\end{lemma}
			\begin{proof}
				The special Lagrangian equation for $\mfa$ could be written as $-(d+d^{\st})\mfa+Q(\mfa)=0$. By the elliptic estimate, we have $$\|\mfa\|_{\MC^{k+1,\ga}(B)}\leq C_k(\|Q(\mfa)\|_{\MC^{k,\ga}(B')}+\|\mfa\|_{L^2(B')}).$$ By Proposition \ref{prop_structureofspecialLagrangian}, we have $\|Q(\mfa)\|_{\MC^{k,\ga}(B)}\leq C_k' \|\mfa\|^2_{\MC^{k+1,\ga}(B')}.$ Therefore, 
				$$(1-C_kC_k'\|\mfa\|_{\MC^{k+1,\ga}(B)})\|\mfa\|_{\MC^{k+1,\ga}(B)}\leq C_k\|\mfa\|_{L^2(B')}.$$ Take $D_k:=\frac{1}{2C_kC_k'}$ and suppose $\|\mfa\|_{\MC^{k+1,\ga}(B)}\leq D_k$, we have $\|\mfa\|_{\MC^{k+1,\ga}(B)}\leq 2C_k\|\mfa\|_{L^2(B')}$.
			\end{proof}
			
			\begin{definition}
				Let $\Sigma_t$ be a family of embedded codimension 2 submanifolds of $L$ diffeomorphic to each other and $\MI_t$ be a flat $\ZT$ bundle over $L\setminus \Sigma_t$ with corresponding representations $\rho_t:\pi_1(L\setminus \Sigma_t)\to \{\pm 1\}$ such that for any small loop $\gamma$ linking $\Sigma_t$, we have $\rho_t(\gamma)=-1.$ Let $\mfa_t=\al^+_t+\al^-_t$ with $\al^+\in\Omega^1,\;\al^-_t\in \Omega^1(L;\MI)$, 
				Suppose 
				\begin{itemize}
					\item [(i)] $\Sigma_t$ convergence to $\Sigma_0$ as smooth codimension 2 submanifolds and after we identified $\rho_t$ with representation of $\pi_1(L\setminus \Sigma_0)\to \{\pm 1\}$, $\rho_t$ convergence to $\rho_0$,
					\item [(ii)] $\lim_{t\to 0}\|\al_t^{\pm}\|_{\MC^{,\gamma}(L)}=0,$ and $\|\al_t^{\pm}\|_{\MC^{,\gamma}(L)}\LS \|\mfa\|_{L^2(L)}$,
					\item [(iii)] for any proper open subset $B\subset L\setminus \Sigma$, we have $\lim_{t\to 0}\|\mfa\|_{\MC^{k,\ga}(B)}=0$,
					\item [(iv)] for a tubular neighborhood of $\Sigma_0$ containing all $\Sigma_t$, let $r_t$ be the distance to $\Sigma_t$, we have $\liminf_{t\to 0} \frac{|r_t^{\frac12}\al_t^-|(p)}{\|\mfa_t\|_{L^2}}\neq 0$ for any $p\in \Sigma_0$.
				\end{itemize}
				then we say $[\tL_t]$ is a graphic branched deformation family of $2[L]$.
			\end{definition}
			
			\begin{theorem}
				Suppose $[\tL_t]$ is a family of graphic special Lagrangian submanifolds for $\mfa_t$, let $\hal_t^{\pm}:=\frac{\al_t^{\pm}}{\|\mfa_t\|_{L^2}}$ and $\hat{\mfa}_t=\hat{\al}^+_t+\hal^-_t$, then passing to subsequence, there exists $\hat{\al}^{\pm}_{0}$ such that 
				\begin{itemize}
					\item [(i)] for each proper open set $B\subset L\setminus \Sigma$, $\hal_t^{\pm}$ convergence to $\hal_{0}^{\pm}$ smoothly when $t\to 0$ and near $\Sigma$, $\lim_{t\to 0}\|\hal_t^{\pm}\|_{\MC^{,\frac12\ga}}=\|\hal_0^{\pm}\|_{\MC^{,\frac12\ga}}$.
					\item [(ii)] $\hal^+_0$ is a harmonic 1-form on $L$ and $\hal^-_0$ is a nondegenerate $\ZT$ harmonic 1-form. 
				\end{itemize}
			\end{theorem}
			\begin{proof}
				
				We write $k_t=\|\mfa_t\|_{L^2}$, then by Lemma \ref{lem_inversedeformationfamily}, for any proper open set $B\subset L\setminus \Sigma$, we have $\|\hal_t^{\pm}\|_{\MC^{k+1,\ga}(B)}^2\LS 1$. As $\hat{\mfa}_t$ satisfies $$(d+d^{\st})\hat{\mfa}_t+k_t^{-1}Q(\hat{\mfa}_t k_t)=0,$$ and $\|k_t^{-1}Q(\hat{\mfa}_t k_t)\|_{\MC^{k,\ga}(B)}\LS k_t\|\hat{\mfa}_t\|_{\MC^{k+1,\ga}(B)},$ we have $(d+d^{\st})\hal_{0}^{\pm}=0$ over $L\setminus\Sigma$, where $\hal_0^{\pm}$ is the limit of $\hal_t^{\pm}$.
				
				In addition, as $\lim_{t\to 0}\|\al_t^{\pm}\|_{\MC^{,\gamma}}=0$ and $\|\al_t^{\pm}\|_{\MC^{,\gamma}(L)}\leq C\|\mfa_t\|_{L^2(L)}$, the limit of $|\hal_t^{\pm}|$ exists over $L$ as a $\MC^{,\frac12\gamma}$ function. Therefore, $\hal_0^+$ is a harmonic 1-form over $L$. As $\hal_{0}^{-}$ is also a $\MI$-valued harmonic 1-form on $L\setminus \Sigma$ and $|\hal_{0}^-|$ has an extension over $\Sigma$. By Theorem \ref{thm_polyhomo}, $\hal_0^-$ has an expansion near $\Sigma$. As $\liminf_{t\to 0} |r_t^{\frac12}\hal_t^-|(p)\neq 0$ for any $p\in \Sigma_0$, $\hal_0^-$ is nondegenerate.
			\end{proof}
		\end{subsection}
	\end{section}
	
	\begin{section}{Construction of the Approximation Solution}
		\label{sec_6}
		In this section, given a nondegenerate harmonic pair $\mathfrak{a}=\al^++\al^-$, for each small $t\in\mathbb{R}$, we will construct a family of approximate special Lagrangian submanifolds submanifolds close to the graph of $t\mathfrak{a}$. We begin with introducing our main result in subsection \ref{subsec_mainresults}. Over subsection \ref{subsec_primaryestimate}, \ref{subsec_variationsingularset}, we will  estimate on the non-linear terms. In subsection \ref{subsection_constructionapproximate}, we will give the explicitly construction of the approximate special Lagrangian submanifolds.
		\subsection{Main results of the approximation constructions}
		\label{subsec_mainresults}
		We first summarize the main results that will be proved in this section. As usual, let $U_L$ be a neighborhood of the zero section on $T^{\st}L$, $(U_L,J,\omega,\Omega)$ be a Calabi-Yau structure and $\iota_0:L\to U_L$ be the inclusion map of the zero section. We write $g_{U_L}$ be the Calabi-Yau metric and $g$ be the induced metric on the zero section.
		
		As in this section, we will use diffeomorphisms on $L$ to change the based metric $g$ on $L$. For different terms in the special Lagrangian equation, we will write $\SL_g,Q_g,P_g,S_g$ will label $"g"$ to emphasize which Calabi-Yau structure we are considering.
		
		Given a nondegenerate pair $\mfa=\al^++\al^-$, $\al^-$ defines a branched covering $p:\tL\to L$. As previously, we first choose a suitable coordinates near $\Sigma$, then using normal exponential map, we identified a neighborhood of $\Sigma$ in $L$ with a neighborhood $U\subset N_{\Sigma}$ contains the zero section. We choose a local coordinate system $(x_1,x_2,\cdots,x_n)$ such that $z=x_1+\sqrt{-1}x_2$ is a complex coordinate on the normal bundle of $\Sigma$ and $(x_3,\cdots,x_n)$ are coordinates on $\Sigma$ such that locally $\Sigma\cap U=\{z=0\}$. We also write $z=re^{i\theta}$ and $r$ measures the distance to the branch locus $\Sigma$. We extends $r$ to $L$, which we still denote as $r$, such that the only zeros of $r$ is the branch locus. 
		
		The following is our main result in this section.
		\begin{theorem}
			\label{thm_approximation}
			Given an nondegenerate harmonic pair $\mfa=\al^++\al^-$ over $L$ and a positive integer $N$, there exists a family of nondegenerate pairs $\mfA_t=\MA^+_t+\MA^-_t$, smooth diffeomorphisms $\phi_t:L\to L$ with $r_t:=\phi_t^{\st}\circ r$ and a positive number $T_N$, such that for all $t<T_N$, there exists a t-independent constant $C$, the following holds:
			\begin{itemize}
				\item [(i)] $\|r_t\SL_g(\phi_t^{\st}\mfA_t)\|_{\MC^{,\gamma}(L)}\leq C t^{N+1},$
				\item [(ii)] $\|\mfA_t-t\mfa\|_{\MC^{,\gamma}(L)}\leq C t^2,$
				\item [(iii)] $\mfA_t$ is closed, $d\mfA_t=0$, and smooth on any open set of $L\setminus \Sigma$. Near $\Sigma$, $\mfA_t\in \MC^{1,\ga}$ is nondegenerate and polyhomogeneous near $\Sigma$ with expansions
				\begin{equation}
					\begin{split}
						&\MA_t^-\sim a_1r^{\frac12}+\sum_{k=1}^{p}a_{2,k} r^{\frac32}(\log r)^k+\cdots,\\
						&\MA_t^+\sim b_0+b_1r+b_2r^2+\sum_{k=1}^pb_{2,k} r^2(\log r)^k+\cdots.
					\end{split}
				\end{equation}
				\item [(iv)] $\phi_t(\Sigma)$ is an embedded submanifold of $L$ and $\lim_{t\to 0}\phi_t=\Id$,
				\item [(v)] we write $\tiota_{t}:\tL\to U_L$ be the inclusion of the graph of nondegenerate pair $\phi_t^{\st}\mfA_t$ and $\theta_t$ be the Lagrangian angle of $\tiota_t$, then $\|\theta_t\|_{\ca(\tL)}\leq C t^{N-1}.$
				\item [(vi)] $\tiota_t(\tL)$ convergence to $2\iota_0(L)$ as current and $\|\tiota_t-\iota_0\circ p\|_{\ca(\tL)}\leq Ct$.
			\end{itemize}
			
			Here the $\ca(\tL)$ norm on (v) and (vi) will the the induced metric $\tiota_t^{\st}g_{U_L}$, later by Theorem \ref{thm_metricconvergence}, the $\ca(\tL)$ norm will be independent of $t$. 
			
		\end{theorem}
		
		\subsection{Primary estimates along the branch locus $\Sigma$}
		\label{subsec_primaryestimate}
		In this subsection, we will give estimate for quadratic terms of a pair $\mfa=\al^++\al^-$. We call $\mfa$ a polyhomogeneous nondegenerate pair if both $\al^-,\al^+$ are polyhomogeneous and $\al^-$ is nondegenerate along $\Sigma$. $\mfa$ is said to be closed if both $\al^{+}$ and $\al^-$ are closed.
		
		\begin{lemma}
			\label{lem_quadratictermfortwovaluedfunction}
			For $\mfa=\al^++\al^-$ is a polyhomogeneous nondegenerate closed pair, then $$\|rQ(t\mfa)\|_{\MC^{,\gamma}}\leq Ct^2,$$ where $C$ depends on $\mfa$ but independent of $t$. In addition, near $\Sigma$, $Q_g(t\mfa)^-$ and $Q_g(t\mfa)^+$ has leading expansions $$Q_g(t\mfa)^-\sim t^2r^{-\frac12},\;Q_g(\mfa)^+\sim t^2r^{-1}.$$
		\end{lemma}
		\begin{proof}
			By Proposition \ref{prop_specialLagrangianlineartermstructure}, we only need to consider the behavior near the branch locus.
			
			As $Q_g(t\mfa)=P_g(t\mfa)+S_g(t\mfa)$, we will give similar estimates on both $P_g(t\mfa)$ and $S_g(t\mfa)$. Based on the choice of coordinates, we have $\pa_{x_i}z^{s}=\MO(z^{s-1})$ for $i=1,2$. Recall that we have $P_g(t\mfa)=\sum_{k=1}^{\lceil \frac{n-1}{2}\rceil}(-1)^kP_{2k+1}(t\mfa)$,
			with 
			$$
			P_{2l+1}(t\mfa)=\st (\iota_{t\mfa}^{\st}K_{2l+1})=\ep_{k_1\cdots k_{2l+1}}\ep_{j_1\cdots j_{2l+1}}H^{k_1}_{j_1}(t\mfa)H^{k_2}_{j_2}(t\mfa)\cdots H^{k_{2l+1}}_{j_{2l+1}}(t\mfa),
			$$
			where $H^i_j(t\mfa)=H^i_j(t\al^+)+H^i_j(t\al^-)$, with $H^i_j(t\al^{\pm})=t\sum_kg^{ik}\na\al^{\pm}(\pa_{x_j},\pa_{x_k})$. 
			
			Therefore, $H^i_j(t\al^-)$ have the following asymptotic behaviors near the branch locus: when $i,j\in \{1,2\}$, $H^i_j(t\al^-)\sim tr^{-\frac12}$, when $i\in \{1,2\}$, $j\in \{3,\cdots,n\}$, $H^i_j(t\al^-)\sim tr^{\frac12}$ and when $i,j\in \{3,\cdots,n\}$, $H_j^i(t\al^-)\sim tr^{\frac32}$.
			
			The most singular terms of $P(t\mfa)^-$ will be $H^1_1(t\al^-)H^2_2(t\al^+)H_3^3(t\al^+)$ and other similar combinations. $P(t\mfa)^-$ have the leading asymptotic $P_g(t\al)^-\sim tr^{-\frac12}$ near $\Sigma$. Similarly, the most singular terms of $P(t\mfa)^+$ would be $H^1_1(t\al^-)H^2_2(t\al^-)H_3^3(t\al^+)$ and other similar combinations and we have the worst leading asymptotic $P_g(t\mfa)^+\sim tr^{-1}$.
			
			Now, we will estimate $S(t\mfa)$. By Proposition \ref{prop_structureofspecialLagrangian} and a similar arguement as the "P" terms, we have $S_g(t\mfa)^-\sim t^2r^{-\frac12},\;S_g(t\mfa)^+\sim t^2r^{-1}$, which implies the claim.
		\end{proof}
		\begin{lemma}
			\label{prop_variationofQ}
			For $\mfa,\mfb$ are polyhomogeneous nondegenerate closed pairs, we have 
			$$\|r(Q_g(t\mfa+t^k\mfb)-Q_g(t\mfa))\|_{\MC^{,\gamma}}\leq C t^{k+1},$$
			where $C$ depends on $\mfa,\mfb$ but independent of $t$. 
		\end{lemma}
		\begin{proof}
			By the previous Lemma \ref{lem_quadratictermfortwovaluedfunction}, we have $\|rQ_g(t\mfa+t^k\mfb)\|_{\ca}\leq C$ and $\|rQ_g(t\mfa)\|_{\ca}\leq C$, which means we don't need to concern the singularity coming from the branch locus. By \eqref{eq_nonlineartermP} and Proposition  \ref{prop_structureofspecialLagrangian}, we have $$P_g(\mfa)=\MO(|\na\mfa|^3),\;S_g(\mfa)=\MO(|\mfa|^2+|\na\mfa|),$$ which imply the desire bound with order $t^{k+1}$ 
		\end{proof}
		
		In summary, we obtain the following estimates.
		\begin{proposition}
			\label{prop_specialLagrangiansingularbehavior}
			Let $\mfa=\al^++\al^-$ be a polyhomogeneous nondegenerate pair, then $\|r\SL_g(\mfa)\|_{\ca}\leq C$, where $C$ depends on $\mfa$. In addition, near $\Sigma$, we have the leading expansions,
			$$
			\SL_g(\mfa)^-\sim r^{-\frac12},\;\SL_g(\mfa)^+\sim r^{-1}.
			$$
		\end{proposition}
		\begin{proof}
			As $d^{\st}\mfa=d^{\st}\al^-+d^{\st}\al^+$ and $\al^-\sim r^{\frac12}$, thus $\|rd^{\st}\mfa\|_{\ca}\leq C$.  
			In addition, as $\SL_g(\mfa)=d^{\st}\mfa+Q_g(\mfa)$, by Lemma \ref{lem_quadratictermfortwovaluedfunction}, we obtain the desire estimate.
		\end{proof}
		
		\begin{comment}
		\begin{lemma}
		\label{lem_correction_on_r}
		In the local coordinate $(z=re^{i\theta},t)$, for any function in the form $r^{-\frac{1}{2}}b(\theta,t)$, there exists a function $f$ which can be written as $f=r^{\frac{3}{2}}a(\theta,t)$ such that
		$$
		|\Delta f-r^{-\frac12}b|_{\ca}\leq C
		$$
		\end{lemma}
		\begin{proof}
		In this local coordinate, we could write $\Delta=\pa_r^2+\frac{1}{r}\pa_r+\frac{1}{r^2}\pa_{\theta}^2+\Delta_t$, then 
		$$\Delta f=\frac94 r^{-\frac12}a(\theta,t)+r^{-\frac{1}{2}}\pa_{\theta}^2a(\theta,t)+r^{\frac32}\pa_t^2a(\theta,t)+e.$$ Let $a$ be the periodic solution to the ODE $\pa^2_{\theta}a(\theta,t)+\frac94 a(\theta,t)=b(\theta,t)$ with initial value $a(\theta,0)=\pa_{\theta}a(\theta,0)=0$, then $a(\theta,t)$ will be have the same regularity as $b$.
		
		In addition, we have $$\Delta f-r^{-\frac12}b=r^{\frac{3}{2}}\pa_t^2a(\theta,t)+e,$$ which implies the desire estimate.
		\end{proof} 
		
		\end{comment}
		
		\begin{subsection}{Variation of the singular set}
			\label{subsec_variationsingularset}
			Let $(L,\iota_0)$ be the zero section of $(U_L,J,\omega,\Omega)$ with induced Riemannian metric $g$. Let $\Sigma$ be the codimension 2 submanifold of $L$ with normal bundle $N_{\Sigma}$ and let $v\in \Gamma(N_{\Sigma})$ be a section.  We choose an extension of $v$ to a vector field over $L$, which we still denote by $v$ to save notation.
			
			The vector field $v$ generates 1-parameter family of diffeomorphisms $$\lam_s:L\to L,\;\mathrm{with}\;\frac{d}{ds}\lam_s=v\circ \lam_s,\; \lam_0=\Id,$$ and we write $\lam:=\lam_s|_{s=1}$. $\lam_s$ could also induced a family of diffeomorphisms on $T^{\st}L$: let $(x,y)\in T^{\st}L$ with $x\in L$ and $y\in T^{\st}_xL$, we define 
			\begin{equation}
				\begin{split}
					\hmu_s:T^{\st}L\to T^{\st}L,\;\hmu_s(x,y):=(\lam_s(x),\lam_{-s}^{\st}(y)),
				\end{split}
			\end{equation}
			then under the identification $\Gamma:T_x^{\st}L\oplus T_xL\to T_{(x,y)}(T^{\st}L)$, we have 
			\begin{equation}
				\label{eq_vectorfiedinducedoverwholemanifold}
				\begin{split}
					\hV=\frac{d}{ds}|_{s=0}\hmu_s(x,y)=\Gamma(v,-L_vy)\in T_xL\oplus T_x^{\st}L,
				\end{split}
			\end{equation}
			here we consider $y$ as a 1-form and $L_v$ is taking the Lie derivative for $y$.
			
			Let $U_L'$ be a proper subset of $U_L$, $\chi$ be a cut-off function such that $\chi|_{U_L'}=1$ and $\chi|_{U_L^{\mathsf{c}}}=0$. We define $V=\chi \hV$ and write $\mu_s:U_L\to U_L$ be the 1-parameter family of diffeomorphisms such that $\frac{d\mu_s}{ds}=V\circ \mu_s$ with $\mu:=\mu_s|_{s=1}.$
			
			We would like to consider the special Lagrangian equation over the Calabi-Yau structure $(U_L,\mu^{\st}J,\mu^{\st}\omega,\mu^{\st}\Omega)$ for a nondegenerate pair $\mfa:=\al^++\al^-\subset U_L'$, which is defined to be
			\begin{equation}
				\label{eq_variationosspeicalLagrangianequation}
				\SL_{\lam^{\st}g}(\mfa):=\st_{\lam^{\st}g}\iota_{\mfa}^{\st}(\mu^{\st}\Im\Omega).
			\end{equation}
			
			Similarly, we could define $$S_{\lam^{\st}g}(\mfa):=\st_{\lam^{\st}g}\iota_{\mfa}^{\st}(\mu^{\st}T),\;P_{\lam^{\st}g}(\mfa)=P(\na_{\lam^{\st}g} \mfa),\; Q_{\lam^{\st}g}(\mfa)=S_{\lam^{\st}g}(\mfa)+P_{\lam^{\st}g}(\mfa),$$
			where $\na^{\lam^{\st}g}$ is the Levi-Civita connection for the metric $\lam^{\st}g$.
			\begin{proposition}
				\label{prop_pullbackmetricrelationship}
				For $\mfa=\al^++\al^-\subset U_L'$, we have
				\begin{equation}
					\begin{split}
						\lam^{\st}\SL_g(\mfa)=\SL_{\lam^{\st}g}(\lam^{\st}\mfa),\;\lam^{\st}P_g(\mfa)=P_{\lam^{\st}g}(\lam^{\st}\mfa),\;\lam^{\st}S_g(\mfa)=S_{\lam^{\st}g}(\lam^{\st}\mfa).
					\end{split}
				\end{equation}
				In addition, we could write $$\SL_{\lam^{\st}g}(\mfa)=-d^{\st_{\lam^{\st}g}}\mfa+Q_{\lam^{\st}g}(\mfa).$$
			\end{proposition}
			\begin{proof}
				By the definition of $\lam$ and $\mu$, we obtain $\iota_{\mfa}\circ\lam=\mu\circ\iota_{\lam^{\st} \mfa}$. We compute 
				$$
				\lam^{\st}\SL_g(\mfa)=\lam^{\st}(\st_g\iota_{\mfa}^{\st}\Im\Omega)=\st_{\lam^{\st}g}(\iota_{\mfa}\circ\lam)^{\st}\Im\Omega=\st_{\lam^{\st}g}(\mu\circ \iota_{\lam^{\st}\mfa})^{\st}\Im\Omega=\SL_{\lam^{\st}g}(\lam^{\st}\mfa).
				$$
				In addition, we compute
				\begin{equation*}
					\begin{split}
						\SL_{\lam^{\st}g}(\mfa)=\lam^{\st}\SL_g((\lam^{-1})^{\st}\mfa)&=-\lam^{\st}d^{\st_g}(\lam^{-1})^{\st}\mfa+\lam^{\st}Q_g((\lam^{-1})^{\st}\mfa)=-d^{\st_{\lam^{\st}g}}\mfa+Q_{\lam^{\st}g}(\mfa).
					\end{split}
				\end{equation*}
			\end{proof}
			
			Now, we will estimate the differences between terms in the special Lagrangian equation under the variation. 
			\begin{lemma}
				\label{lem_lemvariationmetric}
				\begin{itemize}
					\item [(i)] We write $g_{s}=\lam_s^{\st}g$, then the following estimates holds, $$\|\frac{d}{ds}g_s\|_{\MC^{,\gamma}}\LS \|v\|_{\MC^{1,\ga}},\;\|\frac{d^2}{ds^2}g_s\|_{\MC^{,\gamma}}\LS \|v\|_{\MC^{2,\ga}},\; \|\frac{d}{ds}\Gamma_{ij}^k\|_{\MC^{,\gamma}}\LS \|v\|_{\MC^{2,\ga}},$$
					where $\Gamma_{ij}^k$ is the Christoffel symbols w.r.t metric $g_s$ and any smooth coordinates.
					\item [(ii)] Let $\alpha$ be a n-form on $L$, then 
					\begin{equation}
						\begin{split}
							\frac{d}{ds}\st_{g_s}\al=-\frac12\Tr(g_s^{-1}\pa_sg_s)\st_s\al.
						\end{split}
					\end{equation}
				\end{itemize}
			\end{lemma}
			\begin{proof}
				(i) follows directly from the definition of Lie derivative on tensors. For (ii), let $(x_1,\dots,x_n)$ be a coordinates on $L$, $\st_s$ be the Hodge star operator for the metric $g_s$, then we write $\al=fdx_1\we\cdots \we dx_n$ and $\st_s\al=|g_s|^{-\frac12}f$, where $|g_s|$ is the determinant of the metric. 
				
				We compute
				\begin{equation*}
					\begin{split}
						\frac{d}{ds}\st_s\al=\frac{d}{ds}f|g_s|^{-\frac12}=-\frac12|g_s|^{-\frac32}\pa_s|g_s|f=-\frac12|g_s|^{-1}\pa_s|g_s|\st_s\al=-\frac12\Tr(g_s^{-1}\pa_sg_s)\st_s\al,
					\end{split}
				\end{equation*}
				where we use $\pa_s|g_s|=|g_s|\Tr(g_s^{-1}\pa_sg_s)$ for the last identity.
			\end{proof}

			\begin{lemma}
				\label{lem_laplacianestimate}
				Over the Riemannian manifold $(L,g)$, for $f$ is either a function or a section of $\MI$, we have the following expressions:
				\begin{equation}
					\label{eq_lem_laplacianestimate}
					\Delta_{\lam^{\st}g}f=\Delta_gf+\na_v(\Delta_g f)-\Delta_g (\na_v f)+e_{\Delta}(f,\lam,g),
				\end{equation}
				with
				$$\|r^{\frac12}e_{\Delta}(f,\lam,g)\|_{\MC^{,\gamma}}\LS \|r^{\frac12}\na\na f\|_{\MC^{,\gamma}}\|v\|_{\MC^{2,\ga}}.$$ 
			\end{lemma}
			\begin{proof}
				For $x\in L$, we define $F(s,x):=\Delta_{\lam_s^{\st}g}f$, then the Taylor expansion in terms of $s$ would be
				
				$$\Delta_{\lam_1^{\st}g}f=\Delta_gf+\frac{d}{ds}|_{s=0}\Delta_{\lam_s^{\st}g}f s+e_{\Delta},$$
				where $e_{\Delta}$ is the integral reminder term which could be written as $$e_{\Delta}=\int_0^1\pa_s^2F(s,x)(1-s)ds=\int_0^1\pa_s^2\Delta_{\lam^{\st}g}f ds.$$
				
				As $\lam_s^{\st}(\Delta_g f)=\Delta_{\lam_s^{\st}g}\lam_s^{\st}f$, we take derivative of $s$ in both side at $s=0$, we obtain 
				$$(\frac{d}{ds}|_{s=0}\Delta_{\lam_s^{\st}g})f=\na_v(\Delta f)-\Delta(\na_v f).$$
				
				Let $(x_1,\cdots,x_n)$ be local coordinate on $L$ and now we compute $\pa_s^2F(s,x)$ in this coordinate. We write $g=\sum_{i,j=1}^ng_{ij}dx_i\otimes dx_j$, then the Laplacian operator in this coordinate could be written as $$\Delta_gf=\sum_{i,j=1}^n-|g|^{-\frac12}\pa_i(|g|^{\frac12}g^{ij}\pa_jf)=\sum_{i,j=1}^n-\pa_i(g^{ij}\pa_jf)+\frac12 \pa_i\log|g| g^{ij}\pa_jf,$$
				where $|g|=\det(g_{ij})$. We compute 
				\begin{equation}
					\begin{split}
						\pa_s^2F(s,x)=&\frac{d^2}{ds^2}\Delta_{\lam_s^{\st}g}f\\
						=&\sum_{i,j=1}^n-\frac{d^2}{ds^2}(\lam_s^{\st}g)^{ij}\pa_i\pa_jf-\frac{d^2}{ds^2}\pa_i(\lam_s^{\st}g^{ij})\pa_j f-\frac{1}{2}\pa_i\frac{d^2}{ds^2}(\log |\lam_s^{\st}g| (\lam_s^{\st}g)^{ij})\pa_j f.
					\end{split}
				\end{equation}
				A straight forward computation shows that 
				$$\|\frac{d}{ds}\lam_s^{\st}g\|_{\MC^{,\gamma}}=\|\lam_s^{\st}(L_vg)\|_{\MC^{,\gamma}}\LS \|v\|_{\MC^{1,\ga}},\;
				\|\frac{d^2}{ds^2}(\lam_s^{\st}g)\|_{\MC^{,\gamma}}=\|\lam_s^{\st}L_v^2g\|_{\MC^{,\gamma}}\LS \|v\|_{\MC^{2,\ga}}.
				$$
				
				As for each $s$, $\|r^{\frac12}\pa_s^2F(s,x)\|_{\MC^{,\gamma}}\LS \|r^{\frac12}\na\na f\|_{\MC^{,\gamma}}\|v\|_{\MC^{2,\ga}}$, we obtain the desire estimate for $r^{\frac12}e_{\Delta}$.
			\end{proof}
			
			\begin{lemma}
				\label{lem_app_estimate_forall}
				Let $\mfa$ be a polyhomogeneous nondegenerate pair, then the following estimates hold
				\begin{itemize}
					\item [(i)]$$\|r(P_{\lam^{\st}g}(t\mfa)-P_g(t\mfa))\|_{\MC^{,\gamma}}\leq C t^3\|v\|_{\MC^{2,\ga}},$$
					\item [(ii)] $$\|r(S_{\lam^{\st}g}(t\mfa)-S_g(t\mfa))\|_{\MC^{,\gamma}}\leq C t^2\|v\|_{\MC^{2,\ga}},$$
					\item [(iii)] $$\|r(Q_{\lam^{\st}g}(t\mfa)-Q_{g}(t\mfa))\|_{\MC^{,\gamma}}\leq C t^2\|v\|_{\MC^{2,\ga}},$$
				\end{itemize}
				where $C$ is a constant depends on $\mfa$ but independents of $v$ and $t$.
			\end{lemma}
			\begin{proof}
				For (i), we write $g_s=\lam_s^{\st}g$, then $P_{\lam^{\st}g}(t\mfa)-P_g(t\mfa)=\int_{0}^1\pa_sP_{g_s}(t\mfa)ds$. By Lemma \ref{lem_lemvariationmetric}, we have $$\|\pa_sg^{ik}\|_{\MC^{,\gamma}}\leq C \|v\|_{\MC^{1,\ga}},\;\;\|\pa_s\Gamma_{ij}^k\|_{\MC^{,\gamma}}\leq C \|v\|_{\MC^{2,\ga}}.$$ By the expression of $\pa_{s}P_{g_s}(t\mfa)$ and the same argument of Lemma \ref{prop_variationofQ}, near $\Sigma$, we obtain $r\pa_sP_{g_s}(t\mfa)\sim t^3r^{\frac12},$ which implies $\|r\pa_sP_{g_s}(t\mfa)\|_{\MC^{,\gamma}}\leq C\|v\|_{\MC^{2,\ga}}t^3.$
				
				For (ii), as $\iota_{t\mfa}\circ\lam_s=\mu_s\circ\iota_{t\lam_s^{\st}\mfa}$, we compute
				\begin{equation}
					\begin{split}
						S_{\lam^{\st}g}(t\al)-S_g(t\al)=\int_0^1\frac{d}{ds}\st_{g_s}\iota_{t\mfa}^{\st}(\mu_s^{\st}T)ds=\int_0^1\frac{d}{ds}\lam_s^{\st}\st_g(\iota_{t\lam_{-s}^{\st}\mfa}^{\st}T)ds.
					\end{split}
				\end{equation}
				As $\st_g(\iota_{t\lam_{-s}^{\st}\mfa}^{\st}T)=S_g(t\lam_{-s}^{\st}\mfa)$, by Lemma \ref{lem_quadratictermfortwovaluedfunction}, the leading expansions of $S_g(t\lam_{-s}^{\st}\mfa)$ will be 
				$$
				S_g(t\lam_{-s}^{\st}\mfa)^-\sim \lam_{-s}^{\st}r^{-\frac12},\;S_g(t\lam_{-s}^{\st}\mfa)^+\sim \lam_{-s}^{\st}r^{-1}.
				$$
				As $\lam_s^{\st}\lam_{-s}^{\st}r=r$, which is independent of the order of $r$ direction. Therefore, the singular terms of $\frac{d}{ds}\lam_s^{\st}\st_g(\iota_{t\lam_{-s}^{\st}\mfa}^{\st}T)$ will have order $r^{-1}$ or $r^{-\frac12}$ near $\Sigma$, which implies $r\frac{d}{ds}\lam_s^{\st}\st_g(\iota_{t\lam_{-s}^{\st}\mfa}^{\st}T)$ is bounded near $\Sigma$. 
				
				On the other hand, we compute
				\begin{equation}
					\begin{split}
						\frac{d}{ds}\lam_s^{\st}\st_g(\iota_{t\lam_{-s}^{\st}\mfa}^{\st}T)=&\lam_s^{\st}(\na_v\st_g(\iota_{t\lam_{-s}^{\st}\mfa}^{\st}T))+\lam_s^{\st}(\st_g \iota_{t \lam_{-s}^{\st}(-L_v \mfa)}^{\st}T)\\
						=&\lam_s^{\st}\na_vS_g(t\lam_{-s}^{\st}\mfa)+\lam_s^{\st}S_g(t\lam_{-s}^{\st}(-L_v\mfa)).
					\end{split}
				\end{equation}
				
				Near $\Sigma$, the most singular terms of $\lam_s^{\st}\na_vS_g(t\lam_{-s}^{\st}\mfa)$ and $\lam_s^{\st}S_g(t\lam_{-s}^{\st}(-L_v\mfa))$ will be order $r^{-2}$ or $r^{-\frac32}$. However, as $r\frac{d}{ds}\lam_s^{\st}\st_g(\iota_{t\lam_{-s}^{\st}\mfa}^{\st}T)$ is bounded, the singular terms at order $r^{-2}$ or $r^{-\frac32}$ will cancel with each other.
				
				By Proposition \ref{prop_structureofspecialLagrangian}, we could schematically write $$\lam_s^{\st}\na_vS_g(t\lam_{-s}^{\st}\mfa)+\lam_s^{\st}S_g(t\lam_{-s}^{\st}(L_v\mfa))=t^2v\otimes A+t^2\na v\otimes B,$$
				where $A,B$ are $\ca$ terms depends on $v$ and $\mfa$. We obtain $\|r\frac{d}{ds}\lam_s^{\st}\st_g(\iota_{t\lam_{-s}^{\st}\mfa}^{\st}T)\|_{\ca}\leq Ct^2\|v\|_{\ca}$ which implies
				$$
				\|S_{\lam^{\st}g}(t\mfa)-S_g(t\mfa)\|_{\MC^{,\gamma}}\leq Ct^2\|v\|_{\MC^{2,\ga}}.
				$$
				(iii) follows by (i) and (ii). 
			\end{proof}
			
			We consider $\mfa=d\mff=df^++df^-,\mfb=d\mfh=dh^++dh^-$ be nondegenerate pairs, where $f^-,h^-$ are sections of $V^-$ and $f^+,h^+$ are sections of a suitable affine line bundle such that $df^+,dh^+$ are 1-forms on $L$. We have the following estimates.
			\begin{lemma}
				\label{lem_structureofSLequationsundervariation}
			Let $v$ be vector on $L$ which generates the 1-parameter family of diffeomorphisms $\lam_s$ with $\lam:=\lam_s|_{s=1}$, then we could write
				\begin{equation}
					\label{eq_structureSLequationsfortwoterms}
					\begin{split}
						\SL_{\lam^{\st}g}(\mfa+\mfb)=\SL_g(\mfa)-(\Delta_g\mfh+\na_v\Delta_g\mff-\Delta_g\na_v\mff)+\EDQ+\EQG+E_{\Delta},
					\end{split}
				\end{equation}
				where \begin{equation}
					\begin{split}
						\EDQ&:=Q_g(\mfa+\mfb)-Q_g(\mfa),\;\EQG:=Q_{\lam^{\st}g}(\mfa+\mfb)-Q_{g}(\mfa+\mfb),\\
						E_{\Delta}&:=e_{\Delta}(\mff+\mfh,\lam,g)+\na_v\Delta_g\mfh-\Delta_g\na_v\mfh,
					\end{split}
				\end{equation}
				where $e_{\Delta}(\mfa+\mfb,\lam,g)$ is the error term defined in \eqref{eq_lem_laplacianestimate}.
			\end{lemma}
			\begin{proof}
				We write $\MF:=\mff+\mfh$ and we write $\SL_{\lam^{\st}g}(\mfa+\mfb)=\Delta_{\lam^{\st}g}\MF+Q_{\lam^{\st}g}(\mfa+\mfb)$.
				
				In addition, we compute
				\begin{equation}
					\begin{split}
						\Delta_{\lam^{\st}g}\MF&=\Delta_g\MF+\na_v\Delta_g\MF-\Delta\na_v\MF+E_{\Delta}(\MF,\lam^{\st}g)\\
						&=\Delta_g\mff+(\Delta_g\mfh+\na_v\Delta \mff-\Delta \na_v \mff)+E_{\Delta},
					\end{split}
				\end{equation}
				where $E_{\Delta}=\na_v\Delta_g\mfh-\Delta_g\na_v\mfh+e_{\Delta}(\MF,\lam,g)$.
				We compute
				\begin{equation}
					\begin{split}
						\SL_{\lam^{\st}g}(\mfa+\mfb)=\SL_g(\mfa)-(\Delta_g\mfh+\na_v\Delta_g \mff-\Delta_g\na_v\mff)+E_{\Delta}+Q_{\lam^{\st}g}(\mfa+\mfb)-Q_g(\mfa).
					\end{split}
				\end{equation}
				As $Q_{\lam^{\st}g}(\mfa+\mfb)-Q_{g}(\mfa)=E_{\delta Q}+E_{Q_{\delta g}}$, we obtain \eqref{eq_structureSLequationsfortwoterms}.
			\end{proof}
			
			\begin{proposition}
				\label{prop_solvethetwovaluedequation}
				\begin{itemize}
					\item [(i)]Let $\rho_0^+\in\MC^{2,\ga}$, $\rho_1^+\in \MC^{,\gamma}$ be polyhomogeneous along $\Sigma$ with index set $\mathbb{Z}$, suppose $\int_{L}(r^{-1}\rho_0^++\rho_1^+)d\Vol=0$, then there exists a polyhomogeneous solution $f^+\in\MC^{2,\ga}$ to $\Delta f^+=r^{-1}\rho_0^++\rho_1^+$ with expansion
					$$
					f^+\sim b_0+b_1r+b_2r^2+\sum_{k=0}^p b_{3,k}r^3(\log r)^k+\cdots.
					$$
					\item [(ii)]Let $\rho_0^-,\rho_1^-\in\MC^{,\gamma}$ be sections of $\MI$ over $L$ which is polyhomogeneous along $\Sigma$ with index set $\{\mathbb{Z}+\frac12\}$, then there exists a polyhomogeneous solution $f^-\in\Gamma(\MI)$ to $\Delta f^-=r^{-\frac12}\rho_0^-+\rho_1^-$ with expansion
					$$
					f^-\sim a_1r^{\frac12}+a_2r^{\frac32}+\sum_{k=0}^{p}a_{3,k} r^{\frac52}(\log r)^k+\cdots.\\
					$$
				\end{itemize}
			\end{proposition}
			\begin{proof}
				For (i), note that $r^{-1}\rho_0^+$ is not a $L^2$ function, we could not directly apply the classical theory. However, we could choose $\mu=r\rho_0^+$, then we compute $\Delta\mu=r^{-1}\rho_0^++e,$ with $e\in\ca$. The equation we would like to solve becomes $\Delta (f^+-\mu)=\rho_1^+-e$ with $\int (\rho_1^+-e)d\Vol=0.$ In addition, as $\rho_1^+-e\in \ca$ and by Theorem \ref{thm_polyhomo}, we could find polyhomogeneous solution with $f^+\in\MC^{2,\ga}$.
				
				For (ii), we define $\mu^-=r^{\frac{3}{2}}\rho_0^-$, then we could write $\Delta \mu^-=r^{-\frac12}\rho^-_0+e^-$ with $e^-\in \MC^{,\gamma}$. For the equation $\Delta (f^--\mu^-)=\rho_1^--e^-$, by Theorem \ref{thm_polyhomo}, we could find a polyhomogeneous solution with the desire leading expansion.
			\end{proof}
		\end{subsection}

		\begin{subsection}{Constructing approximate solution}
			\label{subsection_constructionapproximate}
			Now we will start the construction of Theorem \ref{thm_approximation} by induction. Given a nondegenerate harmonic pair $\mfa_1=\al_1^++\al_1^-$, for each integer $k$ and $t\geq 0$, we will construct the following data:
			\begin{itemize}
				\item [(i)] An nondegenerate closed pair $\mfA_k$ with expression $\mfA_k=\sum_{i=1}^k\mfa_it^{i}$, with $$\mfa_i=\al_i^++\al_i^-=df_i^++df_i^-,$$ where $f_i^+$ is a function on $L$ for $i\geq 2$ and $f_i^-$ is a section of $V^-$. We also write $\mff_i=f_i^++f_i^-$ and $\mfF_k=\sum_{i=1}^k\mff_i$.
				
				\item [(ii)] A diffeomorphism $\vp_{k,t}$ and a positive number $T_k$ such that for $t<T_k$, $\vp_{k,t}^{-1}(\Sigma)$ is an embedded codimension 2 submanifold and $\|r\SL_{g_{k,t}}(\mfA_{k})\|_{\MC^{,\gamma}}\LS t^{k+1}$, where $\SL_{g_{k,t}}(\mfA_{k})$ is the special Lagrangian equation for the variation of Calabi-Yau structure $\vp_{k,t}^{\st}(U_L,J,\omega,\Omega)$ defined in \eqref{eq_variationosspeicalLagrangianequation} with Calabi-Yau metric $\vp_{k,t}^{\st}g_{U_L}$ and $g_{k,t}:=\vp_{k,t}^{\st}g$ is the induced Riemannian metric on the zero section.
			\end{itemize}
		
			As $\al_1^+$ is a harmonic 1-form, we could also write $\al_1^+=df_1^+$, where $f_1^+$ is a section of a affine line bundle. Even $f_1^+$ and $f_2^+$ are sections of different bundles, to save notation, we still write $f_1^++f_2^+$ as our result only depends on $df_1^++df_2^+$, which make sense as a 1-form on $L$.
			
			We break this construction into several steps:
			
			$\bullet\;\mathbf{Step\;1}:$ We set $\phi_1=\Id$ and prove $\|r\SL_g(\mfA_1)\|_{\ca}\LS t^2.$
			
			Let $\mfA_1:=t\mfa_1$, then we compute 
			$$
			\SL_g(\mfA_1)=td^{\st_g}\mfa_1+Q_g(t\mfa_1)=Q_g(t\mfa_1).
			$$
			By Lemma \ref{lem_quadratictermfortwovaluedfunction}, $Q_g(t\mfa_1)$ is polyhomogeneous with leading expansion 
			$$
			Q_g(t\mfa_1)^-\sim t^2r^{-\frac12},\;Q_g(t\mfa_1)^+\sim t^2r^{-1}.
			$$
			Therefore, $\|r\SL_g(\mfA_1)\|_{\ca}\LS t^2.$
			
			$\bullet\;\mathbf{Step\;2}:$  
			Constructing nondegenerate multivalued form $\mfA_2=t\mfa_1+t^2\mfa_2$, the diffeomorphism $\vp_{2,t}$ with $\|r\SL_{g_{2,t}}(\mfA_{2})\|_{\ca}\LS t^{3}$.
			
			$\bullet\;\mathbf{Step\;2.1}:$ we construct vector field $v_1$, pair $\mfA_2=\mfA_1+t^2\mfa_2$ and a positive number $T_2$ such that for $t<T_2$, $\mfA_2$ is nondegenerate and $\vp_{2,t}^{-1}(\Sigma)$ is an embedded codimension 2 submanifold.
			
			By \textbf{Step 1}, we have constructed $\mfA_1$ such that $\|r\SL_g(\mfA_1)\|_{\MC^{,\gamma}}\LS t^2$. We define $$\mathbf{\rho}_2:=\lim_{t\to 0} t^{-2}\SL_g(\mfA_1),$$ which is the $t^2$ order term in the $t$ expansion. By \textbf{Step 1}, $\rho_2=\rho_2^++\rho_2^-$ has leading expansions
			$\rho_2^-\sim r^{-\frac12},\;\rho_2^+\sim r^{-1}.$ In addition, we could write
			\begin{equation}
				\label{eq_errorfromSL}
				\SL_g(\mfA_1)=\rho_2 t^2+E_{\SL}.
			\end{equation}
			
			We would like to solve the following equations for $v_1$, $f_2^{\pm}$ and we set $\mfa_2:=(df_2^+,df_2^-)$.
			\begin{equation}
				\begin{split}
					\rho_2^--\Delta_g( f_2^--\na_{v_1}f_1^-)=0,\;\rho_2^+-\Delta_g (f_2^+-\na_{v_1}f_1^+)=0.
				\end{split}
			\end{equation}
			
			To begin with, we will find $v_1$ and $f_2^-$ to solve the first equation. By Theorem \ref{thm_polyhomo}, we could a solution $P_2$ to solve the equation $-\Delta_g P_2+\rho_2^-=0$, with $$P_2=\Re(a\ze^{\frac12}+b\ze^{\frac32})+E_{P_2}.$$
			
			Recall that $f_1^-$ is a nondegenerate 2-valued harmonic function with expansion $$f_1^-=\Re(B_1z^{\frac32})+E_{f_1},$$ and $B_1$ is no-where vanishing along $\Sigma$. Therefore $\frac{2a}{3B_1}$ is a well-defined section of $N_{\Sigma}$. We take $v_1$ be an extension of $-\frac{2a}{3B_1}$ over $L$ and compute $$\na_{v_1}f_1^-=-\Re(az^{\frac12})+\na_{v_1}E_{f_1}.$$

			We define $f_2^-:=P_2+\na_{v_1}f_1^-$ and compute
			\begin{equation}
				\begin{split}
					f_2^-=&\Re(az^{12}+b\ze^{\frac32})+E_{P_2}-\Re(az^{12})+\na_{v_1}E_{f_1}\\
					=&\Re(b\ze^{\frac32})+E_{P_2}+\na_{v_1}E_{f_1}.
				\end{split}
			\end{equation}
			Thus, we could write $f_2^-=\Re(B_2\ze^{\frac32})+E_{f_2^-}$ with $|E_{f_2^-}|\leq Cr^{\frac32+\ep}$. Moreover, the $r^{\frac32}$ term of $\mfF_2^-:=tf_1^-+t^2f_2^-$ could be written as $\Re((tB_1+t^2B_2)\ze^{\frac32})$. We take $T_2'=\frac{\min_{\Sigma}|B_1|}{2\max_{\Sigma}|B_2|}$, then for any $t\leq T_2'$, $\mfF_2^-$ will be a nondegenerate $\ZT$ function.
			
			Next we will solve the following equation for a function $f_2^+:$
			\begin{equation}
				\rho_2^+-\Delta_g (f_2^+-\na_{v_1}f_1^+)=0.
				\label{eq_firststepplusinterationequation}
			\end{equation}  By Proposition \ref{prop_integrationvanishes} and the definition of $\rho_2^+$, we have $\int_L\rho_2^+d\Vol_g=0$. In addition, as $\na_{v_1}f_1^+$ is a smooth function and $\int_L\Delta_g(\na_{v_1}f_1^+)d\Vol_g=0$. By Proposition \ref{prop_solvethetwovaluedequation}, we could find a polyhomogeneous solution $f_2^+\in \MC^{2,\ga}$ to $\eqref{eq_firststepplusinterationequation}$ with expansions 
			$$f_2^+\sim f_{2,0}^++f_{2,1}^+r+f_{2,2}^+r^2+\sum_{k\geq 3}\sum_{0\leq p\leq p_k} f_{2,k,p}^+r^k(\log r)^p,$$
			for $p_k$ finite integers. We define $\mff_2=f_2^++f_2^-$ and $\mfa_2=d\mff_2$. 
			
			We write the 1-parameter family of diffeomorphisms as $$\frac{d}{ds}\lam_{1,s}=tv_1\circ\lam_{1,s},\;\mathrm{and}\;\lam_{1}:=\lam_{1,s}|_{s=1}.$$
			In addition, we define $\vp_{2,t}:=\lam_1$. As $\lim_{t\to 0}\vp_{2,t}=\Id$, there exists a positive number $T_2''$ such that for any $t<T_2''$, $\vp_{2,t}^{-1}(\Sigma)$ is an embedded submanifold. 
			
			$\bullet\;\mathbf{Step\;2.2}:$  We will prove for $\|r\SL_{\lam_1^{\st}g}(\mfA_2)\|_{\ca}\LS t^3$.
			
			We define $\mfA_2:=\mfA_1+t^2\mfa_2$ and we will compute the formal power series expansions of the special Lagrangian equation in terms of $t$. By Proposition \ref{lem_structureofSLequationsundervariation}, we could write 
			\begin{equation}
				\label{eq_firstcorrectintermid}
				\begin{split}
					\SL_{\lam_1^{\st}g}(\mfA_2)=&\SL_g(\mfA_1)-t^2(\Delta_g\mff_2-\Delta_g\na_{v_1} \mff_1)-t\na_{v_1}\Delta_g\mff_1+E_{\delta Q}+E_{Q_{\delta g}}+E_{\Delta},\\
					=&t^2(\rho_2-\Delta_g(\mff_2-\na_{v_1} \mff_1))-t\na_{v_1}\Delta_g\mff_1+E_{\SL}+E_{\delta Q}+E_{Q_{\delta g}}+E_{\Delta}.
				\end{split}
			\end{equation}
			where 
			\begin{equation}
				\begin{split}
					&E_{\Delta}=e_{\Delta}(d\mfF_1,\lam_1,g)+t^3(\na_{v_1}\Delta \mff_2-\Delta \na_{v_1}\mff_2),\;\EQG:=Q_{\lam_1^{\st}g}(\mfA_1)-Q_g(\mfA_1),\\
					&\EDQ:=Q_g(\mfA_1)-Q_g(\mfA_0),\;E_{\SL}=\SL_g(\mfA_1)-\rho_2 t^2.
				\end{split}
			\end{equation}
			
			We define $E:=E_{\Delta}+\EDQ+\EQG+E_{\SL}+t\na_{v_1}\Delta_g\mfF_1$, then \eqref{eq_firstcorrectintermid} could be rewritten as
			\begin{equation}
				\begin{split}
					\SL_{\lam_1^{\st}g}(\tmfA_2)=t^2(\rho_2-\Delta_g\na_{v_1}\mff_1+\Delta_g \mff_2)+E.
				\end{split}
			\end{equation}
			
			By Lemma \ref{prop_variationofQ}, \ref{lem_laplacianestimate}, \ref{lem_app_estimate_forall} and Proposition \ref{prop_specialLagrangiansingularbehavior}, we could write $E=t^3E'$ with $\|rE'\|_{\ca}\LS 1$, which implies 
			$$
			\|r\SL_{g_{2,t}}(\mfA_{2})\|_{\ca}\LS t^{3}.
			$$
			
			$\bullet\;\mathbf{Step\;3}:$ We will finish the inductive step. Suppose we have construct the nondegenerate pair $\mfA_k=d\mfF_k$, metric $g_k$ and $T_k$ such that for any $t<T_k$, $\mfA_k$ is nondegenerate and
			$\|r\SL_{g_k}(\mfA_k)\|_{\MC^{,\gamma}}\LS t^{k+1}.$ We will construct $\mfA_{k+1}$, $g_{k+1}$ and $T_{k+1}$ with the desire property.
			
			We could write $$\SL_{g_k}(\mfA_k)=\rho_{k+1}t^{k+1}+E_{\SL}.$$ We will find $\mfa_{k+1}=d\mff_{k+1}=(df_{k+1}^+,df_{k+1}^-)$, a vector field $v_{k+1}$ such that $f_{k+1}^-$ has leading asymptotic $f_{k+1}^-\sim \Re(B_{f_{k+1}^-}z^{\frac32})$ and $\mfa_{k+1}$ solves the following equations:
			\begin{equation}
				\begin{split}
					\rho_{k+1}^--\Delta_{g_{k}}( f_{k+1}^--\na_{v_{k}}f_1^-)=0,\;\rho_{k+1}^+-\Delta_{g_k} (f_{k+1}^+-\na_{v_{k}}f_1^+)=0.
				\end{split}
			\end{equation}
			
			Similarly to $\mathbf{Step\;2.1}$, we first solve $-\Delta_{g_k}P_{k+1}+\rho_{k+1}^-=0$ with $$P_{k+1}=\Re(a_{k+1}z^{\frac12}+b_{k+1}z^{\frac32})+E_{P_{k+1}}.$$ 
			
			We choose $v_{k}$ be an extension of $-\frac{2a_{k+1}}{3B_1}$ over $L$, then $\na_{v_{k}}f_1^-$ has leading expansion $-\Re(a_{k+1}z^{\frac12})$. We define $f_{k+1}^-:=P_{k+1}+\na_{v_k}f_1^-$. By Proposition \ref{prop_integrationvanishes}, \ref{prop_solvethetwovaluedequation}, we could find $f_{k+1}^{+}$ solves the second equation.
			
			Let $\frac{d}{ds}\lam_{k,s}=t^kv_k\circ \lam_{k,s}$ be the 1-parameter family and define 
			\begin{equation}
				\vp_{k+1,t}=\vp_{k,t}\circ \lam_{k,1},
			\end{equation}
			then $g_{k+1,t}:=\vp_{k+1,t}^{\st}g=\lam_{k,1}^{\st}g_{k,t}.$ We define $\mfA_{k+1}:=\mfA_k+t^{k+1}\mfa_{k+1}$, then by Lemma \ref{lem_structureofSLequationsundervariation}, we have 
			\begin{equation}
				\begin{split}
					\SL_{g_{k+1}}(\mfA_{k+1})=&t^{k+1}(\rho_{k+1}-\Delta_g(\mff_{k+1}-\na_{v_k} \mff_1))+E.
				\end{split}
			\end{equation}
			By Lemma \ref{prop_variationofQ}, \ref{lem_laplacianestimate}, \ref{lem_app_estimate_forall}, we could write $E=t^{k+2}E'$ and by Proposition \ref{prop_specialLagrangiansingularbehavior}, $\|rE'\|_{\ca}\LS 1$. Therefore, we obtain 
			$$
			\|r\SL_{g_{k+1}}(\mfA_{k+1})\|_{\ca}\LS t^{k+2}.
			$$
			As $\lim_{t\to 0}\lam_{k,1}=\Id$, there exists $T_{k+1}'$ such that $t<T_{k+1}'$, $\vp_{k+1,t}^{-1}(\Sigma)$ is a embedded submanifold. We write the leading expansion as $$\mfA_{k+1}\sim \Re((tB_{\mfA_k}+t^{k+1}B_{f_{k+1}^{-}})z^{\frac32}),$$ where $B_{f_{k+1}^-}$ is the $r^{\frac32}$ expansion for $f_{k+1}^-$. As $\mfA_{k}$ is nondegenerate, $\min|B_{A_k^-}|>0$. Thus we could define a positive number
			$$
			T_{k+1}:=\min\{T_k,\frac14\frac{\min|B_{\mfA_k}|}{\max|B_{f_{k+1}^-}|},T_{k+1}'\},
			$$
			such that for any $t<T_{k+1}$, $\mfA_{k+1}$ will be nondegenerate. This completes $\mathbf{Step\;3}$.
			
			Now, we will give a proof of Theorem \ref{thm_approximation}.
			
			\textrm{Proof of Theorem \ref{thm_approximation}}: Given positive integer $N$, we define $\phi_t:=\vp_{N,t}^{-1}$ and $\mfA_t:=\mfA_{t,N}$. Recall that $g_{N,t}:=\vp_{N,t}^{\st}g$, then by Proposition \ref{prop_pullbackmetricrelationship}, we have
			$$
			\phi_t^{\st}\SL_{g_{N,t}}(\mfA_t)=\SL_g(\phi_t^{\st}\mfA_t).
			$$
			In addition, as $\phi_t$ is a smooth diffeomorphism with bounded derivative, we have
			$$\|r_t\SL_g(\phi_t^{\st}\mfA_t)\|_{\MC^{,\gamma}}=\|\phi_t^{\st}(r\SL_{g_{N,t}}(\mfA_t))\|_{\MC^{,\gamma}}\LS \|r\SL_{\vp^{\st}g}(\mfA_t)\|_{\MC^{,\gamma}},$$
			where $r_t:=\phi_t^{\st}r.$ Therefore, we obtain (i): $\|r_t\SL_g(\phi_t^{\st}\mfA_t)\|_{\ca}\LS t^{N+1}.$
			
			(ii), (iii) and (iv) follows straight forward by the construction. For (v), as $t\to 0$, $\lam_i$ convergence smoothly to the identity map, when $t\to 0$, we have $\phi_t$ convergence smoothly to the identity map. 
			
			We still need to check (v).  Near $p\in\Sigma$, let $(x_1,x_2,\cdots,x_n)$ be the coordinate system on a neighborhood $U$ of $p\in\Sigma$, which we used in \ref{subsec_primaryestimate} with $z=x_1+\sqrt{-1}x_2$ and $\Sigma\cap U=\{z=0\}$. Over $T^{\star}L|_U$, we write the corresponding fiber coordinate as $(y_1,\cdots, y_n)$, then the Riemannian metric could be written as
			$$
			g_{U_L}|_U=\sum_{i,j=1}^nA_{ij}dx_i\otimes dx_j+B_{ij}dx_i\otimes dy_j+C_{ij}dy_i\otimes dy_j.
			$$
			As $[\tL_t]$ is the graph manifold of $\phi_t^{\st}\mfA$ and WLOG, we could assume $\phi_t=\Id$ and write $\phi_t^{\st}\mfA=d\mfF$, $\mfF_{ij}=\pa_{x_i}\pa_{x_k}\mfF$. Suppose we write $\tg^t$ be the induced metric on $\tL_t$ and $\tg_{ij}^t=\tg^t(\pa_{x_i},\pa_{x_j})$, then  
			$$
			\tg_{ij}^t=A_{ij}+\sum_{p=1}^nB_{ip}\mfF_{pj}+\sum_{p,q=1}^n \mfF_{ip}\mfF_{jq}C_{pq}.
			$$
			As $\mfF_{ij}\sim r^{-\frac12}$ when $i,j\in\{1,2\}$, $\mfF_{ij}\sim \MO(1)$ in other situation, the worst singularity comes from the third terms, which will be order $r^{-1}$. 
			We are in particular interested in the behavior of $\lim_{r\to 0}r\tg_{ij}^t$ and it is straight forward to check that suppose one of $i$ or $j\notin\{1,2\}$, we have $\lim_{r\to 0}r\tg_{ij}^t=0$.
			
			Let $\mff$ be the leading term of $\mfF$, then we could write $\mff=\Re(Bz^{\frac32})$ with $B$ is no where vanishing along $\Sigma$. We write $\mff_{ij}=\pa_{x_i}\pa_{x_j}\mff$, then we compute
			$$\mff_{11}=\Re(Bz^{-\frac12}),\;\mff_{22}=-\Re(Bz^{-\frac12}),\;\mff_{12}=-\Im(Bz^{-\frac12}).$$
			Moreover, for $i,j\in\{0,1\}$, we compute
			$$\lim_{r\to 0}r\tg_{ij}^t=\lim_{r\to 0}t^2\sum_{p,q=1}^2C_{pq}r\mff_{pi}\mff_{qj}.$$
			We define the vector fields $v_i=r^{\frac12}(\mff_{1i}\pa_{y_1}+\mff_{2i}\pa_{y_2}),$ then by the previous computations, for $i\in\{1,2\}$, $v_i$ are nowhere vanishing in this neighborhood. In addition, for $i,j\in \{1,2\}$, we have
			$$
			\lim_{r\to 0}r\tg_{ij}^t=t^2g(v_i,v_j),
			$$
			which is also nowhere vanishing near $\Sigma$. 
			
			Therefore, we could find a positive function $\mu$ such that $$r\det(\tg_{ij}^t)dx_1\we\cdots \we dx_n=\mu dx_1\we\cdots \we dx_n,$$ 
			with a lower bound $\mu\geq t^2c_0$. Let $g_0$ be the Riemannian metric on the zero section, then $$t^2d\Vol_{g_0}=t^2\det(g_0)dx_1\we\cdots \we dx_n\LS r\det(\tg_{ij}^t)dx_1\we\cdots \we dx_n=rd\Vol_{\tg^t}.$$
			
			In addition, we have $\tiota_{t}^{\st}\Im\Omega=\sin\theta_td\Vol_{\tg^t}$ and $\tiota_t^{\st}\Im\Omega=\SL(\phi_t^{\st}\mfA_t)d\Vol_{g_0}$. Therefore, we obtain $$\|\theta_t\|_{\ca}\LS t^{-2}\|r_t\SL(\phi_t^{\st}\mfA_t)\|_{\ca}\LS t^{N-1}$$
			
			For (vi), over a neighborhood $U$ of the zero section $T^{\st}\tL$, as currents, $[\tL_t]$ will be the graph of $p^{\st}\mfA_t$ and $2[L]$ is the zero section in $T^{\st}\tL$. By (iii), $p^{\st}\mfA_t$ have the following asymptotic expansions
			$$
			p^{\st}\mfA_t\sim \ta_0+\ta_1\tilde{r}+\ta_2\tilde{r}^2+\sum_{k=1}^p\ta_{3,k}\tr^3(\log \tr)^k+\cdots,
			$$
			where $\tr^2=p^{\st}r.$ Therefore, for a smooth Riemannian metric $g'$ over $\tL$, $p^{\st}\mfA_t$ is a $\MC^{2,\ga}_{g'}$ 1-form with $\lim_{t\to 0}\|p^{\st}\mfA_t\|_{\MC_{g'}^{2,\ga}}=0.$ Therefore, for any differential form $\omega$ on $U$, we have $\lim_{t\to 0}\int_{\tL}\iota_t^{\st}\omega=\int_{\tL}(\iota_0\circ p)^{\st}\omega$, which implies the current convergence $\tiota_t(\tL)\to 2\iota_0(L)$.
			\qed

			\begin{remark}
				As the construction on Theorem \ref{thm_approximation} depends on a positive integer $N$, we could label the approximate nondegenerate pair as $\mfA_{t,N}$ and $\Sigma_{t,N}$. As $N$ become large, the distance between $\Sigma_{t,N}$ and $\Sigma_{t,N+1}$ will be size with order $t^{N+1}$. Suppose we obtain a uniform bound on $v_k$ during the construction process, we might prove that $\lim_{N\to\infty}\Sigma_{t,N}$ exists and this will be the branch locus of the real solution. However, in terms of proving the existence of solutions, we don't need infinite order correction. 
			\end{remark}

			\begin{comment}
			
			We define $\Sigma_t:=\phi_t(\Sigma)$, write $\iota_{t}:\tL\to U_L$ be the image of the graph of nondegenerate pair $\phi_t^{\st}\mfA_t$ and $\iota_0:L\to U_L$ be the inclusion map of the zero section.  (iii) in Theorem \ref{thm_approximation} guarantee that $[\tL_t]=(\tL,\iota_t)$ is a Lagrangian submanifold. Let $[L_0]=(L,\iota_0)$ be the zero section of $U_L$ and $p:\tL\to L$ be the branched covering map.
			\begin{theorem}
			$[\tL_t]$ has a $\MC^{2,\ga}$ manifold structure and we have the convergence of maps 
			\begin{equation}
			\label{eq_convergencerateofapproximate}
			\lim_{t\to 0}|\iota_t-\iota_0\circ p|_{\MC^{,\gamma}}=0,
			\end{equation}
			where the $\MC^{,\gamma}$ is taking w.r.t. the pullback metric $(\iota_0\circ p)^{\st}g$. 	
			\end{theorem}
			\begin{proof}
			By Theorem \ref{thm_approximation}, on $L$ near $\Sigma$, $\MA_t^-$ has expansion $\MA_t^-\sim ar^{\frac12}+\sum_{k=1}^{p}b r^{\frac32}(\log r)^k+\cdots.$ Then on $\tL$ near $\Sigma$, we could write $p^{\st}\MA_t^-\sim \tilde{a}\tr+\sum_{k=1}^{p}\tilde{b}_k\tr^3(\log \tr)^k+\cdots,$ which is a $\MC^{2,\ga}$ 1-from if we choose any smooth structure on $\tL$. As $[\tL_t]$ is the graph of $p^{\st}\MA_t$ over $\tL$, $[\tL_t]$ has a $\MC^{2,\ga}$ structure. 
			\end{proof}
			\end{comment}
			
		\end{subsection}
		
		\begin{subsection}{The construction over fixed real locus}
			Suppose our initial special Lagrangian submanifold is locally a fixed real locus as in Definition \ref{def_locallyantiholomorphicinvolution}, we could also obtain a symmetry for the approximate solutions we construction in Theorem \ref{thm_approximation}. Let $R_0:T^{\st}L\to T^{\st}L$ be the canonical involution and $(U_L,J,\omega,\Omega)$ be a $R_0$-invariant Calabi-Yau structure. The approximate solutions will also obtain a symmetry.
			
			\begin{theorem}
				\label{thm_approximatefamilyinvolution}
				Suppose $[L]$ is locally a fixed real locus and $\mfa=\al^-\in T^{\st}L\otimes \MI$ be a nondegenerate $\ZT$ harmonic 1-form, then in Theorem \ref{thm_approximation}, we could choose the construction in Theorem \ref{thm_approximation} such that $\mfA_t=\mfA^-_t\in T^{\st}L\otimes \MI$. In addition, let $[\tL_t]=(\tL,\tiota_t)$ be the graph of $\mfA_t$, then $R_0\circ \tiota_t=\tiota_t\circ \sigma$.
			\end{theorem}
			\begin{proof}
				By Proposition \ref{prop_ztsymmetry}, we have $Q(\al^-)^+=0$. Suppose we start with $\mfa=\al^-\in T^{\st}L\otimes \MI$, then in each step of the construction in Subsection \ref{subsection_constructionapproximate}, we could always choose the correction to be $\mfA_i=\al_i^-\in T^{\st}L\otimes \MI.$ The rest follows straight forward.
			\end{proof}
			
			If we write $\theta_t$ be the Lagrangian angle of $[\tL_t]$ and $\tg_t:=\tiota_t^{\st}g$ be the induced metric, then as $R_0\circ \tiota_t=\tiota_t\circ \sigma$, we obtain $$\sigma^{\st}\tg_t=\tg_t,\;\sigma^{\st}e^{i\theta_t}=e^{-i\theta_t}.$$ In particular, over the branch locus $\Sigma$, we have $\sin\theta_t|_{\Sigma}=0$.
			
			By Proposition \ref{prop_immersion}, $[\tL_t]$ is an immersed submanifold and we could choose a Weinstein neighborhood of $[\tL_t]$ and consider the special Lagrangian equation $\SL_{\tg_t}(\al)$ for the graph of 1-form $\al$ on $\tL$.
			
			\begin{corollary}
				\label{cor_invariaitonslequationinvolution}
				$\sigma^{\st}\SL_{\tg_t}(\al)=-\SL_{\tg_t}(-\sigma^{\st}\al)$. 
			\end{corollary}
			\begin{proof}
				We define $\tiota_{\al}:\tL\to U_L$ as $\tiota_{\al}(x):=\exp_{\tiota_t(x)}(J (\tiota_t)_{\st}V_{\al})$, where $V_{\al}$ is the vector field on $\tL$ dual to $\al$ given by the Riemannian metric $\tg_t$. We will first prove that  $R_0\circ \tiota_{\al}(\tx)=\tiota_{-(\sigma^{-1})^{\st}\al}\circ\sigma.$ As $R_0$ is an isomorphism and $R_0\circ \tiota_t=\tiota_t\circ \sigma$, we compute
				\begin{equation}
					\begin{split}
						R_0\circ\tiota_{\al}(x)=&R_0 \exp_{\tiota_t(x)}(J (\tiota_t)_{\st}V_{\al})\\
						=&\exp_{R_0\tiota_t(x)}((R_0)_{\st}J (\tiota_t)_{\st}V_{\al})\\
						=&\exp_{\tiota_t\sigma(x)}(-J(R_0)_{\st}(\tiota_t)_{\st}V_{\al})\\
						=&\exp_{\tiota_t\sigma(x)}(-J(\tiota_t)_{\st}\sigma_{\st}V_{\al})\\
					\end{split}
				\end{equation}
				As $\sigma^{\st}g_t=g_t$, we obtain $-\sigma_{\st}V_{\al}=V_{(-\sigma^{-1})^{\st}\al}$. Therefore, we obtain
				\begin{equation}
					\label{eq_identityinvolution}
					R_0\circ\tiota_{\al}(x)=\exp_{\tiota_{t}\sigma(x)}(J(\tiota_t)_{\st}V_{(-\sigma^{-1})^{\st}\al})=\tiota_{-(\sigma^{-1})^{\st}\al}\circ\sigma.
				\end{equation}
				
				As $\tiota_{\al}^{\st}\Im\Omega=\SL_{\tg_t}(\al)d\Vol_{\tg_t}$, we compute
				$$
				-\SL_{\tg_t}(\al)d\Vol_{\tg_t}=-\tiota_{\al}^{\st}\Im\Omega=\tiota_{\al}^{\st} R_0^{\st}\Im\Omega=\sigma^{\st}\tiota_{-(\sigma^{-1})^{\st}\al}^{\st}\Im\Omega=\sigma^{\st}\SL_{\tg_t}(-(\sigma^{-1})^{\st}\al)d\Vol_{\tg_t},
				$$
				which implies the claim. 
			\end{proof}
			
		\end{subsection}

	\end{section}
	
	\begin{section}{The Geometry of the approximate Special Lagrangian Submanifolds}
		\label{sec_geometryofapproximate}
		In this section, we will study the geometry for the family of graphic special Lagrangian submanifolds constructed in Theorem \ref{thm_approximation}. Let $(U_L,J,\omega,\Omega)$ be the pull-back Calabi-Yau structure on the Weinstein neighborhood with Calabi-Yau metric $g$, and $[\tL_t]=(\tL,\tiota_t)$ be the family of submanifolds constructed in Theorem \ref{thm_approximation}, such that $\tiota_t$ is given by the graph of a nondegenerate pair $\mfA_t=\MA^+_t+\MA^-_t$. We will study the induced Riemannian metric, the second fundamental form and the Schauder estimates on $[\tL_t]$. 
		
		\begin{subsection}{The induced metric on the approximate family}
			In this subsection, we will prove the induced metric over $\tL$ under the pullback of the branched covering map is a cone metric will angle $4\pi$.
			\subsubsection{Cone Metric}
			Let $\mathbb{R}^n=\mathbb{C}\ti \mathbb{R}^{n-2}$ with coordinates $z=re^{i\theta},x_3,\cdots,x_{n}$, consider the model metric 
			\begin{equation}
				\label{eq_conemetricmodel}
				g_{4\pi}=4r^2(dr^2+r^2d\theta^2)+\sum_{i=3}^{n}dx_i^2,
			\end{equation}
			which has cone singularity of angle $4\pi$ along the submanifold $\{z=0\}$. Let $g_1,g_2$ be two metric, we write $g_1\leq g_2$ if for any $v$, we have $g_1(v,v)\leq g_2(v,v)$.
			
			\begin{definition}
				Let $\Sigma\subset L$ be a codimension two embedded submanifold, we call $g$ is a metric with cone angle $4\pi$ along $\Sigma$ if for every $p\in\Sigma$, we could find coordinates $(z,x_3,\cdots, x_{n})$ centered at $p$ with $\Sigma=\{z=0\}$ and constant $C$ such that $C^{-1}g_{4\pi}\leq g\leq Cg_{4\pi}$.
			\end{definition} 
			
			For \eqref{eq_conemetricmodel}, passing to the branched covering, we write $r=\sqrt{\tr}$, then the model metric become $d\tr^2+4\tr^2d\theta^2+\sum_{i=3}^ndx_i^2,$
			which is quasi-isometric to the Euclidean metric. Therefore, when we define the H\"older norm or $W^{1,2}$ norms using the cone metric, it will be the same as defining using the smooth Riemannian metric.

			\begin{proposition}
				\label{prop_coneangle4pi}
				$p^{\st}g_L$ is a cone metric over $\tL$ with cone angle $4\pi$.
			\end{proposition}
			\begin{proof}
				Let $N_{\Sigma}$ be the normal bundle of $\Sigma\subset L$, we identified a neighborhood $U$ of $\Sigma$ in $L$ with a neighborhood of zero section of $N_{\Sigma}$ and write $z$ be the fibre coordinates of $N_{\Sigma}$. We could define another Riemannian metric $g_0$ on $U$ as 
				$$g_0=|dz|^2+g_{\Sigma},$$
				then $g_0$ is also smooth Riemannian metric which satisfies $C^{-1}g_0\leq g_L\leq Cg_0$.
				
				We could write the branched covering map $p:\tL\to L$ in this local coordinate as $$p(\tz,x_3,\cdots,x_n)=(\tz^2,x_3,\cdots,x_n),$$ then 
				\begin{equation}
					p^{\st}g_0=4\tr^2(d\tr^2+\tr^2d\ttheta^2)+p^{\st}g_{\Sigma},
				\end{equation} which is a cone metric along $\tSigma$ with cone angle $4\pi$. As $C^{-1}p^{\st}g_0\leq p^{\st}g_{L}\leq Cp^{\st}g_0$, $p^{\st}g_{L}$ is also a cone metric with cone angle $4\pi$.
			\end{proof}
			
			\subsubsection{Induced metric and convergence}
			Recall that $\tiota_t:\tL\to U_L$ is given by the graph of the nondegenerate pair $\mfA=\MA_t^++\MA_t^-$ and we would like to understand the map. By the toy model in Section \ref{subsubsection_localmodel}, we need to prove $\tiota_t$ is well-behaved near $\Sigma$ using the nondegenerate condition. Moreover, we define $$S_t:=\{x\in\tiota_t(\tL)|\mathrm{there\;exists\;}x_1\neq x_2\in\tL\;\mathrm{such\;that\;}\tiota_t(x_1)=\tiota_t(x_2)\},$$ which is the self-intersection points of $\tiota_t(\tL)$.
			\begin{proposition}
				\label{prop_immersion}
				For the map $\tiota_t:\tL\to U_L$, the following holds:
				\begin{itemize}
					\item [(i)] $\tiota_t$ is a $\MC^{2,\ga}$ immersion and $\tiota_t^{\st}g$ is a $\MC^{1,\ga}$ Riemannian metric.
					\item [(ii)] for each $t$, there exists a neighborhood $\tU_{\Sigma}$ of $\Sigma\subset \tL$ such that
					$\tiota_t|_{\tU_{\Sigma}}:\tU_{\Sigma}\to U_L$ is an embedding.
					\item [(iii)] $S_t=\{x\in L\setminus \Sigma| |\MA_t^-|(x)=0\}$.
				\end{itemize}
			\end{proposition}
			\begin{proof}
				For (i), outside of the branch locus, $\tiota_t$ will locally given by graph of different 1-forms which is an immersion. Therefore, we only need to consider the map near the branch locus. Near the branch locus, we choose coordinates $(\tz,\tx_3,\cdots,\tx_n)$ on $\tL$ and $(z,x_3,\cdots,x_n)$ on $L$ such that the branched covering map could be written as $p(\tz,\tx_3,\cdots,\tx_n)=(\tz^2,x_3,\cdots,x_n)$.
				
				Then under the identification of $\Gamma:T(T^{\st}L)\cong TL\oplus T^{\st}L$, we have 
				\begin{equation}
					\Gamma\circ\tiota_{\st}(\pa_{\tx_i})=(p_{\st}\pa_{\tx_i},\na_{p_{\st}\pa_{\tx_i}}\MA_t^-+\na_{p_{\st}\pa_{\tx_i}}\MA_t^+).
				\end{equation}
				
				As $p_{\st}\pa_{\tz}=2\tz\pa_z$ and $\MA_t^-$ is nondegenerate with expansion $$\MA_t^-\sim \Re(B z^{\frac12}dz)+\cdots,$$
				where $B$ is nowhere vanishing along $\Sigma$.
				We compute 
				$$\na_{p_{\st}\pa_{\tz}}\MA_t^-|_{\Sigma}=2B\tz \na_{\pa_z}(z^{\frac12}dz)|_{\Sigma}=\frac{1}{2}Bdz\neq 0.$$
				In addition, we compute $\na_{p_{\st}\pa_{\tz}}(\MA_t^+)|_{\Sigma}=\tz\na_{\pa_z}\MA_t^+|_{\Sigma}=0.$
				Therefore, $(\tiota_t)_{\st}$ is injective thus an immersion. 
				
				(ii) follows by the definition of the double branched cover. 
				
				For (iii), for $x\in L\setminus \Sigma$, we write $x_1,x_2\in p^{-1}(x)\in \tL$, then $\tiota_t(x_1)=\tiota_t(x_2)$ if and only if $|\MA_t^-|(x)=0$ and this gives all the self-intersection points.
			\end{proof}
			
			\begin{example}
				Let $\tL_t:=\{w^2=t^2z\}\subset \mathbb{C}^2$ be the graph manifold of $td\Re(z^{\frac{3}{2}})$ over the complex $z$ plane $L=\{w=0\}$, we write $z=re^{i\theta}$, then use the coordinate over $L$, the induced metric $\tg_t$ over $L_t$ using the $z$ plane coordinate is $\tg_t=(1+\frac{t^2}{4r})(dr^2+r^2d\theta^2).$
				Let $\tz=\tr e^{i\ttheta}$ be the coordinate on $\tL$ with $z=\tz^2$, then in the $\tz$ coordinates, the metric would be
				$\tg_t=(4\tr^2+t^2)(d\tr^2+\tr^2d\ttheta^2).$
				When $t\to 0$, $\lim_{t\to 0}\tg_t=4\tr^2(d\tr^2+\tr^2d\ttheta^2),$ which is the cone metric with angle $4\pi$.
			\end{example}

			\begin{theorem}
				\label{thm_metricconvergence}
				Let $g'$ be any smooth Riemannian metric on $\tL$ and $g_L$ be the Riemannian metric on the zero section of $U_L\subset T^{\st}L$, then $\lim_{t\to 0}\|\tg_t-p^{\st}g_L\|_{\MC_{g'}^{1,\ga}}=0$, where $\MC_{g'}^{1,\ga}$ is defined using $g'$.
			\end{theorem} 
			\begin{proof}
				We prove the above theorem using local charts. For $U_L\subset T^{\st}L$, let $(x_1,\cdots,x_n)$ be a local coordinate on $L$ and $(y_1,\cdots,y_n)$ be coordinates on fiber. The Riemannian metric on $U_L$ in these coordinates could be written as
				$$
				g_{U_L}=g_{L}+B_{ij}dx_i\otimes dy_j+C_{ij}dy^i\otimes dy^j.
				$$ 
				Then the pullback metric $\tg_t:=\tiota_t^{\st}g$ could be written as
				\begin{equation}
					\begin{split}
						\tg_t=p^{\st}g_L+t(B_{ij}'d\tx_i\otimes d\Theta_j)+t^2C_{ij}d\Theta_i\otimes d\Theta_j,
					\end{split}
				\end{equation}
				where $\Theta_i:=p^{\st}\lan\pa_{x_i},\frac{1}{t}\mfA_t\ran.$ Therefore, we could write $\tg_t=p^{\st}g_L+t\hg_t$ with $$\|\hg_t\|_{\MC^{1,\ga}}\LS \sum_{i,j=1}^n(\|B_{ij}'\|_{\MC_{g'}^{1,\ga}}+\|C_{ij}\|_{\MC_{g'}^{1,\ga}})+\|\frac{1}{t}p^{\st}\MA_t\|_{\MC_{g'}^{2,\ga}}\LS 1.$$
				
				In addition, by Theorem \ref{thm_approximation}, $p^{\st}\mfA_t$ is a $\MC_{g'}^{2,\ga}$ form over $\tL$. Therefore, $\lim_{t\to 0}\|\tg_t-p^{\st}g\|_{\MC^{1,\ga}}=0$, which implies the statement.
			\end{proof}
		\end{subsection}
		
		\subsection{First eigenvalue estimate}
		Let $[\tL_t]=(\tL,\tiota_t)$ be the immersed submanifold constructed in Theorem \ref{thm_approximation} with $\theta_t$ be the Lagrangian angle and let $\tg_t:=\iota_tg$ be the pull-back metric, which by Proposition \ref{prop_immersion} is $\MC^{1,\ga}$. Let $f$ be a function on $\tL$, we write $\MD_{t}f:=d^{\st_{t}}(\cos\theta_t df)$ be the linear term of the special Lagrangian equation, where $d^{\st_t}$ is taken for the Riemannian metric $\tg_t$.  It is straight forward to check that $\MD_{\theta_t}$ is a self-adjoint operator. If we write $\hg_t={(\cos\theta_t)}^{\frac{2}{n-2}}g_t$, then $\MD_{\theta_t}f=(\cos\theta_t)^{\frac{2}{n-2}}\Delta_{\hg_t}f$, where $\Delta_{\tg_t}$ is the Laplace operator for $\hg_t$.
		
		To begin with, we will understand the eigenvalue for the cone metric. Let $g_0:=p^{\st}g_L$ be the pullback cone metric and $\Delta_{g_0}$ be the corresponding cone Laplace operator. We define $$\lam_0:=\inf_{f}\{\frac{\int |d f|_{g_0}^2d\Vol_{0}}{\int|f|^2d\Vol_{0}}|f\in W^{1,2},\int fd\Vol_{0}=0\},$$
		where $d\Vol_{0}$ is the volume form w.r.t. the metric $g_0$. As the cone metric is quasi-isometric to the smooth metric, the $\ca$ norms and $W^{1,2}$ norms using the cone metric will be same as using any smooth metric.
		
		\begin{theorem}
			$\lam_0$ is the first positive eigenvalue for the cone metric Laplace operator $\Delta_{g_0}$.
		\end{theorem}
		\begin{proof}
			We first prove that $\lam_0$ is positive and could be achieved. Suppose $\lam_0=0$, then there exists a sequence of functions $f_k$ such that $\int |d f_k|^2\leq\frac{1}{k}$ and $\int|f_k|^2=1$, which is a bounded sequence in $W^{1,2}$. Thus after passing to a subsequence, there exists $f_{\infty}$ such that $f_k$ convergence to $f_{\infty}$ weakly in $W^{1,2}$ and strongly in $L^2$. Therefore, $\int |f_{\infty}|^2=1,\int f_{\infty} d\Vol_0=0$. In addition, $\|df_{\infty}\|_{L^2}\leq\liminf_{k\to\infty}\|df_k\|_{L^2}=0,$ thus $f_{\infty}$ is a constant, which contradicts to $\int f_{\infty} d\Vol_0=0$. Therefore, $\lam_0>0$ and we still denote $f_{\infty}$ be the weak limit for a sequence of minimizers. 
			
			By Lagrange multiplier theorem, we have $$\int \lan df_{\infty}, df'\ran d\Vol_{0}=\lam_0 \int f_{\infty} f' d\Vol_{0},
			$$
			for any $f'\in W^{1,2}$. Therefore, $f_{\infty}$ is a weak solution to $\Delta_{g_0}-\lam_0$. In addition, as the cone metric is quasi-isomorphism to the Euclidean metric, the coefficients of the Laplace operator will be $L^{\infty}$, then by the elliptic theory for $L^{\infty}$ coefficients elliptic operator, $f_{\infty}$ is a eigenfunction with eigenvalue $\lam_0$.
		\end{proof}
		
		We write $\lam_t$ be the first eigenvalue of $\MD_{\theta_t}$, then $\lam_t$ could be realized by the following variation
		$$\lam_t=\inf_{f}\{\frac{\int \cos\theta_t|df|_{\tg_t}^2d\Vol_t}{\int|f|_{\tg_t}^2d\Vol_t}|f\in W^{1,2},\int fd\Vol_t=0\},$$
		where $d\Vol_t$ is the volume form for the metric $\tg_t$. 
		
		\begin{theorem}
			\label{thm_eigenvaluelowerbound}
			Let $\lam_t$ be the first eigenvalue of $\MD_{\theta_t}$, then $\lim_{t\to 0}\lam_t=\lam_0>0$. In particular, there exists a constant $c_0$ independent of $t$ such that $\lam_t\geq c_0>0$.
		\end{theorem}
		\begin{proof}
			Let $\Vol_t$ be the volume form of $\tg_t$ over $\tL$ and $\Vol_0$ be the volume form of $g_0$ over $\tL$. In local coordinates $(x_1,\cdots, x_n)$, we could write $$\cos\theta_t|df|_{\tg_t}^2d\Vol_t= \sum_{i,j=1}^n\cos\theta_t(\tg_t)^{ij}\pa_{x_i}f\pa_{x_j}f \det(\tg_t)dx_1\we\cdots \we dx_n.$$
			
			As $\lim_{t\to 0}\tg_t=g_0$ and $\lim_{t\to 0}\cos\theta_t=1$, we obtain 
			\begin{equation}
				(1-\ep)\int |d f|_{g_0}^2d\Vol_0\leq \int \cos\theta_t|d f|^2_{\tg_t}d\Vol_t\leq (1+\ep)\int |d f|^2_{g_0}d\Vol_0.
			\end{equation}
			Let $f_t$ be the eigenfunction $\MD_{\theta_t}f_t=\lam_tf_t$, then we compute
			\begin{equation*}
				\begin{split}
					&\int \cos\theta_t|d f_t|_{\tg_t}^2d\Vol_t=\lam_t \int |f_t|_{\tg_t}^2\Vol_t\leq \lam_t(1+\ep)\int |f_t|_{g_0}^2d\Vol_0,\\
					&\int \cos\theta_t|d f_t|_{\tg_t}^2d\Vol_t\geq (1-\ep)\int |d f_t|^2_{g_0}d\Vol_0\geq (1-\ep)\lam_0\int |f_t|_{g_0}^2d\Vol_0.
				\end{split}
			\end{equation*}
			Therefore, $\lam_t\geq \frac{1-\ep}{1+\ep}\lam_0$. 
			
			On the other hand, by the definition of $\lam_0$, for every $\ep$, we can find a function $f$ such that $$\int|d f|_{g_0}^2d\Vol_{0}\leq (\lam_0+\ep)\int |f|^2_{g_0}d\Vol_{0}.$$  In addition, we have
			\begin{equation*}
				\begin{split}
					\int |d f|^2_{g_0}d\Vol_{g_0}\geq (1-\ep)\int |d f|^2_{g_t}d\Vol_t\geq (1-\ep)\lam_t\int |f|^2d\Vol_{t}\geq (1-\ep)^2\lam_t\int |f|^2d\Vol_0
				\end{split}
			\end{equation*}
			Therefore, we conclude that $\lam_t\leq \frac{\lam_0+\ep}{(1-\ep)^2}$, which implies the claim.
		\end{proof}

		\subsection{Curvature and injective radius estimate for the graph family}
		In this subsection, we will compute the second fundamental form and the injective radius for $[\tL_t]$.
		\begin{definition}
			Let $(X,g)$ be a Riemannian manifold with Levi-Civita connection $\na$, let $\iota:L\to X$ be an immersion, the second fundamental form $K_m$ of $(L,\iota)$ at $m\in L$ is defined as 
			\begin{equation*}
				\begin{split}
					K_m&:T_mL\ti T_mL\ti N_{m}\to \mathbb{R},\\
					K_m&(X,Y,n):=\lan \na_{\iota_{\st}X}(\iota_{\st}Y), n\ran,
				\end{split}
			\end{equation*}
			where $N_m$ is the normal bundle of the immersion at $\iota(m)$ with section $n$ of $N_m$, $X,Y$ are vector fields on $L$.
		\end{definition}
		
		Let $[\tL_t]$ be the family of immersed Lagrangians constructed in Theorem \ref{thm_approximation} on $(U_L,J,\omega,\Omega)$ and $\iota_0:L\to U_L$ be the inclusion of the zero section.
		\begin{proposition}
			\label{prop_secondfundamentalform}
			For $[\tL_t]$, let $\tm\in \tL_t$, let $K^t_{\tm}$ be the second fundamental form of $X$ at $\tm$, then
			\begin{itemize}
				\item [(i)] For $\tm\notin\tSigma$,  we write $m=p(\tm)\in L$ and $K^0_{m}$ be the second fundamental form of $\iota_0$. Let $n_t$ be a sequence of smooth section of $N_{\tiota_t(\tm)}$ which convergence continuously to $n_0\in N_{\iota_0(m)}$ on $T^{\st}U_L$, then for any $X,Y\in T_{\tm}\tL$, we have
				$$
				\lim_{t\to 0}K^t_{\tm}(X,Y,n_t):=K^0_m(p_{\st}X,p_{\st}Y,n_0),
				$$
				\item [(ii)] in a suitable neighborhood of $\Sigma$, we write $K^t:=\max_{|X|=|Y|=|n|=1,m\in\Sigma}|K_{m}^t(X,Y,n)|$, then we have $|K^t|\leq \frac{C}{r+t^2}\leq Ct^{-2}.$ 
			\end{itemize}
		\end{proposition}
		\begin{proof}
			For (i), let$(\tx_1,\cdots,\tx_n)$ be a coordinate on a neighborhood of $\tm$ and $(x_1,\cdots,x_n)$ be a coordinate on a neighborhood of $m$ such that $p_{\st}\pa_{\tx_{i}}=\pa_{x_i}$. We also write $(y_1,\cdots,y_n)$ be the fiber coordinates. We define $E_i^t:=(\tiota_t)_{\st}\pa_{\tx_i}$, then under the identification $\Gamma:T(T^{\st}L)\to TL\oplus T^{\st}L$, we could write
			$$
			\Gamma(E_i^t)=(p_{\st}\pa_{\tx_i},\na_{p_{\st}\pa_{\tx_i}}\MA_t^++\na_{p_{\st}\pa_{\tx_i}}\MA_t^-)
			$$
			
			As $m\in L\setminus \Sigma$, $\lim_{t\to 0}\|\na_{p_{\st}\pa_{\tx_i}}\MA_t^+\|_{\MC^{k,\ga}}+\|\na_{p_{\st}\pa_{\tx_i}}\MA_t^-\|_{\MC^{k,\ga}}=0$, we have $\lim_{t\to 0}E_i^t=\pa_{x_i}.$ If we write $\bna$ be the Levi-Civita connection of $(U_L,g_{U_L})$, then $\lim_{t\to 0}\bna_{E_i^t}E_j^t=\bna_{p_{\st}\pa_{\tx_i}}p_{\st}(\pa_{\tx_j}).$
			
			As $\lim_{t\to 0}n_t=n_0$, we compute $$\lim_{t\to 0}K^t_{\tm}(\pa_{\tx_i},\pa_{\tx_j},n_t)=\lim_{t\to 0}\lan \bna_{E_i^t}E_j^t,n_t\ran=\lan \bna_{p_{\st}\pa_{\tx_i}}p_{\st}(\pa_{\tx_j}),n_0\ran=K^0_m(p_{\st}\pa_{\tx_i},p_{\st}\pa_{\tx_j},n_0),$$
			which implies (i).
			
			For (ii), near $\Sigma$, under the identification $\Gamma:T(T^{\st}L)\to TL\oplus T^{\st}L$, we have
			\begin{equation}
				\Gamma\circ\tiota_{\st}(\pa_{\tx_i})=(p_{\st}\pa_{\tx_i},\na_{p_{\st}\pa_{\tx_i}}\MA_t^++\na_{p_{\st}\pa_{\tx_i}}\MA_t^-).
			\end{equation}
			In addition, we could write the expansion of $\MA_t^-\sim t\Re(B_tz^{\frac12}dz)$ with $|B_t|\geq C_0>0$ on $\Sigma$. Then we compute
			$$p_{\st}\pa_{\tz}=2\tz\pa_{z},\;\na_{p_{\st}\pa_{\tz}}\MA^-|_{\Sigma}=tB_tdz,\;\na_{p_{\st}\pa_{\tz}}(\MA^+)=0.$$
			
			We denote $E_i^t:=(\tiota_t)_{\st}(\pa_{\tx_i})$, then by previous computation, we obtain
			$$|E_i|^2\leq r+t^2,\;\mathrm{for \;i=1 \;or \;2 \;and \;}|E_i|^2\sim \MO(1)\;\mathrm{for \;i}\neq 1,2,$$
			which implies the claim.
		\end{proof}
		
		\begin{proposition}
			\label{prop_injectiveradiussubmnaifold}
			Let $(X,g)$ be a Riemannian manifold, $\iota:Y\to X$ be an immersion and let $K_Y$ be the second fundamental form of $Y$, suppose $|K_Y|\leq \tau_Y$, then the injective radius $\inj_p$ of $Y$ at any point $p\in Y$ have
			$\inj_p\geq \frac{\pi}{C+\tau_Y}$, where $C$ is a constant depends on the Riemannian metric on $X$ but independent of $\tau_Y.$  
		\end{proposition}
		\begin{proof}
			By the Gauss equation, let $\Riem$ be the Riemannian curvature tensor of $(Y,\iota^{\st}g)$, then $\max|\Riem|\leq \tau^2$. As $(Y,\iota^{\st}g)$ is a complete Riemannian manifold, for any point $p\in Y$, the injective radius of $Y$ satisfies
			\begin{equation*}
				\inj(Y)\geq \min \{\frac{\pi}{\tau},\frac12 l(Y)\},
			\end{equation*}
			where $l(Y)$ is the infimum of lengths of closed geodesics. 
			
			Now, we will control $l(Y)$. We take a isometric embedding $(X,g)\to \mathbb{R}^N$, if we write $\na^{\mathbb{R}^N}$ be the Levi-Civita connection on $\mathbb{R}^N$, then for $\al,\be\in TY$, we could write 
			$$
			\na_{\al}^{\mathbb{R}^N}\be=\na^X_{\al}\be+K_X,
			$$
			where $\na^X$ is the Levi-Civita connection on $X$ and $K_X$ is a section of the normal bundle of $X$ in $\mathbb{R}^N$. Similarly, we have $\na^X_{\al}\be=\na^Y_{\al}\be+K_Y$, where $\na^Y$ is the Levi-Civita connection on $Y$ and $K_Y$ is a section of the normal bundle of $Y$ in $X$.
			Therefore, we obtain
			\begin{equation}
				\na_{\al}^{\mathbb{R}^N}\be=\na_{\al}\be+K_Y+K_X,
			\end{equation}
			where $\na^{\mathbb{R}^N}$ is the covariant derivative on $\mathbb{R}^N$ and $\na_{\al}\be$ is the covariant derivative on $Y$. Let $\gamma$ be a geodesic on $Y$ and we could consider $\gamma$ is a curve on $\mathbb{R}^N$, then the extrinsic curvature $k_{\gamma}$ of $\gamma$ is defined as $\kappa_{\gamma}=|\na_{\gamma'}^{\mathbb{R}^n}\gamma'|$. In addition, we compute
			\begin{equation}
				\begin{split}
					|\na_{\gamma'}^{\mathbb{R}^n}\gamma'|=&|\na_{\gamma'}\gamma'+K_Y(\gamma',\gamma')+K_X(\gamma',\gamma')|\\
					\leq &|K_Y(\gamma',\gamma')|+|K_X(\gamma',\gamma')|\\
					\leq &\tau_X+\tau_Y,
				\end{split}
			\end{equation}
			where $\tau_X$ is the maximal of the second fundamental form of $X$.  
			
			By Fenchel's theorem on $\mathbb{R}^N$ \cite{aepplialfred1965}, we have $\int_{\gamma}\kappa_{\gamma}ds\geq 2\pi$. Let $l(\gamma)$ be the arc length of $\gamma$, then we have
			$$
			2\pi\leq \int_{\gamma}\kappa_{\gamma}ds\leq l(\gamma)(\tau_X+\tau_Y),
			$$
			which implies $l(Y)\geq \frac{2\pi}{\tau_X+\tau_Y}.$ Therefore, $\inj(Y)\geq \frac{\pi}{\tau_X+\tau_Y}.$
		\end{proof}
		
		Over $(\tL,g_t)$, let $\inj_t$ be minimal of injective radius along different points of $\tL$, then we obtain the following estimate.
		\begin{corollary}
			There exist constants $t_0,C$ independent of $t$ such that for any $t\leq t_0$, we have $\inj_t\geq Ct^2.$
		\end{corollary}
		\begin{proof}
			By Proposition \ref{prop_secondfundamentalform}, for the second fundamental form $K_t$, we have $\max_{p\in\tL}|K_t|\leq Ct^{-2}$. By Proposition \ref{prop_injectiveradiussubmnaifold}, we have $\inj_t\geq \frac{t^2\pi}{C+\tau_Xt^2}$, where $\tau_X$ is the maximal of the second fundamental form for an isomorphic embedding into $\mathbb{R}^n$.
			As the ambient manifold $(X,g)$ is fixed, we take $t_0<\frac{1}{\sqrt{\tau_X}}$, then $\inj_t\geq C't^2$.
		\end{proof}
		
		\begin{subsection}{Size estimate for the Weinstein neighborhood.}
			Let $[L]=(L,\iota)$ be an immersed Lagrangian submanifold in a Calabi-Yau manifold $(X,J,\omega,\Omega)$ with Calabi-Yau metric $g_X$. By Theorem \ref{thm_weinsteinnbhd}, we could find a Weinstein neighborhood of $[L]$ and in this subsection, we will prove a Weinstein neighborhood theorem with size controlled by $\tau$, where $\tau$ is maximum of the second fundamental form of $[L]$. Moreover, we assume $\tau\geq 1$. 
			
			We define $g_L:=\iota^{\st}g_X$ and 
			\begin{equation}
				U(s):=\{\al\in T^{\st}L|\|\al\|_{\MC^0}\leq s\},
				\label{eq_weinsteinsize}
			\end{equation}
			where the $\MC^0$ norm is taken using $g_L$. 
			
			To begin with, we introduce several lemmas to prove the Weinstein tubular neighborhood theorem.
			\begin{lemma}{\cite[Proposition 8.3]{dasilva2001lecturesonsymplecticgeometry}}
				\label{lem_isomorphismofsymplecticform}
				Let $V$ be a $2n$-dimensional vector space, let $\Omega_0$ and $\Omega_1$ be symplectic forms in $V$. Let $S$ be a subspace of $V$ and is Lagrangian for $\Omega_0$ and $\Omega_1$, and let $W$ be any complement to $S$ in $V$, then from $W$, we can canonically construct a linear isomorphism $A:V\to V$ such that $A|_{S}=\Id_S$ and $A^{\st}\Omega_1=\Omega_0$.
			\end{lemma}
			
			\begin{lemma}{\cite[Proposition 6.8]{dasilva2001lecturesonsymplecticgeometry}}
				\label{lem_homologyclassvanishes}
				Let $U$ be a tubular neighborhood of $L$ in $X$ and $\be$ is a closed $k$-form. Suppose the homology class of $\be$ restricts to $L$ vanishes, then we could find a $(k-1)$-form $\mu$ on $U$ such that $\be=d\mu$ and $\mu|_x=0$, where $x\in L$.
			\end{lemma}
			
			\begin{theorem}
				\label{thm_weinsteintubularneighborhood}
				Under the previous assumptions, for the immersed Lagrangian submanifold $[L]=(L,\iota:L\to X)$, there exists a neighborhood $V$ of $[L]$, a neighborhood $U$ of the zero section of $T^{\st}L$, a diffeomorphism $\vp:U\to V$ and a constant $C$ independent of $[L]$ such that
				\begin{itemize}
					\item $\vp^{\st}(\omega)=\omcan$,
					\item we could take $U=U(C\tau^{-1})$, to be a size $C\tau^{-1}$ neighborhood of $T^{\st}L$ defined in \eqref{eq_weinsteinsize}.
				\end{itemize}
			\end{theorem}
			\begin{proof}
				Let $U(s)$ be an size $s$ neighborhood on $T^{\st}L$ defined in \eqref{eq_weinsteinsize}, let $\Phi:U(s)\to X$ be decomposition of the complex structure and the normal exponential map defined in \eqref{eq_normalexpoenntialeqution}. Let $\inj_X$ be the injective radius of $X$, then for any $s<\inj_X$, $\Phi$ will be an immersion. Therefore, we could obtain a pullback Calabi-Yau structure $(U(s),\Phi^{\st}J,\Phi^{\st}\omega,\Phi^{\st}\Omega)$ with pull-back metric $g_U=\Phi^{\st}g$.
				
				Let $x\in L$ and we write $(x,0)$ be a point on the zero section of $U(s)$, we write $S=T_xL$ and $W=(T_pL)^{\perp}$, where the $\perp$ is taken using the metric $g$ over $T_{(x,0)}U(s)$. By Lemma \ref{lem_isomorphismofsymplecticform}, there exists a linear isomorphism $A_x:T_{(x,0)}U(s)\to T_{(x,0)}U(s)$ such that $A_x|_S=\Id$ and $A_x^{\st}\omega_t|_{(x,0)}=\omcan$. In addition, as the choice of $A_x$ is canonical, $A_x$ varies smoothly with respect to $x$.  Using $A_x$, we define a diffeomorphism $h:U(s)\to U(s)$ as 
				$$h(p):=\exp_{x}(A_{x}v),\;\mathrm{where}\;p=(x,v)\in U(s)\subset T^{\st}L.$$ Then we have $h|_{(x,0)}=(x,0)$ and $dh|_{(x,0)}=A_x$ for $p\in U$. Therefore, for $x\in L$, we have $h^{\st}\omega|_{(x,0)}=\omcan|_{(x,0)}$.
				
				Over the neighborhood of $L$, we write $\omega_s=s h^{\st}\omega+(1-s)\omcan$ and we wish to find the size of the tubular neighborhood when $\omega_s$ is symplectic. As $h^{\st}\omega|_{(x,0)}=\omcan|_{(x,0)}$ for $x\in L$, $\omega_s$ is also a symplectic form in a small neighborhood of $L$. 
				
				Now, we will try to understand the size the neighborhood that $\omega_s$ is symplectic. As $h^{\st}\omega$ is the K\"ahler form for the metric $h^{\st}g_U$, we have $\|h^{\st}\omega\|_{\MC^{0}}\LS 1$, where $\MC^{0}$ is taken using $h^{\st}g_U$. As $h^{\st}\omega|_x=\omcan|_x$ for $x\in L$, we also have $\|\omcan\|_{\MC^{0}}\LS 1$, thus $\|\omega_s\|_{\MC^{0}}\LS  1$.
				
				Let $R=\tau^{-1}$, we cover $L$ by radius $R$ ball and over each ball $B_{R}$, we rescale the size of ball by $R^{-1}$, then the Riemannian metric is in Euclidean size. After rescaling back, we obtain the following estimate $\|\det(\omega_s)\|_{\ca}\LS \tau^{\ga}.$
				
				Let $C_0=\min_{p\in L}\|\det(\omega_s)\|_{\MC^{0}}$, over a neighborhood $U_L(C\tau^{-1})\subset T^{\st}L$, for $p=(x,v)\in U_L(C\tau^{-1})$, we compute 
				\begin{equation}
					\begin{split}
						|\det(\omega_s)|(p)&\geq |\det(\omega_s)||(x)-\|\det(\omega_s)\|_{\ca}\|v\|_{\MC^0}\\
						&\geq C_0-C\tau^{-\ga}\|\det(\omega_s)\|_{\ca}\\
						&\geq C_0-CC',
					\end{split}
				\end{equation}
				where the last inequality, we use $\|\det(\omega_s)\|_{\ca}\leq C'\tau^{\ga}$. Therefore, for $C\leq \frac{C_0}{2C'}$, over $U_L(C\tau^{-1})$, we have $|\det(\omega_s)|>0$ which implies $\det(\omega_s)$ is symplectic. 
				
				As the 2-form $h^{\st}\omega-\omcan$ is closed and $(h^{\st}\omega-\omcan)_p=0$ for all $p\in L$, then by Lemma \ref{lem_homologyclassvanishes}, there exists a 1-form $\be$ on $U_L(C_1\tau^{-1})$ such that $h^{\st}\omega-\omcan=d\be$ and $\be|_{(x,0)}=0$ for $x\in L$. We solve the Moser equation for vector field $v_s$: $\iota_{v_s}\omega_s=-\be$ with $v_s=0$ on $L$. Let $\rho_s:U_L(C\tau^{-1})\to U_L(C\tau^{-1})$ be family of diffeomorphisms generates by $v_s$ then $\rho_s^{\st}\omega_s=\omcan$ over $U_L(C\tau^{-1})$. Let $\vp:=\Phi\circ h\circ \rho_1$, then $\vp$ satisfies the claim.
			\end{proof}
		\end{subsection}
		
		Let $\tau_t$ be the maximal of the second fundamental form for the family of approximate Lagrangian submanifolds $[\tL_t]$ constructed in Theorem \ref{thm_approximation} and we assume $\tau_t\geq 1$. By Proposition \ref{prop_secondfundamentalform}, we have $\tau_t\leq Ct^{-2}$, then we obtain the following corollary.
		\begin{corollary}
			There exists constant $C$ independent of $t$ such that for each $t$, we could choose $$U_{\tL}(t):=\{\al\in T^{\st}\tL|\|\al\|_{\MC^{0}}\leq Ct^{-2}\}$$
			to be a Weinstein neighborhood of $[\tL_t]$ and $\MC^0$ is taken by the induced metric $\tg_t$ on $[\tL_t]$.
		\end{corollary}
	Note that by Theorem \ref{thm_metricconvergence}, the induced metric $\tg_t$ on $[\tL_t]$ will convergence to a cone metric, thus the $\MC^0$ norms define by different metric will be equivalent.
		
		\begin{subsection}{Schauder estimate and quadratic term estimate on submanifolds}
			\label{subsection_schauderforsubmanifolds}
			In this subsection, we will study the Schauder estimate and quadratic term estimate for the immersed Lagrangian submanifolds. We will first state the Schauder estimate for a general immersed Lagrangian in a Calabi-Yau and later we will consider the family that constructed in Theorem \ref{thm_approximation}.
			
			\subsubsection{Schauder estimate for submanifold}
			Let $(X,J,\omega,\Omega)$ be a Calabi-Yau manifold with Calabi-Yau metric $g$, and let $[L]=(L,\iota:L\to X)$ be a Lagrangian submanifold, where $\iota$ is a $\MC^{2,\ga}$ immersion with $\iota^{\st}g$ a $\MC^{1,\ga}$ metric. Let $\tau$ be the maximal of the second fundamental form of $[L]$ and we assume that $\tau\geq 1$. Let $\theta$ be the Lagrangian angle for $[L]$ and $\MD_{\theta}f:=-d^{\st}(\cos\theta df)$, where $\st$ is taken using $\iota^{\st}g$. In addition, as $\iota^{\st}g$ is only a $\MC^{1,\ga}$ metric, $\MD_{\theta}$ is an elliptic operator with $\ca$ coefficients. We choose a smooth metric $\hg$ such that the second fundamental form of $\hg$ and $\iota^{\st}g$ are bounded by each other. The different choices of $\hg$ will not influence the following estimates as we only would like to obtain a bound in terms of the second fundamental form.
			\begin{proposition}
				\label{prop_Laplacianestimategeneral}
				For any function $f$ over the Riemannian manifold $(L,\iota^{\st}g)$, we have
				\begin{equation}
					\tau\|d f\|_{\MC^{,\gamma}}+\|\na df\|_{\MC^{,\gamma}}\LS\|\MD_{\theta} f\|_{\MC^{,\gamma}}+\tau^{(2+\al+\frac{n}{2})}\|f\|_{L^2},
				\end{equation}
				where $\na,\MD_{\theta}$ are taking for the metric $\iota^{\st}g$.
				
				Let $U\subset T^{\st}L$ be a Weinstein neighborhood of $[L]$ and we still write $\Omega$ to denote the holomorphic volume form on $U$. Let $\hna$ be the Levi-Civita connection induced by the Sasaki metric of $\hg$ to $U$, we have
				\begin{equation}
					\|\hna^k \Omega\|_{\MC^{,\gamma}}\LS \tau^{(k+\ga)}.
				\end{equation}
			\end{proposition}
			\begin{proof} 
				Let $B_R$ be a radius $R$ ball centered at point $p$, then set $R=\tau^{-1}$ and scale by $\tau$, we obtain a radius $1$ ball with Euclidean geometry. Applying the classical Schauder estimate for elliptic operator with $\ca$ coefficients and rescale back, we obtain the following with $R=\tau^{-1}$:
				$$R^{1+\ga}\|\na f\|_{\MC^{,\gamma}(B_{\frac R2})}+R^{2+\ga}\|\na^2f\|_{\MC^{,\gamma}(B_{\frac R2})}\LS R^{2+\ga}\|\MD_{\theta} f\|_{\MC^{,\gamma}(B_R)}+R^{\frac{n}{2}}\|f\|_{L^2(B_R)},$$ 
				which implies the desire estimate.
				
				Moreover, let $\pi:T^{\st}L\to L$ be the projection map, we define $D_R:={\pi}^{-1}(B_R)\cap U$. Then over $D_R$, by a rescale argument, we have $R^{k+\gamma}\|\hna^k\Omega\|_{\MC^{,\gamma}(B_R)}\LS 1$ for any integer $k\geq 0$.
				
				We cover $L$ by size $\tau^{-1}$ balls, which implies the claim. 
			\end{proof}
			
			\begin{corollary}
				\label{coro_deltainverseestimate}
				Under the previous assumption, let $f$ be a function on $L$ with $\int f d\Vol=0$, $\lam$ be the first eigenvalue of $\MD_{\theta}$ and $\Vol(L)$ be the volume of $L$, then we have
				\begin{equation}
					\tau\|\na f\|_{\ca}+\|\na^2f\|_{\MC^{,\gamma}}\leq C(1+\lam^{-1}\Vol(L)\tau^{2+\ga+\frac{n}{2}})\|\MD_{\theta} f\|_{\ca}.
				\end{equation}
			\end{corollary}
			\begin{proof}
				We compute $\|f\|_{L^2(L)}\leq \lam^{-1}\|\MD_{\theta} f\|_{L^2(L)}\leq \lam^{-1}\Vol(L)\|\MD_{\theta} f\|_{\MC^{,\gamma}}$. The corollary follows by the previous computation and Proposition \ref{prop_Laplacianestimategeneral}.
			\end{proof}
			
			\subsubsection{Quadratic term estimate}
			Let $\al$ be a 1-form on $L$, for the immersed submanifold $[L]$ discussed above, we will estimate the quadratic terms $Q(\al)$ of the special Lagrangian equation in a Weinstein tubular neighborhood of $[L]$, which follows from Joyce \cite{joyce2004slag3}.
			\begin{proposition}{\cite[Proposition 5.8]{joyce2004slag3}}
				\label{prop_quadraticestimateoverfamily}
				Let $\al_1,\al_2$ be 1-forms on $L$ with $\|\al_1\|_{\MC^0}+\|\al_1\|_{\MC^0}\leq C\tau^{-1}$ and 
				$\|\na \al_1\|_{\MC^0}+\|\na \al_2\|_{\MC^0}\leq C
				$, we have
				\begin{equation}
					\label{eq_60}
					\begin{split}
						\|Q(\al_1)-Q(\al_2)\|_{\MC^{,\gamma}}\LS &\tau^{\ga}(\tau\|\al_1-\al_2\|_{\MC^{,\gamma}}+\|\na \al_1-\na \al_2\|_{\MC^{,\gamma}})\cdot\\
						&(\tau\|\al_1\|_{\MC^{,\gamma}}+\tau\|\al_2\|_{\MC^{,\gamma}}+\|\na \al_1\|_{\MC^{,\gamma}}+\|\na \al_2\|_{\MC^{,\gamma}}).
					\end{split}
				\end{equation}
			\end{proposition}
			\begin{proof}
				Let $\al_1,\al_2$ be two 1-forms, for each $x\in L$, we write $P(r,s)=Q(r(\al_1-\al_2)+s\al_2)$. 
				
				Then we compute
				$$Q(\al_1)-Q(\al_2)=P(1,1)-P(0,1)=\int_0^1\frac{\pa P}{\pa r}(r,1)dr.$$
				As $dQ(0)=0$, we obtain $\frac{\pa P}{\pa r}(0,0)=0$. Therefore, we obtain
				$$
				\frac{\pa P}{\pa r}(u,1)=\int_0^1(\frac{d}{ds}(\frac{\pa P}{\pa r}(us,s)))ds=\int_0^1u\frac{\pa^2P}{\pa r^2}(us,s)+\frac{\pa^2P}{\pa r\pa s}(us,s) ds. 
				$$
				By a change of variable, we obtain
				\begin{equation}
					\label{eq_differenticeofq}
					\begin{split}
						Q(\al_1)-Q(\al_2)=\int_0^1\int_0^s[\frac{r}{s^2}\frac{\pa^2P}{\pa r^2}(r,s)+\frac{1}{s}\frac{\pa^2P}{\pa r\pa s}(r,s)]drds.
					\end{split}
				\end{equation}
				
				By \eqref{eq_speiclalLagrangisntitlde}, we have $\tSL:\{(x,y,z):x\in L,\;y\in T_x^{\st}L,\;z\in \otimes^2T_x^{\st}L\}\to \mathbb{R}$ such that $\SL(\al)_x=\tSL(x,\al,\na \al)$. Let $\pa_1\tSL$ be taking the derivative for $y$ variable and $\pa_2\tSL$ be taking the derivative for $z$ variable, then
				\begin{equation*}
					\begin{split}
						\frac{\pa^2P}{\pa r^2}(r,s)=&\otimes^2(\al_1-\al_2)\cdot \pa_1^2\tSL+2(\al_1-\al_2)\otimes(\na \al_1-\na \al_2) \cdot\pa_1\pa_2\tSL\\
						&+\otimes^2(\na \al_1-\na \al_2)\cdot\pa_2^2\tSL,\;\mathrm{and}\\
						\frac{\pa^2P}{\pa r\pa s}(r,s)=&(\al_1-\al_2)\otimes \al_2\cdot \pa_1^2\tSL+((\al_1-\al_2)\otimes\na\al_2+\al_1\otimes (\na\al_1-\na\al_2))\cdot\pa_1\pa_2\tSL\\
						&+(\na \al_1-\na \al_2)\otimes \na \al_2 \cdot \pa_2^2\tSL,
					\end{split}
				\end{equation*}
				where we are evaluating $\pa_i\pa_j\tSL$ at $(x,r(\al_1-\al_2)+s\al_2,r(\na\al_1-\na\al_2)+s\na \al_2)$.
				
				We define \begin{equation}
					\begin{split}
						E_1:=&rs^{-2}\|\al_1-\al_2\|_{\ca}^2+s^{-1}\|\al_1-\al_2\|_{\ca}\|\al_2\|_{\ca},\\
						E_2:=&rs^{-2}\|\na\al_1-\na\al_2\|^2_{\ca}+s^{-1}\|\na(\al_1-\al_2)\|_{\ca}\|\na\al_2\|_{\ca},\\
						E_3:=&2rs^{-2}\|\al_1-\al_2\|_{\ca}\|\na(\al_1-\al_2)\|_{\ca}+s^{-1}\|\al_1-\al_2\|_{\ca}\|\na \al_2\|_{\ca}\\
						&+s^{-1}\|\al_2\|_{\ca}\|\na(\al_1-\al_2)\|_{\ca}.
					\end{split}
				\end{equation}
				
				Then \eqref{eq_differenticeofq} implies
				\begin{equation}
					\begin{split}
						&\|Q(\al_1)-Q(\al_2)\|_{\ca}\\
						\leq&\int_0^1\int_0^s(E_1\|\pa_1^2\tSL\|_{\ca}^2+E_2\|\pa_2^2\tSL\|_{\ca}+E_3\|\pa_1\pa_2\tSL\|_{\ca})drds.
					\end{split}
				\end{equation}
				
				As $\|\hna^k\Omega\|_{\MC^{,\ga}}\LS \tau^{k+\ga}$, we have 
				\begin{equation*}
					\|\pa_1^2\tSL\|_{\MC^{,\gamma}}\LS \tau^{2+\ga},\;\|\pa_1\pa_2\tSL\|_{\MC^{1,\ga}}\LS \tau^{1+\ga},\;\|\pa_2^2\tSL\|_{\ca}\LS\tau^{\ga},
				\end{equation*}
				which implies \eqref{eq_60}.
			\end{proof}
			\subsubsection{Corresponding estimates for the family.}
			
			Let $[\tL_t]=(\tL,\tiota_t)$ be the family of special Lagrangian submanifolds we constructed in Theorem \ref{thm_approximation} with $\tg_t:=\tiota_tg$ be the induced $\MC^{1,\ga}$ metric and $\tau_t\geq 1$ be the maximal of the second fundamental form. 
			
			By Theorem \ref{thm_metricconvergence}, $\tg_t$ convergence to a cone metric with bounded coefficients. If we take $g'$ be a smooth background metric on $\tL$, then there exists a $t$-independent constant $C$ such that $\|\tg_t\|_{\MC^{0}_{g'}}\leq C$, where the $\MC^0$ norm is taken using $g'$. For each $t$, we choose $\hg_t$ be a smooth Riemannian metric close to $\tg_t$ on $\tL$ such that $\|\tg_t-\hg_t\|_{\MC_{g'}^{1,\ga}}\leq \tau_t^{-1}$. If we write $\htau$ be the maximal of the second fundamental form of $\hg_t$, then we have 
			$$
			|\htau_t-\tau_t|\leq \|\tg_t\|_{\MC^0_{g'}}\|\hg_t-\tg_t\|_{\MC_{g'}^{1}}+\|\hg_t\|_{\MC^0_{g'}}\|\hg_t-\tg_t\|_{\MC_{g'}^{1}}\leq C\tau_t^{-1},
			$$
			where $C$ is a $t$-independent constant. Therefore, we obtain $\htau_t\leq \tau_t+C\tau_t^{-1}$. Our following estimates will be independent of the choice $\hg_t$ and $t$ as we will estimate using $\htau_t$, which is bounded by $C\tau_t$ for some constant $C$.
			
			We write $\hna_t$, $\na_t$ be the Levi-Civita connections of the Sasaki metric on $T^{\st}\tL$ defined by $\hg_t$, $\tg_t$. Let $\MD_{\theta_t}(f):=d^{\st}(\cos\theta_t d f)$ and $\Vol_t$ be the volume of $(\tL,\tg_t)$. 
			\begin{proposition}
				Let $f$ be a function on $\tL$ with $\int_{\tL}fd\Vol_t=0$, we have
				\begin{equation}
					t^{-2}\|d f\|_{\ca}+\|\na_t df\|_{\ca}\LS (1+t^{-(4+2\ga+n)})\|\MD_{\theta_t} f\|_{\ca}.
				\end{equation}
				For the holomorphic volume form, we have
				$$\|\hna_t^k\Omega\|_{\MC^{,\ga}}\LS t^{-2(k+\ga)}.$$
			\end{proposition}
			\begin{proof}
				We write $\lam_{t}$ be the 1-st eigenvalue on $\tL_t$, then by Theorem \ref{thm_eigenvaluelowerbound}, there exists $c_0>0$ such that $\lam_t\geq c_0>0$. In addition, by Theorem \ref{thm_metricconvergence}, $\Vol_t\leq 2\Vol(L)+1$ which implies the claim.
			\end{proof}
			
			By Proposition \ref{prop_specialLagrangianlineartermstructure}, the special Lagrangian equation on $[\tL_t]$ could be written as
			$$\SL_t(\al)=\sin\theta_t-d^{\st_t}(\cos\theta_t \al)+Q_t(\al).$$ We have the following estimate for the quadratic terms $Q_t(\al).$
			\begin{corollary}
				For 1-forms $\al_1,\al_2$ on $[\tL_t]$ with $\|\al_1\|_{\MC^0}+\|\al_1\|_{\MC^0}\leq Ct^2$ and 
				$\|\na_t \al_1\|_{\MC^0}+\|\na_t \al_2\|_{\MC^0}\leq C
				$, we have
				\begin{equation}
					\begin{split}
						\|Q_t(\al_1)-Q_t(\al_2)\|_{\MC^{,\gamma}}\LS &t^{-2\ga}(t^{-2}\|\al_1-\al_2\|_{\MC^{,\gamma}}+\|\na_t \al_1-\na_t \al_2\|_{\MC^{,\gamma}})\cdot\\
						&(t^{-2}\|\al_1\|_{\MC^{,\gamma}}+t^{-2}\|\al_2\|_{\MC^{,\gamma}}+\|\na_t \al_1\|_{\MC^{,\gamma}}+\|\na_t \al_2\|_{\MC^{,\gamma}}).
					\end{split}
				\end{equation}
			\end{corollary}
			
		\end{subsection}
	\end{section}
	
	\begin{section}{Branched Deformation Family Construction and nearby Special Lagrangian Theorem}
		\label{sec_nearbyspecialLagrangianfamily}
		In this section, we will prove our main theorem and finish our construction of branched deformation family. We will first prove a nearby special Lagrangian theorem in our setting, which is first obtained by Joyce \cite{joyce2004slag2}, then we will apply the nearby special Lagrangian theorem to the family we constructed in Theorem \ref{thm_approximation} to obtain the branched deformations.
		\subsection{Nearby special Lagrangian theorem}
		Let $(X,J,\omega,\Omega)$ be a Calabi-Yau manifold, $[L]=(L,\iota:L\to X)$ be an immersed Lagrangian submanfiold with $\MC^{2,\ga}$ structure and $\theta$ be the Lagrangian angle. By Theorem \ref{thm_weinsteintubularneighborhood}, we could choose the Weinstein neighborhood $U_L$ be
		\begin{equation*}
			U_L=\{\al\in T^{\st}L|\|\al\|_{\MC^0}\leq C\tau^{-1}\},
		\end{equation*}
	where $C$ is a $\tau$ independent constant.
		
		Let $f$ be a function on $L$, such that $df\in U_L$, then by Proposition \ref{prop_specialLagrangianlineartermstructure}, we could write the special Lagrangian equation for $df$ as
		\begin{equation}
			\sin\theta-\MD_{\theta} f+Q(df)=0,
			\label{eq_proofnearbySLthm}
		\end{equation}
		where $\MD_{\theta} f:=d^{\st}(\cos\theta df).$ 
		
		We start with a lemma follows from Aubin \cite[Thorem 4.7]{aubin1982nonlinearanalysis}.
		\begin{lemma}
			For each $h\in \MC^{,\gamma}(L)$ with $\int_{L}h d\Vol=0$, there exists $f\in \MC^{2,\ga}$ such that $\int_{L}f d\Vol=0$ and $\MD_{\theta}f=-h$.
		\end{lemma}
		
		We take $f:=\MD_{\theta}^{-1}h$, then we are looking for suitable $h$ with $\int_{L}hd\Vol=0$ to solve
		$$
		h+Q(d\MD_{\theta}^{-1}h)=-\sin\theta.
		$$
		We define the operator 
		$$T:\ca\to \ca,\;Th:=-Q(d\MD_{\theta}^{-1}h)-\sin\theta,$$ 
		where $\MC^{,\ga}$ is the H\"older norm for functions on $L$, then we have:
		\begin{proposition}
			\label{prop_provingcontractionmap}
			Let $\lam$ be the first eigenvalue of $\MD_{\theta}$, then there exists a constant $C$ such that for $N:=C\tau^{\ga}(1+\lam^{-1}\Vol(L)\tau^{2+\ga+\frac{n}{2}})^2$ and $B_{N^{-1}}:=\{f\in\ca| \|f\|_{\ca}\leq N^{-1}\}$, suppose $\|\sin\theta\|_{\MC^{,\gamma}}\leq \frac{1}{4N}$, then $T:B_{N^{-1}}\to B_{N^{-1}}$ is a contraction mapping.
		\end{proposition}
		\begin{proof}
			For $h_1,h_2\in B_N$, we compute
			\begin{equation*}
				\begin{split}
					\|Th_1-Th_2\|_{\ca}=&\|Q_g(d\MD_{\theta}^{-1}h_1)-Q_g(d\MD_{\theta}^{-1}h_2)\|_{\ca}\\
					\LS &{\tau}^{\ga}(\tau\|d {\MD_{\theta}}^{-1}(h_1-h_2)\|_{\MC^{,\gamma}}+\|\na d{\MD_{\theta}}^{-1}(h_1-h_2)\|_{\MC^{,\gamma}})\cdot\\
					&(\tau\|d \MD^{-1}_{\theta}h_1\|_{\MC^{,\gamma}}+\tau\|d \MD^{-1}_{\theta}h_2\|_{\MC^{,\gamma}}+\|\na d {\MD}^{-1}_{\theta}h_1\|_{\MC^{,\gamma}}+\|\na d{\MD}^{-1}_{\theta}h_2\|_{\MC^{,\gamma}}),\\
				\end{split}
			\end{equation*}
			where for the second inequality we use Proposition \ref{prop_quadraticestimateoverfamily}. 
			By Corollary \ref{coro_deltainverseestimate}, we have 
			\begin{equation*}
				\begin{split}
					&\tau\| d \MD^{-1}_{\theta}h_i\|_{\ca}+\|\na d\MD^{-1}_{\theta}h_i\|_{\MC^{,\gamma}}\LS (1+\lam^{-1}\Vol(L)\tau^{2+\ga+\frac{n}{2}})\|h_i\|_{\ca},\;\mathrm{and}\\
					&\tau\|d \MD^{-1}_{\theta}(h_1-h_2)\|_{\ca}+\|\na d\MD^{-1}_{\theta}(h_1-h_2)\|_{\MC^{,\gamma}}\LS (1+\lam^{-1}\Vol(L)\tau^{2+\ga+\frac{n}{2}})\|h_1-h_2\|_{\ca}.
				\end{split}
			\end{equation*}
			
			Therefore, there exists a constant $C_0$ such that
			\begin{equation*}
				\begin{split}
					\|Th_1-Th_2\|_{\ca}\leq &C_0 \tau^{\ga}(1+\lam^{-1}\Vol(L)\tau^{2+\ga+\frac{n}{2}})^2 \|h_1-h_2\|_{\MC^{,\gamma}}(\|h_1\|_{\ca}+\|h_2\|_{\ca})\\
					\leq &\frac{C_0}{C}\|h_1-h_2\|_{\MC^{,\gamma}}.
				\end{split}
			\end{equation*}
			
			Thus, for any $C\geq 4C_0$, we have $\|Th_1-Th_2\|_{\ca}\leq \frac14 \|h_1-h_2\|_{\ca}$. Moreover, we still need to check $T:B_N\to B_N$ is a well-defined map. By a similar computation, we obtain
			\begin{equation*}
				\begin{split}
					\|Th\|_{\ca}&\leq \|Q(d\MD_{\theta}^{-1}h)\|_{\MC^{,\gamma}}+\|\sin\theta\|_{\MC^{,\gamma}}\\
					&\leq C_0'(\tau^{\ga}(1+\lam^{-1}\Vol(L)\tau^{2+\ga+\frac{n}{2}})^2\|h\|^2_{\ca}+\|\sin\theta\|_{\ca})\\
					&\leq (C_0'C^{-1}+\frac{1}{4})N^{-1}.
				\end{split}
			\end{equation*}
			We take $C=\max\{4C_0',4C_0\}$, then $T:B_{N^{-1}}\to B_{N^{-1}}$ is a well-defined contraction map.
		\end{proof}
		
		\begin{corollary}
			\label{cor_existencedfsecondfundamentalform}
			There exists an unique solution to \eqref{eq_proofnearbySLthm}
			which satisfies
			\begin{equation*}
				\tau\|d f\|_{\ca}+\|\na d f\|_{\ca}\LS \tau^{-\ga}(1+\lam^{-1}\Vol(L)\tau^{2+\ga+\frac{n}{2}})^{-1}.
			\end{equation*}
		In particular, $df$ lies in the Weinstein neighborhood $U_L$.
		\end{corollary}
		\begin{proof}
			Let $h$ be the fixed point of the contraction map $T:B_{N^{-1}}\to B_{N^{-1}}$, then $\|h\|_{\ca}\leq N^{-1}$. By Corollary \ref{coro_deltainverseestimate}, we obtain 
			$$\tau \|d f\|_{\ca}+\|\na d f\|_{\ca}\LS (1+\lam^{-1}\Vol(L)\tau^{2+\ga+\frac{n}{2}})\|h\|_{\ca}\LS N^{-1}(1+\lam^{-1}\Vol(L)\tau^{2+\ga+\frac{n}{2}}).$$
		\end{proof}

		Now, we will consider the regularity of the special Lagrangian submanifold we constructed. Suppose that $\iota:L\to X$ is a smooth immersion, then we could get a good regularity results of solution $f$ by standard elliptic estimate. In particular, as $f$ satisfies $\tSL(x,df,\na df)=0$, which is a second order non-linear elliptic equation. As $\iota$ is a smooth immersion, $\tSL(x,y,z)$ is a smooth function w.r.t. the variables. By \cite[Theorem 3.56]{aubin1982nonlinearanalysis}, any $\MC^2$ solution $f$ to \eqref{eq_proofnearbySLthm} will be smooth. 
		
		Let $\iota_f:L\to U\subset T^{\st}L$ be the inclusion given by the graph of $df$, then $\iota_f$ will be a $\MC^{1,\ga}$ immersion, then $[L_f]=(L,\iota_f)$ is a $\MC^{1,\ga}$ special Lagrangian submanifold. Even $\iota$ is not a smooth immersion, we still could obtain good regularity for $[L_f]$. We have the following regularity theorem for minimal submanifold due to Morrey.
		
		\begin{theorem}{\cite{morrey1966multipleintegrals}}
			Any $\MC^1$ special Lagrangian submanifold is real analytic. 
		\end{theorem}
		In particular, as $[L_f]$ we constructed above is a $\MC^{1,\ga}$ submanifold, $[L_f]$ is real analytic. In summary, we obtain the following nearby special Lagrangian theorem, which will be a special case of Joyce \cite{joyce2004slag2}.
		\begin{theorem}
			\label{thm_nearbyLagrangian}
			Let $[L]=(L,\iota)$ be an immersed Lagrangian submanifold, let $\lam$ be the first eigenvalue of the linearization operator $\MD$ and $\tau$ be the maximal of the second fundamental form. There exists a constant $C$ such that for $\Upsilon:=C(1+\lam^{-1}\Vol(L)\tau^{2+\ga+\frac{n}{2}})^2\tau^{\ga}$, suppose $\|\st\iota^{\st}\Im\Omega\|_{\ca}\leq \Upsilon^{-1}$, then there exists a function $f$ on $L$ with $\int_L f d\Vol_L=0$ and 
			$$\|df\|_{\MC^{1,\ga}}\LS \tau^{-\ga}(1+\lam^{-1}\Vol(L)\tau^{2+\ga+\frac{n}{2}})^{-1},$$
			such that the graph submanifold $[L_f]$ of $df$ in the Weinstein neighborhood will be a real analytic special Lagrangian submanifold.
		\end{theorem}
		
		\subsection{Existence of special Lagrangian submanifolds near the approximate family}
		Let $(X,J,\omega,\Omega)$ be a Calabi-Yau manifold with a special Lagrangian submanfiold $[L]=(L,\iota_0)$. Suppose over $[L]$ there exists a nondegenerate harmonic pair, $[\tL_t^{\app}]=(\tL,\tiota_t^{\app})$ be the family of approximate special Lagrangian submanifolds constructed in Theorem \ref{thm_approximation} and $\theta_t$ be the Lagrangian angle. In this subsection, we will apply the nearby special Lagrangian theorem for the family of Lagrangian submanifold $[\tL_t^{\app}]$ to construct special Lagrangian submanifolds.
		
		Let $\tau_t$ be the maximal of second fundamental form of $[\tL_t^{\app}]$, then by Proposition \ref{prop_secondfundamentalform}, we have $\tau_t\LS t^{-2}.$ We define $\tg_t=(\tiota_t^{\app})^{\st}g$ be the pull-back Riemannian metric. 
		
		By Theorem \ref{thm_metricconvergence}, $\tg_t$ convergence to a cone metric which is quasi-isomorphic to any smooth Riemannian metric on $\tL$. The $\ca$ norm defined using different $\tg_t$ or the cone metric are equivalent. Therefore, for the rest of this subsection, we don't label $t$ for different $\ca$ norms. Moreover, by Theorem \ref{thm_weinsteintubularneighborhood}, for each $t$, we could choose the Weinstein neighborhood of $[\tL_t^{\app}]$ to be $$U_t:=\{\al\in T^{\st}\tL|\|\al\|_{\ca}\leq Ct^2\}.$$
		
		Let $f_t$ be a function on $\tL$ such that $df_t\in U_t$, the special Lagrangian equation for $df_t$ in this Weinstein neighborhood will be
		$$
		\sin\theta_t-\MD_{\theta_t} f_t+Q_t(df_t)=0.
		$$
		We write $f_t=-\MD_{\theta_t}^{-1} h_t$ with $\int_{\tL}h_td\Vol_t=0$, where $d\Vol_t$ is the volume form w.r.t. the Riemannian metric $g_t$, then the special Lagrangian equation for $h_t$ will be 
		$$\sin\theta_t+h_t+Q_t(d\MD_{\theta_t}^{-1}h_t)=0.$$
		We write $T_th_t:=\sin\theta_t+Q_t(d\MD_{\theta_t}^{-1}h_t)$, by Proposition \ref{prop_provingcontractionmap} and Corollary \ref{cor_existencedfsecondfundamentalform}, we obtain 
		\begin{proposition}
			\label{prop_familyexistence}
			For $N_t=t^{8+4\ga+n}$, there exists a constant $C$ independent of $t$ such that suppose $\|\sin\theta_t\|_{\MC^{,\gamma}}\leq \frac{N_t}{4}$, then $T_t:B_{t}\to B_{t}$ is a contracting map for $B_t:=\{\|\al\|_{\ca}\leq C^{-1}N_t\}.$ Moreover, let $h_t$ be the fixed point of $T_t$ with $f_t=\MD_{\theta_t} h_t$, then $f_t$ satisfies 
			$$t^{-2}\|d f_t\|_{\MC^{,\gamma}}+\|\na_t d f_t\|_{\ca}\leq C^{-1}t^{4+n+4\ga}.$$
		\end{proposition}
		
		By Theorem \ref{thm_approximation}, for any $k$, we can choose $[\tL_t^{\app}]$ such that $\|\sin\theta_t\|_{\ca}\leq C_0t^k$. When $t$ small enough and $t^k\geq t^{8+4\ga+n+1}$, we could obtain the desire approximate solutions.
		
		From the above estimate, we know that $df_t\subset U_t$. We could write the inclusion induces by $df_t$ as $\tiota_t':\tL\to U_t\subset T^{\st}\tL$, which is a special Lagrangian submanifold. Using $\Phi_t:U_t\to U_t'\subset X$, we define $\tiota_t:=\Phi_t\circ \tiota_t'$, then $[\tL_t]=(\tL,\iota_t)$ will be a family of special Lagrangian submanifolds on $X$. 
		
		Now, let's discuss the rate of the convergence. By Proposition \ref{prop_familyexistence}, we have $$\|\tiota_t-\tiota_t^{\app}\|_{\ca}\LS \|df_t\|_{\ca}\leq Ct^{4+n+4\ga}.$$
		In addition, by Theorem \ref{thm_approximation} (vi), we have $\|\tiota_t^{\app}-\iota_0\circ p\|_{\ca}\leq C t.$
		Therefore, we obtain $$\|\tiota_t-\iota_0\circ p\|_{\ca}\leq C \|\tiota_t-\tiota_t^{\app}\|_{\ca}+\|\tiota_t^{\app}-\iota_0\circ p\|_{\ca}\leq  Ct.$$
		
		\begin{comment}
		The above convergence is much stronger than convergence in current. For any n-form $\chi$ on $X$, we have
		\begin{equation*}
		\begin{split}
		\lim_{t\to 0}\int_{\tL}\iota_t^{\st}\chi=
		\end{split}
		\end{equation*}
		\end{comment}
		
		On the other hand, let $\mfa$ be the nondegenerate harmonic pair and $\phi_t$ be the diffeomorphism constructed in Theorem \ref{thm_approximation}. We write $\tiota_{\phi_t^{\st}(t\mfa)}:\tL\to U_L$ be the inclusion map induces by the graph of $t\phi_t^{\st}(\mfa)$, then by Theorem \ref{thm_approximation} (ii), we have $\|\tiota_t^{\app}-\tiota_{t\phi_t^{\st}(\mfa)}\|_{\ca}\leq Ct^2$,
		which implies 
		$$
		\|\tiota_t-\tiota_{t\phi_t^{\st}(\mfa)}\|_{\ca}\leq Ct^2.
		$$
		
		In summary, we could state our main theorem.
		\begin{theorem}
			\label{thm_maintheoremversion}
			Let $(X,J,\omega,\Omega)$ be a Calabi-Yau manifold and $[L]=(L,\iota_0: L\to X)$ be a special Lagrangian submanifold with induced metric $g_L$. Suppose there exists a nondegenerate harmonic pair $\mfa=\al^++\al^-$ over $[L]$, let $p:\tL\to L$ be the branched covering of $L$ induces by $\al^-$, then there exists a positive constant $T$ and a family of special Lagrangian submanifold $\tiota_t:\tL\to X$ for $|t|\leq T$ such that  
			\begin{itemize}
				\item [(i)] $\lim_{t\to 0}\|\tiota_t-\iota_0\circ p\|_{\ca}=0$.
				\item [(ii)] $\tiota_t(\tL)$ convergence to 2$\iota_0(L)$ as current.
				\item [(iii)] In a Weinstein neighborhood $U_L\subset T^{\st}L$ of $[L]$, there exists a family of diffeomorphisms $\phi_t:L\to L$ such that suppose we write $\iota_{t\phi_t^{\st}(\mfa)}:\tL\to U_L$ be the inclusion map induces by the graph of $t\phi_t^{\st}(\mfa)$, then  $$\|\tiota_t-\tiota_{t\phi_t^{\st}(\mfa)}\|_{\ca}\leq C t^2,$$
				where $C$ is a t-independent constant and $0<\gamma<\frac12$,
			\end{itemize}
		where the $\ca$ norms are defined using the metric $p^{\st}g_L$.
 		\end{theorem}
		
		Suppose $L$ is a Riemannian surface, by Proposition \ref{prop_RiemanniansurfacespecialLagrangian}, the branched deformations generates by quadratic differentials will be embedded. However, in general, suppose $\iota_0$ is an embedding, $\tiota_t$ we constructed above will not be embedded. By Proposition \ref{prop_immersion}, $\tiota_t^{\app}$ could have self intersections, as $\tiota_t$ could be regarded as a small deformation of $\tiota_t^{\app}$, thus we don't expect $\tiota_t$ is embedded.
		
		\begin{proposition}
			Let $[L]=(L,\iota_0: L\to X)$ is an embedded special Lagrangian submanifold and $\mfa=\al^++\al^-$ is a nondegenerate harmonic pair over $[L]$, suppose $|\al^-|^{-1}(0)=\Sigma$, then there exists a positive constant $T_0$ such that for $|t|<T_0$, 
			$\tiota_t$ constructed in Theorem \ref{thm_maintheoremversion} is an embedding. 
		\end{proposition}
		\begin{proof}
			We only need to prove this in a Weinstein neighborhood $U_L$. Let $\tiota_{t\mfa}:\tL\to U_L$ be the inclusion induces by the graph of $t\mfa$, as $|\al^-|^{-1}(0)=\Sigma$, $\tiota_{t\mfa}:\tL\to U_L$ is an embedding. Let $\MA_t=\MA_t^++\MA_t^-$ be a nondegenerate 1-form we constructed in Theorem \ref{thm_approximation} with inclusion $\tiota_t^{\app}$, by Theorem \ref{thm_approximation}, as $\|t\mfa-\MA\|_{\ca}\leq Ct^2$, then for $t$ sufficiently small $\tiota_t^{\app}$ is a embedding. 
			Moreover, as $\|\tiota_t-\tiota_t^{\app}\|_{\ca}\leq Ct^{4+n+4\ga}$, $\tiota_t$ is also an embedding for sufficiently small $t$.
		\end{proof}
		
		\subsection{The branched deformation theorem for fixed real locus}
		Suppose the special Lagrangian is a fixed real locus, then the branched deformations we constructed in Theorem \ref{thm_maintheoremversion} could also have extra symmetry. 
		\begin{theorem}
			Let $(X,J,\omega,\Omega)$ be a Calabi-Yau manifold and $[L]=(L,\iota_0: L\to X)$ be a special Lagrangian submanifold which is locally a fixed real locus of anti-holomorphic involution $R$. Suppose there exists a $\ZT$ harmonic 1-form $\al^-$ over $[L]$, then the family of special Lagrangian submanifolds $[\tL_t]=(\tL,\tiota_t)$ constructed in Theorem \ref{thm_maintheoremversion} will satisfy $R\circ\tiota_t=\tiota_t\circ \sigma$, where $\sigma$ is the involution on $\tL$.
		\end{theorem}
		\begin{proof}
			By Theorem \ref{thm_approximatefamilyinvolution} and Corollary \ref{cor_invariaitonslequationinvolution}, the special Lagrangian equation over $[\tL_t^{\app}]$ for $df_t$ over $\tL$ will satisfy $\sigma^{\st}\SL(df_t)=-\SL(-\sigma^{\st}df_t).$ Therefore, suppose $f_t$ is a solution, then $-\sigma^{\st}df_t$ is also a solution. By Corollary \ref{cor_existencedfsecondfundamentalform}, the solution is unique for the contracting map, therefore $df_t=-\sigma^{\st}df_t$. Recall that $\tiota_t:\tL\to X$ is given by the graph of $df_t$, which defined as $\tiota_t(x)=\exp_{\tiota_t^{\app}(x)}(J(\tiota_t^{\app})_{\st}V_{df})$, where $V_{df_t}$ is the dual vector of $df_t$ given by the Riemannian metric. By the same computation as \eqref{eq_identityinvolution}, we obtain $R\circ\tiota_t=\tiota_t\circ \sigma.$
		\end{proof}
		
		In general, the image of the branch locus $\tiota_t(\Sigma)$ will not lie in $[L]$. However, in this case, by $R\circ\tiota_t=\tiota_t\circ \sigma,$ we obtain $R\circ \tiota_t(x)=\tiota_t(x)$ for any $x\in\Sigma$.
		\begin{corollary}
			The branch locus $\tiota_t(\Sigma)$ is a subset of $\iota_0(L)$. 
		\end{corollary}
	\end{section}
	
	\section{Applications and Generalizations}
	\label{sec_applications}
	In this section, we will discuss a few cases that our branched deformation theorem could be applied and generalized.
	\subsection{Family of Calabi-Yau structures}
	Combining with the McLean's deformation theorem, our main theorem could be generalized to a smooth family of Calabi-Yau submanifolds. Let $(X,J_s,\omega_s,\Omega_s)$ be a family of Calabi-Yau manifolds for $s\in (-\ep,\ep)$ and $[L_{0,0}]=(L,\iota_{0,0})$ be a special Lagrangian submanifold on $(X,J_0,\omega_0,\Omega_0)$. Suppose the homology class of $\omega_s$ and $\Im\Omega_s$ restricted to $[L_0]$ vanishes, then by the deformation theorem for family of Calabi-Yau structures \cite[Theorem 5.3]{joycelecturesonSLgeometry05}, $[L_0]$ could be extended smoothly to a family of special Lagrangian submanifolds $[L_{s,0}]=(L,\iota_{s,0}).$ We write the induced metric on $[L_{s,0}]$ as $g_{s,0}=\iota_{s,0}^{\st}g$. 
	
	Suppose over $[L_{0,0}]$, there exists a nondegenerate $\ZT$ harmonic 1-form $\al_0$, then by Theorem \ref{thm_donaldsondeformation}, there exists $\ep_0\leq \ep$ such that for each $s\in (-\ep_0,\ep_0)$, there exists a family of nondegenerate $\ZT$ harmonic 1-forms $(g_{s,0},\Sigma_s,\al_s)$ over $[L_{s,0}]$. Even $\al_s$ is branched along a different singular set $\Sigma_s$, $\Sigma_s$ are close to $\Sigma$ and the branched covering map $p_s:\tL\to L$ has the same diffeomorphism type. For each $v_s$, by Theorem \ref{thm_maintheoremversion}, we obtain a family of special Lagrangian submanifolds $$[\tL_{s,t}]=(\tL,\tiota_{s,t}:\tL\to X),$$
	such that $\lim_{t\to 0}\|\tiota_{s,t}-\iota_{s,0}\circ p_s\|_{\ca}=0.$ 
	
	\begin{theorem}
		\label{thm_brancheddeformationforfamily}
		Let $(X,J_s,\omega_s,\Omega_s)$ be a family of Calabi-Yau manifold, let $[L_{0,0}]=(L,\iota_{0,0})$ is a special Lagrangian submanifold for $(X,J_0,\omega_0,\Omega_0)$. Suppose 
		\begin{itemize}
			\item[(i)] the homology class of $\omega_s,\;\Im\Omega_s$ restricted to $[L_{0,0}]$ vanishes, 
			\item[(ii)] over $[L_{0,0}]$, there exists a $\ZT$ harmonic form $v_0$, with branched covering $p_0:\tL\to L$.
		\end{itemize}
		then for small $\ep_0$ with $|s|\leq \ep_0$, there exists a family of inclusions $\tiota_{s,t}:\tL\to X$ such that
		\begin{itemize}
			\item [(i)] $[\tL_{s,t}]=(\tL_{s,t},\tiota_{s,t})$ are special Lagrangian submanifolds on $\mbX_s$,
			\item [(ii)] $\lim_{s\to 0}\lim_{t\to 0}\|\tiota_{s,t}-\iota_{0,0}\circ p\|_{\ca}=0$.
		\end{itemize}
	\end{theorem}

	\subsection{The moduli space of the branched deformation family}
	The moduli space of special Lagrangian submanifold have been studied in \cite{joycelecturesonSLgeometry05,mclean98defomation,hitchin1997modulispace}. In this subsection, we will briefly discuss the meaning of the branched deformation theorem for the moduli space.
	
	Let $[L]=(L,\iota_0)$ be a special Lagrangian submanifold in a Calabi-Yau manifold $(X,J,\omega,\Omega)$. Let $\MM_{[L]}$ be the connected component of the set of special Lagrangian submanifold containing $[L]$. Then McLean's deformation theorem \ref{theorem_classicaldeformation} implies that $\MM_{[L]}$ is a smooth manifold with dimension $b^1(L)$, which is the first Betti number of $L$. 
	
	Suppose there exists a nondegenerate harmonic pair $(\al^+,\al^-)$ on $[L]$, and we write $[\tL_s]=(\tL,\tiota_s)$ be the family of branched deformation constructed in Theorem \ref{thm_maintheoremversion}. We could consider the moduli space $\MM_{[\tL]}$, which is the connected components containing all the $[\tL_s]$, then $\MM_{[\tL]}$ is also a smooth moduli space with dimension $b^1(\tL)$. By Proposition \ref{lem_decompositioninvolutionss}, $b^1(\tL)=b^1_-+b^1(L)$. Therefore, we might enlarge $\MM_{[\tL]}$ to a compactifcation $\overline{\MM_{[\tL]}}$ by adding $\MM_{[L]}$ as a codimension $b^1_-$ boundary. 
	
	The inverse of the above discussions is incorrect. Let $[L]=(L,\iota_0)$ be an embedded special Lagrangian submanifold such that the induced metric on $[L]$ is Ricci positive, then by \cite{abouzaidImagi21,tsai2018mean}, $[L]$ doesn't admit any deformation. However, by Example \ref{ex_branched1from}, in the case when $L=S^3$, we could construct branched covering $p:\tL\to S^3$ such that $b^1_-(\tL)\geq 1$. We could also regard $(\tL,\iota_0\circ p)$ be a special Lagrangian submanfold, which is not immersed. The harmonic 1-forms on $(\tL,\iota_0\circ p)$ will not generate a deformation.

	\subsection{Branched deformation of the zero section}
	In this subsection, we will introduce examples that our deformation theorem could be applied using the Calabi-Yau neighborhood theorem.
	
	\begin{definition}
		Let $(L,g_L)$ be a real analytic Riemannian manifold, let $U_L$ be a neighborhood of the zero section in $T^{\st}L$. A Calabi-Yau neighborhood of $L$ is a Calabi-Yau structure $(U_L,J,\omega,\Omega)$ with Calabi-Yau metric $g$ such that
		\begin{itemize}
			\item [(i)]$\omega$ is the canonical symplectic form,
			\item [(ii)]the restriction of $g$ to the zero section is $g_L$,
			\item [(iii)]the zero section is a special Lagrangian submanifold in this Calabi-Yau structure.
		\end{itemize}
	\end{definition}
	
By Stenzel \cite{stenzel93}, every compact rank 1 global symmetric space admits a Calabi-Yau neighborhood. Moreover, there are more examples.
	
	\begin{theorem}{\cite{bryant2000calibratedembedding,doicu2015}}
		\label{thm_CYstructureoncotangentbundle}
		Let $(L,g_L)$ be a real analytic Riemannian manifold, suppose $\chi(L)=0$, then $L$ admits a Calabi-Yau neighborhood $(U_L,J,\omega,\Omega)$. In addition, let $R$ be the canonical involution on $T^{\st}L$, then $R$ is an anti-holomorphic involution on $U_L$.
	\end{theorem}
	
	Using our existence theorem, we obtain:
	\begin{theorem}
		\label{thm_Calabiyaunearbycollapsing}
		Let $(L,g_L)$ be a real analytic Riemannian manifold with $\chi(L)=0$, let $(U_L,J,\omega,\Omega)$ be a Calabi-Yau structure constructed in Theorem \ref{thm_CYstructureoncotangentbundle}, suppose there exists a nondegenerate multivalued harmonic form on $(L,g_L)$ with associated branched covering $p:\tL\to L$, then there exists a family of immersed special Lagrangian submanifolds $\tiota_t:\tL\to U_L$ such that $\lim_{t\to 0}\|\tiota_t-\iota_0\circ p\|_{\ca}=0$, where $\iota_0$ is the inclusion of the zero section. In addition, $\iota_t$ will satisfy $R\circ\tiota_t=\tiota_t\circ \sigma.$
	\end{theorem}
	
	By \cite{he21z3symmetry}, the author construct nondegenerate $\ZT$ harmonic 1-forms over some rational homology 3-spheres. As a rational homology 3-sphere has vanishing first Betti number, it is rigid in McLean sense as in Theorem \ref{theorem_classicaldeformation}. Therefore, we obtain the following.
	\begin{corollary}
		There exist special Lagrangian submanifolds that are rigid in classical sense but could have branched deformations.
	\end{corollary}
	
	\subsection{Topological constraint of the existence of nondegenerate $\ZT$ harmonic 1-forms}
	Let $(L,g)$ be a Riemannian manifold and $(\Sigma,\MI,\al)$ be a nondegenerate $\ZT$ harmonic 1-form. By the work of Taubes \cite{Taubes20133manifoldcompactness,taubes2013compactness}, nondegenerate $\ZT$ harmonic 1-forms are the ideal boundary of the gauge theory compactification of $\mathrm{PSL}(2,\mathbb{C})$ flat connection and it is still very mystery. We would like to understand the following question.
	\begin{question}
		Is there any topological constraint to the existence of nondegenerate $\ZT$ harmonic 1-forms? 
	\end{question}
	
	Using the branched deformation theorem we proved, we could actually give a topological constraint for the existence of nondegenerate $\ZT$ harmonic 1-forms. Our starting point come from the work of Abouzaid and Imagi \cite{abouzaidImagi21}.
	\begin{theorem}{\cite{abouzaidImagi21}}
		\label{thm_nearbyspecialLagrangin}
		Let $U_L$ be a neighborhood of the zero section of $T^{\st}L$, and $(U_L,J,\omega,\Omega)$ be a Calabi-Yau structure such that the zero section $[L_0]$ a special Lagrangian submanifold, then for any unobstructed nearby immersed special Lagrangian $[L']$ which is $\MC^0$ close to $[L_0]$,
		\begin{itemize}
			\item [(i)] suppose $\pi_1(L)$ is finite, then $[L']=[L_0]$,
			\item [(ii)] suppose $\pi_1(L)$ has no nonabelian free subgroup, then $[L']$ is an unbranched deformation of $[L_0]$ that constructed in Corollary \ref{cor_unbranchedcoevering}.
		\end{itemize}
	\end{theorem}
	The definition of a Lagrangian submanifold is unobstructed is in \cite{fooo09}, which is not relevant to most of the parts in the paper, so we omit the definition. We will list a few cases where the unobstructed condition is satisfied. Let $(L,g)$ be a Riemannian manifold, suppose there exists a nondegenerate $\ZT$ harmonic 1-form on $L$ and suppose there exists a Calabi-Yau neighborhood $(U_L,J,\omega,\Omega)$. We write $[L_0]=(L,\iota_0)$ be the inclusion of the zero section. Then by Theorem \ref{thm_Calabiyaunearbycollapsing}, we obtain a family of special Lagrangian submanifolds $[\tL_t]$ which will be $\MC^{0}$ close to $[L_0]$. If any $[\tL_t]$ is unobstructed and suppose $\pi_1(L)$ is either finite or has no nonabelian free subgroup, this will lead to a contradiction based on Theorem \ref{thm_nearbyspecialLagrangin}. 
	
	\begin{proposition}
		\label{prop_obstructionforexistence}
		Let $(L,g_L)$ be a real analytic Riemannian manifold and $(\Sigma,\MI,\al)$ be a multivalued harmonic 1-form on $L$ with associate branched covering $p:\tL\to L$. Suppose 
		\begin{itemize}
			\item[(i)] $\pi_1(L)$ is finite or contains no nonablian free subgroup,
			\item[(ii)] $(L,g_L)$ has a Calabi-Yau neighborhood $(U_L,J,\omega,\Omega)$ in $T^{\st}L$,
			\item [(iii)] every special Lagrangian submanifold over $U_L$ which is diffeomorphic to $\tL$ is unobstructed, 
		\end{itemize}
		then $(\Sigma,\MI,\al)$ is not nondegenerate. In addition, suppose the canonical involution $R$ on $T^{\st}L$ induces an anti-holomorphic involution, then (iii) could be changed to 
		\begin{itemize}
			\item [(iii)$'$]every $R$-invariant special Lagrangian submanifold over $U_L$ which is diffeomorphic to $\tL$ is unobstructed.
		\end{itemize}
	\end{proposition}
	
	By \cite[Proposition 2.8]{abouzaidImagi21}, the condition that $L$ has no nonabelian free subgroup is a very strong assumption, $L$ only have the following four possibility: (i) $\pi_1(L)$ is finite, (ii)$L\cong S^1\times S^2$, (iii)$L\cong RP^3\sharp RP^3$, (iv) there exists a finite covering $L'\to L$ where $L'$ is the mapping torus of an orientation preserving diffeomorphism $T^2\to T^2$.

	In general, it is very hard to check whether a Lagrangian submanifold is unobstructed or not. By the work of FOOO \cite{fooo09}, there is a topological condition one could verify for the obstruction condition.
	\begin{proposition}{\cite{fooo09}}
		Let $(\tL,\tiota:\tL\to U_L)$ be a Lagrangian submanifold, suppose $\tiota^{\st}:H^2(U_L;\mathbb{R})\to H^2(\tL;\mathbb{R})$ is surjective, then $(\tL,\tiota)$ is unobstructed.
	\end{proposition}
	
	As a corollary, we obtain extra constraints of nondegenerate $\ZT$ harmonic 1-form.
	\begin{theorem}
		\label{thm_secondhomologyvanishes}
		Let $(L,g_L)$ be a real analytic Riemannian manifold and let $(\Sigma,\MI,\al)$ be a nondegenerate $\ZT$ harmonic 1-form on $L$ with branched covering $p:\tL\to L$. Suppose 
		\begin{itemize}
			\item [(i)] $\pi_1(L)$ is either finite or contains no nonabelian free subgroup,
			\item [(ii)] $(L,g_L)$ has a Calabi-Yau neighborhood on $T^{\st}L$,
		\end{itemize}
		then $b_2(\tL)>b_2(L)$, where $b_2$ is the second Betti number.
	\end{theorem}
	
	Unfortunately, the above theorem is trivial when $\dim(L)=3$. By Corollary \ref{cor_bettinumberfirst}, the existence of $\ZT$ harmonic 1-form will require that $b_1(\tL)>b_1(L)$ which implies $b_2(\tL)>b_2(L)$ by Poincare duality. However, this topological constraint is non-trivial in for four dimensional or higher and we will introduce an example. Let $\sigma_n:S^n\to S^n$ be the involution map fixing the equator and let $\tL:=S^1\ti S^{n-1}$. We take $\sigma:=(\sigma_1,\sigma_{n-1})$, then $\sigma$ is an involution action on $\tL$ with fixed point $\mathrm{Fix}(\sigma)$ is diffeomorphic to $S^0\ti S^{n-2}$ which is two copies of $S^{n-2}$. The quotient $L:=\tL/\lan \sigma\ran$ under the involution, would be $L:=S^{n}$. Moreover, the quotient map $p:\tL\to L$ will be a branched covering map. When $n\geq 4$ and $n$ is odd, as $H^1(\tL;\mathbb{R})^-\neq 0$, over $L$, there exists a $\ZT$ harmonic 1-form $\al$ over $\tL$. However, as $H^2(\tL;\mathbb{R})=0$ and $\chi(\tL)=0$, by Theorem \ref{thm_secondhomologyvanishes}, this $\ZT$ harmonic 1-form will not be nondegenerate.
	
	When $n$ is even, over $L=S^n$, as $\chi(L)\neq 0$, we don't know the existence of Calabi-Yau neighborhood for generic metric on $L$. However, over $T^{\st}S^n$ we have the Stenzel metric \cite{stenzel93}, which extend the round metric on $S^{n}$ to a complete Riemannian metric over the whole $T^{\st}S^n$. Thus over $T^{\st}L$, we could associate a Calabi-Yau structure. By Theorem \ref{thm_secondhomologyvanishes}, the $\ZT$ harmonic 1-form will not be nondegenerate. In summary, we obtain
	
	\begin{corollary}
		Let $p:S^1\ti S^{n-1}\to S^n$ be the branched covering map considered above, $\Sigma$ be the branch locus, $\chi\in H^1(S^1\ti S^{n-1};\mathbb{R})$, then
		\begin{itemize}
			\item [(i)] If $n$ is odd, then  for any Riemannian metric over $S^{n}$, the $\ZT$ harmonic 1-form corresponding to $\chi$ will not be nondegenerate. 
			\item [(ii)] If $n$ is even, then for the round Riemannian metric over $S^n$, the $\ZT$ harmonic 1-form corresponding to $\chi$ will not be nondegenerate. 
		\end{itemize}
	\end{corollary}	
	
	Here we provide another example. Let $M$ be a connected oriented closed manifold, by Neofytidis \cite{neofytidis14branched}, there exists a $\pi_1$ surjective branched double covering 
	\begin{equation}
		\label{eq_brancheddoublecovermapneofytidis}
		p: S^1\times \sharp_{k}(M\times S^1) \to \sharp_k(M\times S^2).
	\end{equation}
	We write $\tL=S^1\times \sharp_{k}(M\times S^1)$ and $L=\sharp_k(M\times S^2)$, suppose $M$ satisfies $b_1(M)=b_2(M)=0$, then $b_1(L)=0,\;b_2(L)=k$ and $b_1(\tL)=1+k,\;b_2(\tL)=k$. Then for any $\chi\in H^1(\tL;\mathbb{R})$, it will give a multivalued harmonic 1-form on $L$. However, as $b_2(\tL)=b_2(L)$, all these multivalued harmonic 1-forms could not be nondegenerate.
	\begin{corollary}
		Let $M$ be a closed manifold which is odd dimensional and $b_1(M)=b_2(M)=0$, then for the double branched covering in \eqref{eq_brancheddoublecovermapneofytidis} and for any Riemannian metric on $\sharp_k(M\times S^2)$, the multivalued harmonic 1-form corresponding to $\chi\in H^1(S^1\times \sharp_{k}(M\times S^1);\mathbb{R})$ for \eqref{eq_brancheddoublecovermapneofytidis} is not nondegenerate. 
	\end{corollary}
	
	\appendix
	
	\begin{section}{Polyhomogeneous Expansions and Mellin Transformations}
		\label{sec_backgroundpolyhomogeneoussolution}
		In this appendix, we will introduce the theory of polyhomogeneous expansions and elliptic edge operators. All results surveyed here could be found in \cite{Mazzeo1991,melrose93atiyahpatodi}.
		
		\subsection{Mellin transformation}
		Mellin transformation is a tool to study the expansion of functions along a singular set, which gives a nice correspondence between meromorphic functions and functions with expansions. We refer to \cite[Section 2]{Mazzeo1991} and \cite{melrose93atiyahpatodi} for more details.
		
		Let $u(r)$ be a compact supported function defined on the positive real axis $0<r<\infty$, the Mellin transformation $Mu:=u_M(\zeta)$ of $u$ is a function defined on the complex plane defined as
		\begin{equation*}
			Mu(\zeta)=u_M(\zeta):=\int_0^{\infty} u(r) r^{i\zeta-1}dr,
		\end{equation*}where $\zeta:=\xi+i\eta$ is a complex variable. 
		
		The Mellin transformation is related to the Fourier transformation. Let $r=e^s$, then $$u_M(\zeta)=\int_{-\infty}^{+\infty}u(e^s)e^{i\zeta s}ds,$$ which is the Fourier transform for the function $u(e^{s})$. Therefore, there is a Plancherel relationship extended the definition of Mellin transform to $L^2$ function, or more generally, weighted $L^2$ space $r^{\delta}L^2(\mathbb{R}^+)$ such that the Mellin transform induces an isomorphism
		\begin{equation}
			r^{\delta}L^2(\mathbb{R}^+)\cong L^2(\eta=\delta-\frac12),
		\end{equation}
		where $r^{\delta}L^2:=\{r^{\delta}f|f\in L^2\}$.
		
		As for $\delta'\leq\delta$, we have $r^{\delta}L^2\subset r^{\delta'}L^2$, the domain that Mellin transform exists will be $\{\Im\zeta<\be\}$. Analogy to the Fourier transform, the Mellin transform also have an inverse formula. 
		$$M^{-1}(u_M(\zeta))=\frac{1}{2\pi i}\int_{\Im\zeta=C}u_M(\zeta)x^{-i\zeta}d\xi.$$
		In particular, if $u\in r^{\delta}L^2$ and $C=\delta-\frac12$, we have $M^{-1}(u_M(\zeta))=u(x)$.
		
		\begin{comment}
		
		\begin{proposition}
		\label{prop_Mellintransform}
		The Mellin transform have the following basic properties:
		\begin{itemize}
		\item [(i)]Suppose $u$ have the asymptotic behavior $u=\MO(x^{-\al})$ when $x\to 0^+$ and $u=\MO(x^{-\be})$ when $x\to 0^-$ and $\al<\be$, then $u_M$ exists in the stripe $S_{\al,\be}$.
		\item [(ii)]  
		\item [(iii)]
		\end{itemize}
		\end{proposition}
		\end{comment}

		Now, we will focus on the expansion at $r=0$. Let $\phi(r)$ be a smooth function over $\RP$ such that $\phi$ equals $1$ near $r=0$ and $\phi$ vanishes for large $t$, we consider $u(r):=r^{s}(\log r)^p\phi(r)$. In addition, as $r^s(\log r)^p\phi(r)\in x^{\delta}L^2$, for any $\delta<\Re s+\frac12$, $u_M(\zeta)$ is well defined on $\{\Im\zeta<\be\}$.
		
		In addition, as $$\int_0^{1}r^{s}(\log r)^pr^{i\zeta-1}dr=(-1)^{p+1}\frac{p!}{(i\zeta+s)^{p+1}},$$ an integration by parts shows that $u_M$ can be extended to a meromorphic function in the whole complex plane, with poles of order $p+1$ at $\zeta=is$. 
		
		In our situation, let $\Sigma$ be an embedded submanifold on a Riemannian manifold, Let $U$ be a tubular neighborhood of $\Sigma$ and we choose local coordinates $(z=re^{i\theta},t)$ along $U$ such that $\Sigma=\{z=0\}$. Let $u(r,\theta,t)$ be a polyhomogeneous function near $\Sigma$ with index set $E$, then the above arguments could be applied to $u(r,\theta,t)$. The Mellin transform of $r$ coordinate and $(\phi(r) u)_M(\zeta,\theta,t)$ could be understood as a function on $(\zeta,\theta,t)\in\mathbb{C}\ti S^1\ti \Sigma$. In addition, for $s\in E$, $(\phi(r) u)_M$ has a meromorphic extension at $\zeta=is$.
		
		We have the following fundamental correspondence of Mellin transformation:
		\begin{theorem}{\cite[Proposition 5.27]{melrose93atiyahpatodi}}
			\label{thm_Mellintransformcorrespondence}
			The Mellin transformation gives an isomorphism between the polyhomogeneous function with index set $E$ and meromorphic functions with values in $\MC^{\infty}(Y)$ having poles at $\ze=is$ for $s\in E$.
		\end{theorem}
		
		\subsection{The expansion of $\ZT$-harmonic function}
		Let $\Sigma$ be a codimension 2 submanifold of $L$ and $\chi:\pi_1(L \setminus \Sigma)\to \Gamma$ be the representation satisfies (ii) of Definition \ref{def_multivaluedeifnition}, with associative affine bundle $V^-$ and corresponding vertical line bundle $\MI$. Using the flat structure, we will study the operator $\Delta:\Gamma(V^-)\to \Gamma(\MI)$ near the branch locus.
		
		We could consider the following local model for a neighborhood near $\Sigma$. Let $\mathbb{R}^n=\mathbb{C}\ti \mathbb{R}^{n-2}$ and we will represents a point in $\mathbb{R}^n$ by $(z,t)$ and write $z=re^{i\theta}$. Let $U$ be a open set in a neighborhood of $\Sigma$ and we also regard $U$ as an open set of $\mathbb{R}^n$ such that $U\cap\Sigma\subset \{0\}\ti\mathbb{R}^{n-2}$. For convenient, we might think $V^-$ is defined over $\mathbb{C}^{\st}\ti \mathbb{R}^{n-2}$ with fiber $\mathbb{R}$ and holonomy $-1$, while the sections of $V^-$ could be identified with 2-valued function $f$ such that $f(w^2,t)$ is an odd function on $w$. 
		
		We might write
		$$r^2\Delta=(r\pa_r)^2+\pa_{\theta}^2+(r\pa_t)^2+rE,$$
		where $E$ is smooth in $r\pa_r,\pa_{\theta}$ or $r\pa_t$. Here $r^2\Delta$ is called an elliptic edge operator, which considered in \cite{Mazzeo1991}. The indicial family $I_{\zeta}:=\pa_{\theta}^2+\zeta^2$ of the operator $r^2\Delta$ is given by
		$$
		r^2\Delta(r^{\zeta}(\log r)^pf(r,\theta,t))=r^\zeta(\log r)^pI_{\zeta} (f(0,\theta,t))+\MO(r^{\zeta}(\log r)^{p-1}),
		$$
		where the above formula works for $\forall f\in\MC^{\infty}$, $s\in\mathbb{C}$ and $p\in\mathbb{N}_0$. In local coordinates, we might write $I:=(r\pa_r)^2+\pa_{\theta}^2$, then under the Mellin transformation, we have $(If)_{M}=I_{i\zeta}f_M$. We define $$L^2_{\odd}(S^1\ti\Sigma)=\{f\in\MI|_{S^1\ti\Sigma}|f(\theta+2\pi,t)=-f(\theta,t).\;f\in L^2\},$$ which is the $V^-$ valued $L^2$ function, then we might regard $$I_{\zeta}:L^2_{\odd}(S^1\ti\Sigma)\to L^2_{\odd}(S^1\ti\Sigma).$$
		
		\begin{definition}
			The indicial roots of $r^2\Delta$ is the set of values $\zeta\in\mathbb{C}$ for which $I_{\zeta}$ fails to be invertible on $L^2_{\odd}(S^1\ti\Sigma)$ and we write $E_{\ind}$ be the sets of all indicial roots. 
		\end{definition}
		
		\begin{proposition}
			$E_{\ind}=\{k+\frac12|k\in\mathbb{Z}\}$ is the collection of all indicial roots for the operator $r^2\Delta$, 
		\end{proposition}
		\begin{proof}
			Using the Fourier transform, we might write $f\in$ $L^2_{\odd}(S^1\ti\Sigma)$ as $$f=\sum_k \Re(f_{k}(r,t)e^{i(k+\frac12)\theta}),$$
			where $f_k(r,t)$ is a complex function. 
			Then we compute $$I_{\zeta}f=\sum_k(\zeta^2-(k+\frac12)^2)\Re(f_k(r,t)e^{i(k+\frac12)\theta}).$$
			Therefore, the indicial roots will be $\zeta=k+\frac12$ for $k\in\mathbb{Z}$.
		\end{proof}
		
		Let $\rho\in L^2$ be a section of $\MI$ such that $\rho$ is polyhomogeneous with index set $\{k+\frac12|k\in \mathbb{N}\}$. We now consider the polyhomogeneous expansions of solutions, which will be a special case of \cite[Section 7]{Mazzeo1991}.
		
		\begin{theorem}{\cite[Section 7]{Mazzeo1991}}
			\label{thm_polyhomoexpansion}
			Let $f\in\Gamma(V^-)$ be a solution to $\Delta f=\rho$, then $f$ has a polyhomogeneous expansion with index set $E'=\{k+\frac12|k\in\mathbb{N}\}$.
		\end{theorem}
		\begin{proof}
			We only need to show the Mellin transform $f_M(\zeta,\theta,t)$ has a meromorphic extension to the whole $\zeta$ plane and the poles of $f_M$ lies only on $i(k+\frac12)$. It is easier we write $r^2\Delta f=\rho'$ with $\rho':=r^2\rho$.
			
			We might write $r^2\Delta f=I(f)+E'(f)$, then for the equation $I(f)+E'(f)=\rho'$, using the Mellin transform, we obtain $$I_{\zeta}f_M+(E'f)_M=\rho_M.$$ 
			By Theorem \ref{thm_Mellintransformcorrespondence}, $\rho'_M(\zeta,\theta,t)$ has a meromorphic extension to the whole $\zeta$ plane $\mathbb{C}$.
			By \cite[Section 3]{donaldson2019deformations}, $f$ is conormal, and we define $f_{\theta,t}(r):=f(r,\theta,t)$, then $f_{\theta,t}(r)\in r^{\delta}L^2(\mathbb{R}^+)$ for $\delta<\frac{3}{2}$ and $f_{\theta,t}$ smoothly depends on $\theta,t$. We also write 
			$f_{M,\theta,t}(\zeta)=f_M(\zeta,\theta,t)$, as $f_{\theta,t}(r)\in r^{\delta}L^2(\mathbb{R}^+)$, $f_{M,\theta,t}$ is holomorphic on $\{\Im\zeta<1\}$. We will prove by induction that $f_{M,\theta,t}(\zeta)$ has the meromorphic extension. Suppose $f_{M,\theta,t}(\zeta)$ has a meromorphic extension to $\{\Im\zeta<N\}$. Then we consider the term $(E'f)_M$. As $E'$ consists of differentials with higher order in $r$, such as $r\pa_\theta$, $r^2\pa_r$, etc. For example, we might have $(r\pa_{\theta}f)_{M,\theta,t}(\zeta)=(\pa_{\theta}f)_{M,\theta,t}(\zeta-i)$. Therefore, $E'f$ is holomorphic on $\{\Im\zeta<N+1\}$. In addition, as $(\rho_{M,\theta,t})(\zeta)$ is meromorphic on $\{\Im\zeta<N+1\}$, $(I_{\zeta}f_M)(\cdot,\theta,t)$ is meromorphic on $\{\Im\zeta<N+1\}$. In the region $\{N\leq \Im\zeta<N+1\}$, $I_{\zeta}$ is invertible except at $\Im\zeta=\frac{1}{2}+N$, which is the indicial root set. Thus $f_M$ has a meromorphic extension to $\{\Im\zeta<N+1\}$ will extra poles induces at $\Im\zeta=\frac{1}{2}+N$. We keep on this induction, $f_M$ will has a meromorphic extension to the whole $\zeta$ plane with extra poles at the indicial roots set $\Im\zeta=\frac12+k$.
		\end{proof}
		
		Under the previous assumption, for the solution $\Delta f=\rho$, we might write $$f=\sum_{k\geq 0}\Re (f_k(r,t)e^{i(k+\frac{1}{2})\theta}),$$ then from Theorem \ref{thm_Mellintransformcorrespondence} and Theorem \ref{thm_polyhomoexpansion}, in a suitable coordinate, $f$, thus $f_k$, has a polyhomogeneous expansion 
		\begin{equation}
			\label{eq_expressionofexpansion}f_k(r,t)\sim \sum_{l\geq 0}\sum_{0\leq p\leq p_l}r^{l+\frac12}(\log r)^pf_{k,l,p}(t),
		\end{equation}
		where $l,\;p_l\in\mathbb{N}$. 
		
		The regularity of the derivatives of $f$ in $r$ direction depends on the appearance of the $\log r$ terms in the expansion. However, from the general theory, we could not avoid the appearance of $\log r$ terms in the expansions, except special cases. One could step-by-step check the appearance of the $\log r$ terms by comparing different terms in both side of the equation. If we have extra regularity of $f$, for example, $f\in\MC^{2,\al}$, then for \eqref{eq_expressionofexpansion}, we could conclude that $p_0=p_1=0$. 
		
		Using the language of polyhomogeneous, Theorem \ref{thm_donaldsonremainningestimate} could be rephrased to be the following:
		\begin{proposition}
			\label{prop_appendixpolyh}
			Suppose $f$ is a $\ZT$ valued harmonic function, $\Delta f=\rho$, suppose $\rho\in L^2$ and $\rho$ has a polyhomogeneous expansion with index set $E$, then $f$ has a polyhomogeneous expansion and the expansion of $f$ could be written as
			$$
			f=\Re(Az^\frac12+Bz^{\frac32})+E,
			$$
			where $E$ has expansions $$E\sim \sum_{k\geq 2}\sum_{0\leq p\leq p_k} E_{k,p}(\theta,t)r^{k+\frac12}(\log r)^p.$$
		\end{proposition}
		\begin{remark}
			Suppose the coordinate $\theta$ is harmonic $\Delta \theta=0$, then for the equation $\Delta f=0$, the solution $f$ will not have $\log r$ terms in its' expansion. However, for the equation $\Delta f=\rho$, even suppose in the expansion of $\rho$, it doesn't contain $\log r$ terms, $f$ could still have $\log r$ terms in its' expansion.
		\end{remark}
		
	\end{section}
	
	\begin{comment}
	
	\begin{section}{Tubular Neighborhood Theorem}
	\label{app_tubularneighb}
	In this section, we will proof a tubular neighborhood theorem for immersed submanifold.
	\end{section}
	\end{comment}
	
	\begin{comment}
	
	\begin{proposition}
	Suppose $f$ is a $\ZT$ valued harmonic function, $\Delta f=0$, then the expansion of $f$ could be written as 
	$$f\sim \sum_{k\geq 0}r^{k+\frac12}\Re(B_k(t)e^{i(k+\frac12)\theta}).$$
	In particular, there is no log terms in the expansion of $f$.
	\end{proposition}
	\begin{proof}
	We will show $f_{k,l,p}(t)=0$ if $l\neq k+\frac12$ or $p\neq 0$. We might write $r^2\Delta f=If+Ef$, where $I f=(r\pa_r)^2f+\pa_{\theta}^2f$. As $e^{i(k+\frac12)\theta}$ form a local base of $L^2_{\odd}(S^1\ti\Sigma)\otimes \mathbb{C}$, we could write 
	\begin{equation}
	\label{eq_expr2Deltaf}
	r^2\Delta f=\sum_k\Re(((r\pa_r)^2-(k+\frac12)^2)f_k)e^{i(k+\frac12)\theta}+Ef=0.
	\end{equation}
	
	First, all the $log$ terms in the expression will vanishes. Suppose $l$ is the smallest number such that $p_{l}\neq 0$. The the order $r^l\log r^{p_l}$ terms in \eqref{eq_expr2Deltaf} is $$(l^2-(k+\frac12)^2)f_{k,l,p_l}r^l\log r^{p_l},$$ which requires that $l=k+\frac{1}{2}$. In addition, the order $r^l\log r^{p_l-1}$ term will be 
	$$
	((l^2-(k+\frac12)^2)f_{k,l,p_l-1}+2lp_l f_{k,l,p_l})r^l\log r^{p_l-1}
	$$
	\end{proof}
	In general, 
	
	Suppose $\rho$ has an expansion $\rho\sim r^{\frac12}\rho_0+\sum_{}$
	
	\end{comment}
	
	\bibliographystyle{plain}
	\bibliography{references}
\end{document}